\documentclass[12pt]{article}
\usepackage{fancyhdr}
\usepackage{amssymb}
\usepackage{amsmath}
\usepackage{amsthm}
\usepackage{graphicx}
\usepackage{psfrag}
\usepackage{graphics}
\usepackage{amsthm}
\usepackage{hyperref,enumitem}
\usepackage[dvipsnames]{xcolor}
\usepackage[english]{babel}

\newcommand{\F}{F}

\newcommand{\cI}{\mathcal{I}}
\newcommand{\cJ}{\mathcal{J}}

\newcommand{\esp}{\hspace{0.06cm}}
\newcommand{\N}{\mathbb N}

\newcommand{\T}{\mathbb T}

\newcommand{\clo}{\mathrm{S}^1}

\newcommand{\dist}{V_{\infty}}

\newcommand{\var}{\mathrm{var}}
\newcommand{\rot}{\mathbf{rot}}

\newcommand{\SO}{\mathrm{SO}}

\newcommand{\R}{\mathbb{R}}

\newcommand{\Z}{\mathbb Z}

\newcommand{\Diff}{\mathrm{Diff}}

\newcommand{\eps}{\varepsilon}

\newcommand{\X}{X}
\newcommand{\Y}{Y}
\newcommand{\psiX}{\psi_0}
\newcommand{\psiY}{\psi_1}
\newcommand{\de}{\delta}

\def\Fix{{\mathrm {Fix}}}

\def\res#1{\mathbin{|}{}_{#1}}
\def\rep#1{\{#1\}}

\theoremstyle{theorem}

\newtheorem{thm}{Theorem}[section]

\newtheorem{main}{Main Theorem}

\newtheorem*{analytic}{Theorem C}

\newtheorem{prop}[thm]{Proposition}
\newtheorem{defprop}[thm]{Definition-Proposition}
\newtheorem*{propsn}{Szekeres' theorem}
\newtheorem*{kop}{Kopell's lemma}

\newtheorem{cor}[thm]{Corollary}
\newtheorem{qs}{Question}

\newtheorem{lem}[thm]{Lemma}
\newtheorem{interp}[thm]{The Interpolation Lemma}
\newtheorem{reg}[thm]{The Regularization Lemma}

\theoremstyle{definition}

\theoremstyle{remark}
\newtheorem{rem}[thm]{Remark}

\newcommand{\id}{\mathrm{id}}
\newcommand{\f}{\varphi}

\begin{document}

\date{}
\vspace{-1cm}

\date{}
\author{H\'el\`ene Eynard-Bontemps \,\, \& \,\, Andr\'es Navas\\
with an appendix in collaboration with Th\'eo Virot}

\title{\vspace{-1.5cm}The space of $C^{1+ac}$ actions of $\Z^d$ on a one-dimensional manifold is path-connected}

\maketitle

\noindent{\bf Abstract.} We show path-connectedness for the space of $\Z^d$ actions by $C^1$ diffeomorphisms 
with absolutely continuous derivative on both the closed interval and the circle. We also give a new and short proof of 
the connectedness of the space of $\Z^d$ actions by $C^2$ diffeomorphisms on the interval, as well as an analogous 
result in the real-analytic setting. 

\vspace{0.1cm}


\vspace{0.1cm}

\noindent{\bf MCS 2020:} 37C05, 37C10, 37C15, 37E05, 37E10, 57S25.\medskip

\section*{Introduction}

Actions of free abelian groups on one-dimensional manifolds have been deeply studied from the 60's on, 
with seminal works such as \cite{Sz,Ko,Ta,Se,Yo}, and have known a regain of interest in the past decade 
\cite{Na14,BE}.\footnote{At this point, we should also refer to the closely-related work \cite{EN2}. 
However, this article will remain unpublished since most of the results therein are improved 
(and, actually, simplified) by those of this article. It is worth pointing out this article is 
totally self-contained; in particular, no result from \cite{EN2} is a prerequisite.}
A historical motivation is the study of foliations of $3$-manifolds by surfaces, where actions of $\Z^2=\pi_1(\T^2)$ on the interval 
appear as holonomy representations of (germs of) foliations near toric leaves, which play a special role in the $3$-dimensional context. 
Actually, there is a deep relation between $\Z^2$ actions and the problem of deformation of foliations by surfaces (seen as integrable plane fields) on any 
3-dimensional manifold. Namely, an obvious necessary condition to deform one such foliation to another one is the existence of a homotopy between the 
corresponding plane fields, and it is natural to ask whether this condition is also sufficient (the plane fields in the homotopy being non necessarily tangent 
to foliations). In \cite{Ey16}, this is shown to be true in class $C^{\infty}$ provided the space of $C^{\infty}$ actions of $\Z^d$ on $[0,1]$ is path connected, 
which shows the relevance of this (still unsolved) question.

Actions on the circle are also involved in related problems in the context of codimension-one foliations, 
with an additional source of interest coming from small denominator phenomena related to linearization of 
circle diffeomorphisms. In this direction, a fundamental (and still unsolved) question raised by Rosenberg in 
the 70's asks whether the space of $\Z^2$ actions on the circle by $C^{\infty}$ diffeomorphisms is locally 
path-connected.

Note that a $\mathbb{Z}^d$ action is completely determined by the data of $d$ commuting diffeomorphisms
 of the one-dimensional manifold, so the space of such actions inherits a natural topology from that of the 
group of diffeomorphisms under scrutiny. Although we might fail to specify it on some occasions, all the actions
 considered in this article are by orientation-preserving diffeomorphisms. Also, in a slight abuse of terminology, 
 actions by $C^k$ diffeomorphisms will be called {\em $C^k$ actions}.\medskip

The question of connectedness can be asked in any differentiability class, and it turns out that the phenomena 
at play highly depend on the regularity. 
In a nutshell, this subject is the theater of a thrilling fight between the virtues of flexibility and rigidity. 

In low regularity, commutativity is not so restrictive, and there is no simple universal model of what such an action 
should look like. 
This may seem inconvenient but, on the upside, low regular actions are more ``flexible''.  
An enlightening illustration of this view is the method employed in \cite{Na14} to prove the path-connectedness 
for $\Z^d$ actions on the interval or the circle in $C^1$ regularity. Namely, therein, it is proved that 
the only obstruction to deform such an action to one by isometries through a path of $C^1$ conjugates 
is the existence of hyperbolic periodic points, and that these can be somewhat ``flattened'' 
along a ($C^1$-continuous) path of topologically conjugated $C^1$ actions. 

In higher regularity, things are not as flexible. 
For example, \cite{Na23} highlights a new obstruction,  the so-called \emph{asymptotic variation} 
(cf. Section \ref{ss:rat}), that arises when one tries to deform an action to one by isometries through conjugates
 in regularity higher than $C^1$. A priori, this obstruction cannot be as ``easily'' cancelled through deformation 
 as the hyperbolicity of periodic points. Nevertheless, the rigidity of higher regularity limits the possible forms 
 of the actions we start with, and this provides a natural strategy to deform them (cf. Section~\ref{ss:idea}).

To be more concrete, we present a summary of the known ``definitive'' results on the subject in chronological order. 
Below, $M$ denotes either the closed interval or the circle:

\begin{itemize}

\item The space of $C^0$ actions of $\Z^d$ on $M$ is path-connected. 
For $M=[0,1]$, this can be shown using the so-called Alexander trick (cf. \cite{BE}), 
which actually implies that the space of $C^0$ actions of \emph{any finitely generated group} on $[0,1]$ 
is \emph{contractible}! 
For $M=\T^1$, it follows from ideas of Herman, made explicit in Section \ref{section-herman}, that the space 
of $C^0$ actions of $\Z^d$ on $\T^1$ retracts onto $(\SO(2,\R))^d$; in particular, it is path-connected \cite{EN2}. 

\item The space of $C^1$ actions of $\Z^d$ on $M$ is path-connected \cite{Na14}. 
However, it is unknown whether this space is contractible for the case of the interval, 
or if it is homotopically equivalent to $(\SO(2,\R))^d$ in the circle case. 

\item The space of $C^\infty$ actions of $\Z^d$ on $[0,1]$ is connected \cite{BE}, but, as mentioned before, the question of the 
path-connectedness remains open (the connectedness is obtained by showing that the path-connected component 
of the trivial action is dense in the whole space, but it is unknown whether it \emph{is} the whole space). 
For the case of the circle, even the connectedness is unknown, though \cite{FK,Be,BE} give partial results.

\item The space of $C^2$ actions of $\Z^d$ on $[0,1]$ is connected \cite{EN2}, and, as in the $C^{\infty}$ case, 
the path-connectedness for $[0,1]$ and the connectedness for the circle remain open.

\end{itemize}

In \cite{EN2}, we also showed an important result for this series, namely, that any two $C^2$ actions 
of $\Z^d$ on $M$ can be connected by a path of $C^{1+ac}$ actions. 
Recall that a function $u$ on $M$ is \emph{absolutely continuous} ($ac$, for short) if it is the ``primitive'' 
of an $L^1$ function, which we will denote by $Du$. 
A $C^{1+ac}$ diffeomorphism of $M$ is a $C^1$ diffeomorphism whose derivative 
(or equivalently, the logarithm of it) is absolutely continuous. 
In Section \ref{notations}, we recall what is the natural topology on the space $\Diff^{1+ac}_+(M)$, 
and we discuss some of its basic features. \medskip

In this work, we give a complete result of path-connectedness for $C^{1+ac}$ actions.

\begin{main}
\label{t:C1ac}
The space of orientation-preserving $C^{1+ac}$ actions of $\Z^d$ on either the closed interval or the circle is path-connected.
\end{main}

It is worth to stress that our strategy allows to connect any given action to the trivial one by an \emph{explicit} path 
of actions, which is \emph{canonical} in the sense that it is determined by the initial action 
and involves no auxiliary choice.

Though this new result may seem like a small step between $C^1$ and $C^\infty$ regularities, it requires new 
techniques compared to the $C^1$ case, which in turn appear unsuitable for higher regularity. 
More importantly, it gives the  strongest possible result, namely path-connectedness, 
which again remains unknown in higher regularity (apart from real-analytic, cf. Theorem C). 
We will see in due time what is missing in order to get a general path-connectedness result for $C^2$ actions on 
the interval (cf. Remark \ref{r:pcC2}).  Notwithstanding, the new strategy allows us to give a simpler 
and more elegant proof of the result below from \cite{EN2}.

\begin{main}
\label{t:C2}
The space of orientation-preserving $C^2$ actions of $\Z^d$ on the interval is connected.
\end{main}

Our last result deals with the real-analytic setting. Although this lacks the motivation coming from foliation theory and 
the methods of proof are based on well-established work in this framework, putting them together requires the use 
of several ingredients of the proofs of Theorems A and B. This leads to the following crystal clear result. 

\begin{analytic}\label{t:anal}
The space of orientation-preserving real-analytic actions of $\Z^d$ on the interval is path-connected.
\end{analytic}

The strategy of proof  of the Main Theorem \ref{t:C1ac} shares some features with the one used in \cite{EN2}. 
However, there is an important difference. 
Namely, in \cite{EN2}, the starting idea was the same as the one implemented in \cite{Na14} for $C^1$ actions, 
which consisted in looking for conjugates of the original action that become closer and closer to an action 
by isometries in the corresponding topology. To this end, in the case of the interval (and of 
non minimal actions on the circle), we needed to deal with hyperbolic fixed points, and this is where we used 
the $C^2$ assumption. Indeed, we wished to apply Sternberg's linearization theorem to them, 
but this classical result of the $C^2$ setting no longer holds in regularity $C^{1+ac}$ 
(in this regard, see \cite{EN4} as well as Appendix~\ref{a:examples} of the present article).

Here, in contrast, we use conjugacy to connect not necessarily the trivial action, but possibly an action 
that embeds into a $C^{1+ac}$ flow (which then can be deformed to the trivial action, but not through 
conjugated ones). 
Concretely, the new ingredients of this strategy are the Interpolation Lemma \ref{l:chem} and the Regularization 
Lemma \ref{l:lissage}, as well as a variation of a classical result of Szekeres \cite{Sz} which is established 
in Appendix~\ref{a:szekeres} together with an extension of a fundamental inequality from the companion 
Appendix \ref{a:varflow}. 

We devote Section \ref{s:conjugate} to the Interpolation Lemma. The Main Theorem \ref{t:C1ac} is then 
proved in Sections \ref{s:C1ac-int} and~\ref{s:C1ac-circle} for the interval and the circle case, respectively. 
Section \ref{s:C2} is devoted to the proof of the Main Theorem \ref{t:C2}, and discusses to what extent 
the involved techniques apply to the case of the circle and to higher regularity. Finally, 
Theorem C is proved in Section \ref{s:analytic}.

\medskip

It is worth to stress that connectedness for $C^2$ actions on the circle remains as an open question. With our new 
strategy, we will show that it holds for actions which either have a rotation number set of rank less than or equal to $1$ 
or which embed into a $C^1$ flow (cf. Section~\ref{ss:C2-circle}). To get the connectedness of the whole space 
of $C^{2}$ actions of $\Z^d$ on the circle, one would need a positive answer to the following question:

\begin{qs} 
\label{q:cercleC2}
Let $(f_1,\dots, f_d)$, with $d\ge2$, be a $d$-tuple of commuting $C^2$ circle diffeomorphisms 
with rationally independent rotation numbers which does not embed into a $C^1$ flow.  
Does $(f_1,\dots,f_d)$ belong to the connected component of $(\id,\dots,\id)$ in 
the space of $d$-tuples of commuting $C^2$ circle diffeomorphisms?
\end{qs}

Examples of (real-analytic) diffeomorphisms with irrational rotation number which do not embed into a $C^1$ flow 
were first given by Arnold in \cite{Ar}. 
According to \cite{FK}, for any prescribed $d$-tuple of non simultaneously diophantine irrational numbers, 
one can obtain a $d$-tuple of commuting smooth diffeomorphisms having these numbers as rotation numbers, 
which does not embed into a $C^1$ flow.

Since $\Diff^2_+(\T^1)$ is path-connected, the conjugacy class of an action is path-connected and its closure 
is connected. Because of this, the answer to the above question will be positive if that to one of the two questions
 below is. Each of them deals with $d$-tuples of diffeomorphisms satisfying the same hypothesis as those in 
 Question \ref{q:cercleC2}. 

\begin{qs}
Can such a $d$-tuple be $C^2$-conjugated to $d$-tuples of commuting diffeomorphisms that are closer and closer 
to $d$-tuples of rotations in the $C^2$ topology?
\end{qs} 

The answer to this question is known to be positive in the $C^{1+ac}$ topology \cite{Na23}, 
and the analogous statement in the $C^{\infty}$ setting has been announced by Avila and Krikorian. 
In the $C^2$ context, this is open even for a single diffeomorphism. 

\begin{qs}
Can such a $d$-tuple be $C^2$-approached by simultaneous $C^2$ conjugates of rotations?
\end{qs}

Again, we do not even know the answer for a single diffeomorphism in this differentiability class.  
In the $C^{\infty}$ setting, Yoccoz proved in \cite[Chapitre III]{Yo} that the answer is positive for 
single diffeomorphisms. A positive answer in $C^{\infty}$ regularity for the general question 
 (for $d$-tuples of commuting diffeomorphisms) would imply the connectedness for smooth $\Z^d$ actions 
 on the circle, according to \cite{BE}. 
\medskip

We close this Introduction by pointing out that connectedness for the space of actions is a phenomenon of only certain classes of 
groups, mainly small ones (mostly Abelian or nilpotent) or free groups. This is very well known for the case of the circle, where several connected 
components arise even for continuous actions of, for instance, fundamental groups of surfaces \cite{Ma15} and Thompson's group $G$ \cite{ghys}. 
For the case of the interval, Alexander's trick shows that the space of continuous actions of every group is path-connected (see Example \ref{r:alex}). 
However, there are groups for which the space of $C^k$ actions is not connected for any $k \geq 1$. A remarkable example is Thompson's group $F$ 
(see \cite{GS87} for $C^{\infty}$ actions on the interval of this group), for which nonconnectedness  follows from the next two facts:
\begin{itemize}
\item Every action of $F$ either is faithful or factors through an action of the abelianization $F / F' \sim \mathbb{Z}^2$ (this is classical; se for instance 
\cite{CFP96}).
\item Every faithful action of $F$ by $C^1$ diffeomorphisms of the interval with no global fixed point in the interior is topologically 
semiconjugate to the standard action (this is a recent and remarkable result from \cite{BMRT21}). 
\end{itemize}
\noindent We claim that these imply that the subspace of faithful actions is separated from that of unfaithful ones. To see this, 
choose two elements $f,g$ in $F'$ (that is, with trivial derivative at the endpoints; see \cite{CFP96}) for which, in the canonical action, 
there are points $u,v$ satisfying the ``resilience  property'' $u < f(u) < f(v) < g(u) < g(v) < v$. Then for every unfaithful action of $F$, 
both $f$ and $g$ act trivially. However, for every faithful $C^1$ action of $F$, there must be a point at which one of $f,g$ has 
derivative smaller than $1/2$. Indeed, such an action has to be semiconjugate to the standard one on some fixed interval having no 
global fixed point inside. Since resilience is stable under semiconjugacy, this forces the announced property.

%%%%%%%%%%%%%%%%%%%%%%%%%%%%%%%%%%%%%%%%%%%%%%%%%%%%%%%%%%%%%%%%%%%%%%%%%%%%
\section{Notation and terminology}
\label{notations}

Throughout this text, the compact one-dimensional manifold $M$ will be always normalized so that its total length is 1. 
In particular, the circle will be parametrized by $[0,1]$ via the identification $0 \sim 1$.

A function $u$ on $M$ is said to be $C^{bv}$ if it has finite total variation, that is, if the following value is finite:
$$\var (u) := \sup_{0=a_0 < a_1 < \ldots < a_n=1} \sum_{i=0}^{n-1} \big| u (a_{i+1}) - u (a_i) \big|.$$ 
The space of such functions is a Banach space for the norm 
$$\| u \|_{bv} := \| u \|_{\infty}+ \var (u).$$
The function $u$ is said to be $C^{ac}$ if it is the primitive of an $L^1$ function, which is denoted $Du$ 
as if it were a ``genuine'' derivative ($Du$ coincides a.e. with the classical derivative). 
For such a function we have  $\var (u) = \| Du \|_{L^1}$. 
On many occasions, we will use the following fundamental fact: if $u$ is $C^{bv}$ and $\varphi$ is a 
homeomorphism of $M$, then $\var(u)=\var(u\circ \varphi)$. Note also that if $u$ is $C^{ac}$ and 
$\varphi$ is a $C^1$ diffeomorphism, then $u \circ \varphi$ is $C^{ac}$, and 
$\| Du \|_{L^1} = \| D \, (u \circ \varphi) \|_{L^1}.$

A $C^1$ diffeomorphism is said to be $C^{1+bv}$ (resp. $C^{1+ac}$) if $Df$ is $C^{bv}$ (resp. $C^{ac}$). 
One readily checks that this is equivalent to that $\log Df$ has the corresponding property (it is of class 
$C^{bv}$ or $C^{ac}$, respectively). 

The following formulas define distances on $\Diff^{1}_+(M)$ and $\Diff_+^{1+bv}(M)$, respectively (the proof of this 
is a straightforward exercise):
$$d_{1}(f,g) := \| f-g \|_{\infty} + \|\log Df-\log Dg\|_\infty,$$ 
$$d_{1+bv}(f,g):=  \| f-g \|_{\infty}  + \var(\log Df-\log Dg).$$
Moreover, these satisfy 
$$d_1\le  d_{1+bv}.$$ 
Indeed, this easily follows from the fact that, for any two $f,g$ in $\Diff_+^{1+bv}(M)$, the function $\log Df - \log Dg$
has to vanish somewhere, otherwise one could not have $\int_0^1Df=\int_0^1Dg$.

We endow $\Diff_+^{1+ac}(M)\subset \Diff_+^{1+bv}(M)$ with the induced distance $d_{1+ac}$, which satisfies:
$$d_{1+ac}(f,g) =  \| f-g \|_{\infty}  +  \| D \, (\log Df-\log Dg) \|_{L^1}.$$
Moreover, we endow $\Diff_+^{2}(M)$ with the distance $d_2$ defined by:
$$d_{2}(f,g) :=  \| f-g \|_{\infty}  + \|D(\log Df-\log Dg)\|_{\infty}.$$

%\textcolor{red}{
\begin{rem} It is worth mentioning that, endowed with the topologies induced from the corresponding metrics above, the groups 
$\Diff_+^1 (M), \Diff^{1+ac}_+ (M)$ and $\Diff_+^{2}(M)$ are topological groups, though $\Diff_+^{1+bv} (M)$ is not. 
See Appendix 4 for more on this.
\end{rem}
%}

Given a finitely generated group $G$ with a fixed set of generators $g_1,\dots,g_d$,  for every 
pair of actions $\rho,\,\tilde \rho:G\to\Diff^r_+(M)$, where $r\in\{1,1+bv,1+ac,2\}$, we simply let 
$$d_{r}(\rho,\tilde\rho) = \max_i d_r(\rho(g_i),\tilde\rho(g_i)).$$
 When $G=\Z^d$, we will always consider the canonical set of generators. 
  
We fix the terminology for dealing with an action $\rho \!: (\R,+) \to \Diff^r_+(M)$. In this setting, 
one usually denotes $f := \rho (1)$ and $f^t  := \rho (t)$, so that the following relation holds: 
$f^{s+t}=f^s\circ f^t$. We say that $(f^t)$ is a $C^r$ flow if, besides, the map $t \mapsto f^t$ from 
$(\R,+)$ to $\Diff^r_+(M)$ is continuous. This definition makes sense not only for $r \in \mathbb{N}$, but also for 
$r=1+bv$ and $r=1+ac$. 
For $r \in \{1,1+bv,1+ac,2\}$ and two flows $(f^t)$ and $(\tilde{f}^t)$, we let 
$$d_{r}( (f^t), ( \tilde{f}^t )) = \max_{t\in[-1,1]}d_{r} (f^t,\tilde{f}^t).$$

By an important theorem of Bochner and Montgomery \cite{BM}, for all $r \in \mathbb{N}$ and every $C^r$ flow $(f^t)$, 
the map $(t,x) \mapsto f^t (x)$ is of class $C^r$. In particular,  we can define its generating vector field $X$ by 
$$X(x) := \frac{d}{dt}_{|t=0}f^t(x).$$ 
Note that this definition makes sense even for $C^{1+bv}$ or $C^{1+ac}$ flows, as these are automatically 
$C^1$ flows. 

In abuse of notation, we will still use $d_r$ to compare the restrictions of two actions or flows to subsets of $M$ that 
are simultaneously fixed by them. Moreover, for the case of the interval, we will replace the metric $d_r$ by its version 
$d_r^*$ that does not incorporate the term $\| f-g \|_{\infty}$ in the definition. More precisely, 
$$d_1^* (f,g) := \|\log Df-\log Dg\|_\infty,$$ 
$$d^*_{1+bv}(f,g) := \var(\log Df - \log Dg),$$
$$d^*_{1+ac}(f,g) := \|D(\log Df-\log Dg)\|_{L^1},$$
$$d_{2}^* (f,g) := \|D(\log Df-\log Dg)\|_{\infty}.$$
Note that $d_r^* \leq d_r \leq 2 \, d_r^*$, so that $d_r$ and $d_r^*$ are equivalent metrics.
Also, note that $d_r^*$ still makes sense for the circle, yet in this case it is only a pseudometric. Nevertheless, 
it will be employed in situations where certain desired convergence also holds for the $\| \cdot \|_{\infty}$ term above 
because of some specific reason (for example, in the proof of the Interpolation Lemma \ref{l:chem}). 

Through this work, we will denote by ``$\id$'' either the trivial element of a group, the trivial action of a general group 
or the identity transformation of a certain space.
 
%%%%%%%%%%%%%%%%%%%%%%%%%%%%%%%%%%%%%%%%%%%%%%%%%%%%%%%%%%%%%%%%%%%

\section{The starting idea: conjugating actions via averages} 
\label{section-herman}

As we already mentioned, the starting idea of proof of several connectedness results listed in the introduction consists 
(whenever possible) in conjugating the original action so that it becomes closer and closer 
to an action by isometries. The conjugacies are obtained by some classical averaging procedure. For actions by $C^1$ 
diffeomorphisms, this idea was implemented in \cite{Na14} via the logarithmic derivative $\log D (\cdot)$. It was then 
extended to $C^{1+\mathrm{ac}}$ actions on the circle in \cite{Na23} using the affine derivative 
$D  \log D (\cdot)$, and plays a crucial role in Sections \ref{ss:rat} and \ref{ss:infinite} of the present article. 
The general idea of averaging actions is also at the core of the Regularization Lemma~\ref{l:lissage}.

For completeness of this work, below we give an elementary result in the continuous framework whose proof illustrates
 how the averaging procedure works. (Compare \cite[Proposition (2.2), Chapitre VII]{herman} and 
 \cite[Th\'eor\`eme A]{Na14}.)  
\footnote{Let us stress that this was made explicit in \cite{EN2}, but since this article will remain unpublished, 
we have included the complete argument here.}

\medskip

\begin{prop} 
\label{connexe-cont}
The space of $\mathbb{Z}^d$ actions by homeomorphisms of either the interval or the circle is path-connected.
\end{prop}

\begin{proof} 
By identifying the endpoints, the case of the interval can be deduced from that of the circle, so let us only consider 
this one (see also Remark \ref{r:alex} below). 
Let $f_1, \ldots, f_d$ be the images of the canonical generators of $\Z^d$, and let $\F_i$ be a lift of $f_i$ 
to the real line. 
Consider the map $\Phi_n$ defined as 
\begin{equation}\label{eq-conj-top}
\Phi_n (x) := \frac{1}{n^d} \sum_{0 \leq n_i < n} \F_{1}^{n_1} \F_2^{n_2} \cdots \F_d^{n_{d}} (x).
\end{equation}
Note that $\Phi_n$ is a homeomorphism, since it is continuous and strictly increasing. 
Since the maps $\F_j$ commute, for each $\F_i$ we have 
$$\Phi_n (\F_i (x)) 
= \frac{1}{n^d} \sum_{0 \leq n_j < n} \F_1^{n_1} \cdots \F_{i-1}^{n_{i-1}}  \F_i^{1+n_i}  \F_{i+1}^{n_{i+1}} \cdots 
\F_d^{n_d} (x),$$
and, again by commutativity, this equals 
$$\Phi_n (x) + \frac{1}{n^d} \Big[ \sum_{\substack{0 \leq n_j < n \\ j \neq i}} \F_i^n (\F_1^{n_1} \cdots \F_{i-1}^{n_{i-1}}  \F_{i+1}^{n_{i+1}} \cdots \F_d^{n_d} (x))  
    - \F_1^{n_1} \cdots \F_{i-1}^{n_{i-1}}  \F_{i+1}^{n_{i+1}} \cdots \F_d^{n_d} (x) \Big].$$   
Recall that\, $ ( \F^n_i (y) - y ) / n$ \, uniformly converges to the translation number $\rho (\F_i)$. 
Since there are $n^{d-1}$ terms of type \, $\F_i^n (y) - y$ \, in the right-side expression above, 
we deduce the (uniform) convergence 
$$\Phi_n (\F_i (x)) - \Phi_n (x) \longrightarrow \rho (\F_i) \qquad \mbox{ as } \quad n \to \infty.$$
Changing $x$ by $\Phi_n^{-1} (x)$, this yields 
$$\Phi_n (\F_i (\Phi_n^{-1}(x))) \longrightarrow x +  \rho (\F_i) \qquad \mbox{ as } \quad n \to \infty.$$
One readily checks that $\Phi_n$ commutes with the integer translations, hence induces a circle homeomorphism, 
that we denote by $\varphi_n$. 
The convergence above translates into that $\varphi_n f_i \varphi_n^{-1}$ uniformly converges to the rotation by 
$\rho (\F_i)$ mod. $\! \mathbb{Z}$, which is nothing but the rotation number of $f_i$. 
We have thus produced a sequence of conjugated actions that uniformly converges to an action by rotations. 
One can then construct a continuous path of such conjugates just by convex interpolation. 
More precisely, one considers circle homeomorphisms of the form $(1-s) \varphi_n + s \, \varphi_{n+1}$, 
with $s \in [0,1]$. 
Finally, having produced continuous paths of  conjugated actions ending at actions by rotations, 
one can connect any two $\mathbb{Z}^d$ actions just by moving the angles of these latter actions.
\end{proof}

\begin{rem} 
The proof above actually shows more: the natural inclusion of $\mathrm{SO}(2,\mathbb{R})^d$ in the space of 
continuous $\mathbb{Z}^d$ actions on the circle is a homotopy equivalence. (Compare \cite[Proposition 4.2]{ghys}.) 
We do not know whether this result extends to higher regularity.
\end{rem}

\begin{rem} 
One can produce a different proof of Proposition \ref{connexe-cont} by conjugating as in \cite{Na14} 
via quasi-invariant probability measures that evolve towards the Lebesgue measure. 
Both arguments apply more generally to nilpotent group actions. (The structural results from \cite{parkhe} 
can also be adapted to produce still another proof in the nilpotent case.) However, the argument via quasi-invariant 
probability measures applies more generally to actions of groups of subexponential growth, for which 
an averaging method as in the proof above is unavailable.
\end{rem}

\begin{rem}
\label{r:alex}
As previously mentioned, the case of the interval can be also ruled out using the classical Alexander trick. 
Note that this works for any group action by homeomorphisms of the interval, but it does not work for actions 
on the circle, even in the Abelian case. Also note that this argument cannot be applied in higher regularity.
\end{rem}

\begin{rem} As a complement to the previous remark, it is worth pointing out that, for any $k\ge1$, there is no deformation of  the whole 
space of commuting $C^k$-diffeomorphisms of the interval that arises from a global deformation of the group of diffeomorphisms via 
homomorphisms. Indeed, this follows from the nonconnectedness of the space of $C^1$ actions of Thomspson's group $F$ 
discussed at the end of the Introduction, but may be directly established as follows. By restricting the source, any homomorphism 
from $\mathrm{Diff}^k_+ ([0,1])$ into itself gives rise to one from $\mathrm{Diff}^r_c ((0,1))$, with $r=\max(k,3) \neq2$, into 
$\mathrm{Homeo}_+((0,1))$. Such homomorphism is ``standard'', meaning that there are countably many embeddings $\phi_i: (0,1) \to (0,1)$ 
with disjoint images such that $\rho$ is conjugate (via the product of the $\phi_i$) to the diagonal action on (the copies of) 
$(0,1) \times (0,1) \times \cdots$ (see \cite{chen-mann} and references therein). Now, a path of such homomorphisms  
cannot connect the standard action with the trivial one. Indeed, if for two diffeomorphisms $f,g$ there are $u,v$ in $(0,1)$ satisfying 
the resilient property $u < f(u) < f(v) < g(u) < g(v) < v$, then at least one of them has derivative smaller than $1/2$ somewhere in 
$(u,v)$. Since resilience is stable under topological conjugacy, this implies that $f,g$ admit no sequences of simultaneous conjugates 
both converging to the identity in $C^k$ topology. (Compare \cite[Question 24]{Na18}.)  
\end{rem}

\begin{rem} Inspired by the previous remark, one may still ask whether there are deformations of the space of diffeomorphisms that 
preserve commutativity. We suspect that, actually, most (all~?) of them come from group homomorphisms. We state this as a question.
\end{rem}

\begin{qs} Let $\rho$ be a continuous map from $\mathrm{Diff}^r_+(V)$ into itself, where $V$ is any manifold. Assume that 
$\rho$ sends commuting elements to commuting elements. Under what  conditions is $\rho$ a group homomorphism~?
\end{qs}

%%%%%%%%%%%%%%%%%%%%%%%%%%%%%%%%%%%%%%%

\section{Connecting conjugated actions}
\label{s:conjugate}

On many occasions in this work, we will use conjugacy to modify a given action to a ``better one''. 
If the initial action is by $C^r$ diffeomorphisms and the conjugacy $\varphi$ is also of class $C^r$, one can easily 
connect this action to its conjugate by $\varphi$ through a continuous path of $C^r$-conjugated actions, simply by 
connecting $\id$ to $\f$ by a continuous path of $C^r$ diffeomorphisms. 
However, we will sometimes need to ``concatenate'' conjugacies either ``in time'' or ``in space'':
\begin{itemize}
\item In Section \ref{ss:infinite}, for example, we will consider a sequence of conjugated actions 
that converges to an action by isometries, and we will intend to connect all of these conjugated actions 
between them to obtain a (continuous) path ending at the action by isometries 
(this is an example of ``time concatenation'').
\item In Section \ref{ss:idea}, we will subdivide the interval into possibly infinitely many pieces, 
on each of which we will modify the action, sometimes by conjugacy, 
and we will wish that everything matches up nicely (this is an example of ``space concatenation'').
\end{itemize}
 
 For these reasons, it will be important to keep a good control on the $C^r$ size of the path we construct 
 between an initial action and its conjugate. To do this, a technique that has been already implemented in 
 \cite{Na14} and \cite{Na23,EN2} in the $C^1$ and $C^{1+ac}$ settings, respectively, and that works tamely 
 provided $r \leq 2$ (cf. Remark \ref{r:higher}), is to ``affinely interpolate between $0=\log D (\id)$ and 
$\log D(\varphi)$''. 

The new key observation here is that, quite surprisingly, this still works when dealing with two $C^r$ actions 
that are only $C^1$ conjugated, which is the situation we will face in Section \ref{ss:reg}. 
This is precisely the content of the following seemingly inoffensive lemma, which deals with the case $r=1+ac$. 
For the $C^2$ case, see Lemma~\ref{l:chemC2} further on.

\begin{interp}
\label{l:chem}
Let $M$ denote either the interval $[0,1]$ or the circle. 
If two $C^{1+ac}$ actions $\rho_0$ and $\rho_1$ of a finitely generated group $G$ on $M$ are $C^1$ conjugated, 
then they are connected by a path $(\rho_t)_{t\in[0,1]}$ of $C^1$-conjugated $C^{1+ac}$ actions of $G$ 
such that, all along the path,  
$$d^*_{1+ac}(\rho_t,\id)\le\max(d^*_{1+ac}(\rho_0,\id),d^*_{1+ac}(\rho_1,\id)).$$
\end{interp}

\begin{proof} 
Let $f_i$ and $g_i$, $i\in[\![1,d]\!] := \{1,2,\ldots,d\}$, be the images of the canonical generators of $\Z^d$ under 
$\rho_0$ and $\rho_1$, respectively, and let $\varphi$ be a $C^1$ diffeomorphism of $M$ such that 
$\varphi f_i \varphi^{-1}=g_i$ holds for each $i$. For every $t\in[0,1]$, define the map $\varphi_t$ by letting 
$\varphi_t(0) :=0$ and $\log D\varphi_t(x) := t\log D\varphi(x) - c_t$, where $c_t := \log\int_0^1(D\varphi(y))^tdy$ 
is chosen so that $\int_0^1D\varphi_t=1$. 
Then $\varphi_t$ is a $C^{1}$ diffeomorphism of $M$, with $\varphi_0=\id$ and $\varphi_1=\varphi$.  
(In the circle case, note that $\log D \f_t (0) = t \, \log D\f (0) - c_t = t \, \log D\f (1) - c_t = \log D\f_t (1)$.)
Moreover, $(\varphi_t)$ is a continuous path for the $C^1$ topology. 

Define $\rho_t$ as the conjugate of $\rho_0$ by $\varphi_t$. If we let 
$f_{i,t} := \varphi_t f_i \varphi_t^{-1}$ for every $i\in[\![1,d]\!]$, then
\begin{align}
\label{e:chem}
(\log Df_{i,t})\circ \varphi_t 
&= (\log D\varphi_t)\circ f_i + \log Df_i -\log D\varphi_t\notag\\
& = t \, (\log(D\varphi)\circ f_i + \log Df_i -\log D\varphi) + (1-t)\log Df_i\notag\\
& =  t \, (\log D (\varphi f_i \varphi^{-1}) )\circ \varphi + (1-t)\log Df_i \notag\\
& =  t \, ( \log Dg_i) \circ \varphi +  (1-t)\log Df_i.
\end{align}
As a consequence, $f_{i,t}$ is $C^{1+ac}$, with 
\begin{eqnarray*}
\var ( \log Df_{i,t}  ) 
&=& \var ( (\log Df_{i,t}) \, \circ \, \varphi_t ) \\
&\leq&   t \, \var (( \log Dg_i) \circ \varphi ) + (1-t) \, \var ( \log Df_i) \\ 
&=&  t \, \var ( \log Dg_i) +  (1-t) \, \var ( \log Df_i). 
 \end{eqnarray*}
This immediately gives the estimate for $d^*_{1+ac} (\rho_i, id)$ announced in the statement. 
In order to check that $t\mapsto f_{i,t}$ is continuous in the $C^{1+ac}$ topology, 
let $(t_n)$ be a sequence in $[0,1]$ converging to a certain $t$. Then 
$$
\var(\log   Df_{i,t_n} \! -\log Df_{i,t})
=\var( \log Df_{i,t_n}\circ \varphi_{t_n}-\log Df_{i,t} \circ \varphi_{t_n} )$$
is smaller than or equal to
\begin{equation} 
\label{expr}
\var(\log Df_{i,t_n}\circ \varphi_{t_n}-\log Df_{i,t} \circ \varphi_{t} )
+ \var(\log Df_{i,t}\circ \varphi_{t}-\log Df_{i,t} \circ \varphi_{t_n} ).
\end{equation}
The first term of this expression equals
$$\var\Bigl((1 \! - \! t_n)\log Df_i + t_n (\log Dg_i) \! \circ \! \varphi - (1 \! - \! t) \log Df_i - t (\log Dg_i) \! \circ \! \varphi\Bigr),$$
that is, 
$$\! |t-t_n|\cdot\var\Bigl(\log Df_i - (\log Dg_i) \! \circ \! \varphi\Bigr),$$
which goes to $0$ when $n$ goes to infinity. 
For the second term in (\ref{expr}), we have:
$$\var\Bigl( \log Df_{i,t}\circ \varphi_{t}-\log Df_{i,t} \circ \varphi_{t_n}  \Bigr) 
= \var\Bigl( \log Df_{i,t}\circ (\varphi_{t}\circ \f_{t_n}^{-1})-\log Df_{i,t}  \Bigr).$$
Lemma \ref{l:astuce} below shows that this expression converges to $0$ as $n$ goes to infinity. 
Putting all of this together, we get that 
$$d^*_{1+ac} (f_{i,t_n},f_{i,t}) \xrightarrow[n\to+\infty]{}0.$$
This  proves the announced continuity in the interval case, since $d^*_{1+ac}$ induces the $C^{1+ac}$ topology. 
In the circle case, one needs to observe that, moreover, the map $t\mapsto f_{i,t}$ is continuous in the $C^1$ topology 
(hence in the $C^0$ one), since $t\mapsto \varphi_t$ is.
\end{proof}

\begin{rem}
\label{r:higher}
Using formula \eqref{e:chem}, one can readily check that if the initial actions are by $C^r$ diffeomorphisms, 
with $r \in \{1+bv, 1+ac, 2 \}$, then the actions along the deformation path are still of class $C^r$, 
even if the $\varphi_t$'s are only $C^1$. However, if we assume $r>2$, this formula does not ensure that the 
actions along the path are $C^r$. This is one of the places where our techniques fail to generalize to higher regularity 
(cf. Section \ref{ss:Cr}). 
\end{rem}

\begin{lem}
\label{l:astuce}
If $u$ is an absolutely continuous function defined on $M$ and $(\varphi_n)$ is a sequence of $C^1$ diffeomorphisms 
of $M$ converging to $\id$ in the $C^1$ topology, then $\var(u\circ\varphi_n - u)$ converges to $0$.
\end{lem}

\begin{proof} 
Fix $\eps>0$, and let $v$ be a continuous function such that $\|Du - v\|_{L^1}<\frac\eps4$. Then
\begin{align*}
\var(u\circ\varphi_n - u) 
&=\| D (u \circ\varphi_n - u) \|_{L^1} \\
&= \|Du\circ\f_n\cdot D\f_n-Du\|_{L^1}\\
&\le \|Du\circ\f_n\cdot D\f_n-v\circ \f_n \cdot D\f_n\|_{L^1} + \|v\circ \f_n \cdot D\f_n - v\circ \f_n\|_{L^1} \\
&\hspace{8cm}+\|v\circ \f_n - v\|_{L^1}+\|v-Du\|_{L^1}.
\end{align*}
The first and last terms in the sum above are smaller than $\frac\eps4$ (for the first term, this follows from a change of 
variable). 
For the second term we have: 
$$\|v\circ \f_n \cdot D\f_n - v\circ \f_n\|_{L^1}\le \|v\|_\infty\cdot\|D\f_n-1\|_\infty,$$ 
and the last expression is smaller than $\varepsilon / 4$ for $n$ sufficiently large. 
Finally, for the third term, $v$ being uniformly continuous, there exists $\eta>0$ such that $|v(x)-v(y)|\le \frac\eps4$ 
as soon as $|x-y|<\eta$, 
Therefore, if we take a large-enough $n$ so that $\|\f_n-\id\|_\infty<\eta$, then 
$$\|v\circ \f_n - v\|_{L^1}\le \|v\circ \f_n - v\|_{\infty}<\tfrac\eps4$$
which concludes the proof.
\end{proof}

\begin{rem}
Among all the steps involved in the proof of the Main Theorem \ref{t:C1ac} in the interval case,  
%in Section \ref{s:C1ac-int}, 
Lemma~\ref{l:astuce} above (which is used in the proof of the Interpolation Lemma \ref{l:chem}) 
is the only place where we know that $C^{1+ac}$ cannot be replaced by $C^{1+bv}$; see Appendix 4. Actually, 
if it were not for the use of this lemma, we would be also able to prove path-connectedness for $C^{1+bv}$ 
actions of $\mathbb{Z}^d$ on the interval (but still not on the circle; cf.~Section~\ref{ss:infinite}).
\end{rem}

%%%%%%%%%%%%%%%%%%%%% %%%%%%%%%%%%%%%%%%%%%%%%%%%%%%%%%%

\section{Path-connectedness of $C^{1+ac}$ actions on the interval}
\label{s:C1ac-int}

\subsection{General idea and discussion}
\label{ss:idea}

In any regularity $C^r$ with $r \geq 1$ (including $r=1+bv$ and $r=1+ac$), there are two particular types of actions 
$\rho:\Z^d\to\Diff^{r}_+([0,1])$ that are easy to deform to the trivial one: 
those with infinite cyclic image and those whose image embeds into a $C^{r}$ flow. For the former, one is 
reduced to deforming a single diffeomorphism, and this can be done by convex interpolation, for example. 
For the latter, one can continuously deform the flow to the trivial one just by reparametrization of time, 
and this induces a deformation of the original action to the trivial action. Our general strategy to prove the 
Main Theorem \ref{t:C1ac} for the case of the interval consists in trying to reduce to these favorable situations. 

Recall that a {\em global fixed point} of an action is a point that is fixed by all group elements. Given an action 
$\rho:\Z^d\to\Diff^{1+ac}_+ (M)$, let us denote by $P$ the set of global fixed points at which all elements of 
$\rho(\Z^d)$ are parabolic. We will refer to them as {\em parabolic global fixed points}. 

Our proof involves three different aspects:
\begin{itemize}
\item We start by showing that the restriction of a $C^{1+ac}$ action of $\mathbb{Z}^d$ 
to the closure of every connected component of $[0,1]\setminus P$ is either of the first type above 
or \emph{almost} of the second type, in the sense that its image embeds into a $C^1$ flow, but some 
elements of this flow may fail to be $C^{1+ac}$ diffeomorphisms (such a situation can occur, according to 
Proposition~\ref{p:Sergeraert}). 
\item In each of these cases, we develop a method to connect the action to the trivial one through $C^{1+ac}$ actions.  
Roughly, in the former, we  deform the generator of the image group, and in the latter, 
we first connect to an action which embeds into a $C^{1+ac}$ flow, and then we deform the flow.
\item{This needs to be done keeping a good control on the size of the deformation on each component 
(of which there can be infinitely many), so that everything matches up nicely when performing all of them 
simultaneously.} 
\end{itemize}

In order to implement this strategy of proof of the Main Theorem \ref{t:C1ac} for the case of the interval, the three following statements will be crucial:

\begin{prop}
\label{p:noncyclic}
If a $C^{1+bv}$ action of $\Z^d$ on $[0,1]$ has no interior parabolic global fixed point and has a noncyclic image, 
then it embeds into a $C^1$ flow.
\end{prop}

\begin{prop}
\label{p:flot}
Let $\rho_0$ be a $C^{1+ac}$ action of $\Z^d$ on $[0,1]$ that embeds into a $C^1$ flow. 
Then $\rho_0$ can be continuously deformed to the trivial action~$\rho_1=\id$ through a path 
$(\rho_t)_{t \in [0,1]}$ satisfying $d^*_{1+ac}(\rho_t,\id)\le d^*_{1+ac}(\rho_0,\id)$ for every $t\in[0,1]$. Moreover, 
if the diffeomorphisms of the initial action are parabolic at an endpoint, so are all the diffeomorphisms along the 
deformation.
\end{prop}

\begin{prop}
\label{p:cyclic}
Every $C^{1+ac}$ action $\rho_0$ of $\Z^d$ on $[0,1]$ with a cyclic image can be continuously deformed to the
 trivial action $\rho_1$ through a path $(\rho_t)_{t \in [0,1]}$ satisfying 
 $d^*_{1+ac}(\rho_t,\id)\le 2 \, d^*_{1+ac}(\rho_0,\id)$ for every $t\in[0,1]$. 
 Moreover, if the diffeomorphisms of the initial action are parabolic at an endpoint, so are all the diffeomorphisms 
 along the deformation.
\end{prop}

Assuming the validity of these statements, we proceed to the proof of the Main Theorem \ref{t:C1ac} for actions 
on the interval. 

\begin{proof}[Proof of the Main Theorem \ref{t:C1ac} for the interval] 
Let $\rho_0$ be a $C^{1+ac}$ action of $\Z^d$ on the interval, $P$ the set of its global parabolic fixed points 
and $I$ a connected component of $U := [0,1]\setminus P$. By Proposition \ref{p:noncyclic}, the restriction of 
$\rho_0$ to the closure $\bar I$ of $I$ either embeds into a $C^1$ flow or has a cyclic image. 
In the former case, we can apply Proposition~\ref{p:flot} to this restriction; in the latter case, we can apply Proposition 
\ref{p:cyclic}. In any case, we obtain a deformation of the original action restricted to $\bar I$ to the trivial action 
on $\bar I$. 
We claim that the path $(\rho_t)$ obtained by gluing together all these pieces of deformations of actions 
is a continuous path of $C^{1+ac}$ actions connecting $\rho_0$ to the trivial action at $t=1$. To show this, 
we need to check that $\rho_t$ defines a $C^{1+ac}$ action of $\mathbb{Z}^d$ for each $t \in [0,1]$, and that 
these actions depend continuously on $t$ for the $C^{1+ac}$ topology.

For the first point, fix $t\in(0,1)$, and let $f$ be the image by~$\rho_t$ of a standard generator of $\Z^d$. It is clear that 
$f$ is $C^1$ in restriction to the closure of every connected component of $[0,1]\setminus P$.  

To prove that $f$ is $C^1$ on the whole interval, it remains to show that $f$ is differentiable at every point of $P$, 
with derivative equal to $1$, and that the thus well-defined map $Df$ has limit $1$ at every point of $P$. 
To fix ideas, we will only study limits and differentiability from the right, the other side being analogous.

So let $p\in P$. If $p$ is isolated from the right in $P$, that is, if $p$ is the infimum of a connected component 
$I$ of $U$, then Propositions \ref{p:flot} and \ref{p:cyclic} imply that $f\res{\bar I}$ is $C^1$ and parabolic at $p$, 
as required.

Let us now assume that $p$ is an accumulation point of $P$ from the right, and let $x>p$. 
If $x$ is a fixed point of $f$, then $\frac{f(x)-f(p)}{x-p}=1$. 
Otherwise, let $[l_x,r_x]$ be the closure of the connected component $I_x$ of $U$ containing $x$. 
To fix ideas, assume $f(x)>x$.  Then
\begin{align*}
1\le \frac{ f(x) -  f(p) }{x-p}
= 1+\frac{f(x)-x}{x-p}
&\le 1+ \frac{f(x)-x}{x-l_x} 
= \frac{f(x)-f(l_x)}{x-l_x}
=Df(y_x)
\end{align*}
for some $y_x\in [l_x,r_x]$, by the Mean Value Theorem. Hence,
$$\left|\log \left( \frac{f(x)-f(p)}{x-p} \right) \right|\le \sup_{[l_x,r_x]}|\log Df| 
\le \underbrace{|\log Df(l_x)|}_0 + \,\var(\log Df\res{ [l_x,r_x]}).$$
The controls in Propositions \ref{p:flot} and \ref{p:cyclic} show that the last term is bounded above by 
$$2 \, d^*_{1+ac}(\rho\res{[l_x,r_x]},\id\res{[l_x,r_x]}),$$ 
which goes to $0$ when $x$ goes to $p$. One argues similarly when $f(x)<x$. Hence the right-derivative of $f$ at $p$ is well defined and equal to~$1$. 

The proof of the continuity of $Df$ at every point $p\in P$ proceeds similarly. Again, we will only deal with 
right-continuity. The case where $p$ is isolated from the right in $P$ has already been settled. 
Now, if $p$ is the limit of a decreasing sequence $(p_n)$ in $P$, then 
$$\sup_{[p,p_n]}|\log Df|=\sup_{[p,p_n]\setminus P}|\log Df|\le  2 \, d^*_{1+ac}(\rho\res{[p,p_n]},\id\res{[p,p_n]}),$$
which goes to $0$ when $n$ goes to infinity. Therefore, $Df(x)$ goes to $1=Df(p)$ when $x$ goes to $p$, as required.

Hence $\rho_t$ is a $C^1$ action, and it is immediate from the controls in Propositions \ref{p:flot} and \ref{p:cyclic} 
that it is also a $C^{1+ac}$ action, 
with $d^*_{1+ac}(\rho_t,\id) \le 2 \, d^*_{1+ac}(\rho,\id)$. 

Let us finally show that $t\mapsto \rho_t$ is continuous for the $C^{1+ac}$ topology. 
Fix $t_0\in[0,1]$ and $\eps>0$. Let $\cI$ denote the set of connected components of $U$. 
One readily checks that there exists a finite subset $\cJ$ of $\cI$ such that, if $E=[0,1]\setminus \cup_{I\in \cJ}I$, then 
$$d^*_{1+ac}(\rho\res{E},\id\res{E})<\frac\eps8.$$
Let $N$ be the cardinality of $\cJ$. Since  $t\mapsto \rho_t\res{\bar I}$ is continuous for the $C^{1+ac}$ topology 
on every $I\in\cI$, there exists $\eta>0$ such that, for every $I\in\cJ$ and every $t$ satisfying $|t-t_0|<\eta$,
\begin{equation}
\label{e:eps2N}
d^*_{1+ac}({\rho_t}\res{\bar I},{\rho_{t_0}}\res{\bar I})< \frac\eps{2N}.
\end{equation}
Hence,
\begin{align*}
d^*_{1+ac}({\rho_t},{\rho_{t_0}})
&\le d^*_{1+ac}(\rho_t\res{E},\id\res{E})+d^*_{1+ac}(\rho_{t_0}\res{E},\id\res{E})
+\sum_{I\in\cJ}d^*_{1+ac}({\rho_t}\res{\bar I},{\rho_{t_0}}\res{\bar I})\\
&\le 4 \, d^*_{1+ac}(\rho\res{E},\id\res{E})+\frac\eps{2N}\cdot|\cJ| \, 
< \, \eps,
\end{align*}
where the second inequality follows again from the controls ensured by Propositions \ref{p:flot} and \ref{p:cyclic}, 
together with \eqref{e:eps2N}, and this concludes the proof.
\end{proof}

At this point, it is worth comparing in a more accurate way the proof above with that of the weaker result established 
in \cite{EN2}. In \cite{EN1}, we proved that there are two obstructions to connect a given $C^{1+ac}$ action of 
$\mathbb{Z}^d$ on the interval to the trivial action via $C^{1+ac}$ conjugacies: 
the existence of hyperbolic fixed points and the non-$C^1$-flowability. 
With this in mind, in \cite{EN2}, we subdivided the interval into {\em components}, that is, fixed intervals having 
no global fixed point inside, and dealt with each component separately, 
distinguishing between the $C^1$-flowable ones and the others.

For the non $C^1$-flowable components, our method has not changed. 
This situation emerges here in Proposition \ref{p:cyclic} which, actually, essentially corresponds to 
\cite[Proposition 4.2]{EN2}. In order to keep this article as self-contained as possible, 
we provide a short and direct proof in Section~\ref{ss:rat}.

On $C^1$-flowable components, the strategy of \cite{EN2} consisted in somewhat getting rid of 
the hyperbolic fixed points using Sternberg's linearization theorem  \cite{St}, and then deforming the action 
to the trivial one by conjugacy. However, as we already mentioned in the Introduction, 
this theorem requires $C^2$ regularity, and strongly fails in class $C^{1+ac}$. 
In our new approach, hyperbolic fixed points do not play a major role thanks to the Regularization 
Lemma~\ref{l:lissage} in Section \ref{ss:reg}. 
This is the core of the proof of Proposition~\ref{p:flot} above, as it allows us to reduce to a well-controlled 
$C^{1+ac}$-flowable action. 

There is still another subtle difference which is hidden in the statement of Proposition~\ref{p:noncyclic}. 
Namely, instead of dealing with the aforementioned ``components'', here it is much more appropriate to deal with 
fixed intervals having no parabolic global fixed point inside, thus allowing the existence of hyperbolic fixed points 
in the interior. This is related to the fact that our treatment of these points has radically changed.

Let us end this panoramic section by discussing the validity of its core statements, namely Propositions 
\ref{p:noncyclic}, \ref{p:flot} and \ref{p:cyclic}, in higher regularities. 
First, Proposition \ref{p:noncyclic} extends to class $C^r$, $r\ge 2$: see Proposition \ref{p:noncyclicC2} in 
Section \ref{ss:irrat}. 
Also, Proposition \ref{p:flot} extends to the $C^2$ setting: see Proposition \ref{p:flotC2} in Section \ref{ss:C2-int}. 
However, our method of proof of Proposition \ref{p:cyclic} implemented in Section \ref{ss:rat} is very specific to the 
$C^{1+ac}$ (or $C^{1+bv}$) setting; see Remark~\ref{rem-specific} in this regard. 
In particular, it does not yield the desired control in $C^2$ regularity, and this is why we only manage to prove 
connectedness and not \emph{path}-connectedness in the $C^2$ context (cf. Section \ref{s:C2}). 

%%%%%%%%%%%%%%%%%%%%%%%%%%%%%%%%%%%%%%%%%%%%%%%%%%%%%%%%%%%%%%%%%%%%%%%%%%%%%%%%%%%%

\subsection{Noncyclic components are flowable}
\label{ss:irrat}

This section is devoted to the proof of Proposition \ref{p:noncyclic}, for which we closely follow an argument from 
\cite{BE} that will be also employed later in the $C^r$ case, $r \geq 2$ (cf. Proposition \ref{p:noncyclicC2}). 
In the present $C^{1+bv}$ setting, the proof strongly uses a variation of works by Szekeres, Sergeraert and Yoccoz 
presented in Appendix~\ref{ss:Szek}, as well as the famous Kopell Lemma, also recalled therein. 

We start with a couple of general lemmas concerning sets of fixed points of group elements. 
Henceforth, given an action $\rho$ of a group on some space, we denote $\Fix (\rho)$ the set of its global fixed points. 

\begin{lem}
\label{l:hypfp}
If for a $C^{1+bv}$ action of $\Z^d$ on $[0,1]$ one of the elements in the image group has a hyperbolic 
interior fixed point, then every other element that is not the identity near this point must hyperbolically fix it.
\end{lem}

\begin{proof} 
Assume some diffeomorphism $f$ of the image group has a hyperbolic fixed point $a\in(0,1)$. We first claim 
that any other element $g$ must fix $a$. 
Otherwise, by commutativity, the orbit of $a$ under $g$ would be made of fixed points of $f$, 
all of them with the same multiplier. However, these points would accumulate on certain points to the left 
and to the right of $a$, and these are necessarily parabolic fixed points of $f$. This is absurd, hence $g$ fixes $a$.  

Assume now that some element $g$ of the image group fixes $a$ but acts nontrivially on every neighborhood of $a$. 
For the sake of concreteness, assume that nontriviality happens to the right of $a$, the other case being analogous. 
Change $f$ by its inverse if necessary so that it becomes repelling at $a$, and fix $\varepsilon > 0$ such that 
$f$ has no fixed point in $(a, a+ \varepsilon)$. 
According to Kopell's lemma, $g$ has no fixed point in $(a, a + \varepsilon)$. Changing it by its inverse if necessary, 
we may assume that it is topologically repelling on this interval. 
Fix a point $b$ to the right of $a$ but close enough to it so that $f(b) < a+\varepsilon$. 
We may choose $n \in \mathbb{N}$ large enough so that $g^n (b) > f (b)$. If $g$ were parabolic at $a$, 
then we would have $g^n (x) < f (x)$ for all $x > a$ very close to $a$. 
This would imply that $g^{-n} f$ has a fixed point in $(a,b)$. 
Since the map $g^{-n} f$ commutes with $f$, another application of Kopell's lemma would show that it coincides with 
the identity on a right neighborhood of $a$. 
Nevertheless, this would force the derivative of $g^n$ to be equal to that of $f$ at $a$, which is absurd. 
Therefore, $g$ must be hyperbolic at $a$, which completes the proof.
\end{proof}

\begin{rem} 
In class $C^2$, the proof of the preceding lemma can be greatly simplified by using a combination of the classical 
results of Szekeres and Kopell. Indeed, if $f$ has an isolated fixed point $a$ and $g$ commutes with $f$ 
and fixes $a$, these results imply that $f$ and $g$ locally belong to the same flow of a $C^1$ vector field 
vanishing at $a$ (cf. Remark \ref{r:folk}). 
If this zero is hyperbolic/parabolic, all the nontrivial elements of the flow are hyperbolic/parabolic at $a$. 
This argument still works in class $C^{1+bv}$, but requires the variation of Szekeres' theorem 
given in Appendix~\ref{a:szekeres} (cf. Proposition~\ref{p:Szek}).  
This is why we preferred to give the down-to-earth (yet elementary) argument above. 
\end{rem}

\begin{lem}
\label{l:meme-PF}
If a $C^{1+bv}$ action of $\Z^d$ on $[0,1]$ has no interior parabolic global fixed point, then all the 
nontrivial elements of its image have exactly the same fixed points.
\end{lem}

\begin{proof} 
Assume some diffeomorphism $g$ of the image group has a fixed point $a\in(0,1)$ that is not fixed 
by some other element $f$. Let us show that $g$ is necessarily the identity. Let $(b,c)$ be the connected 
component of $[0,1]\setminus\Fix(f)$ containing $a$. By commutativity, the orbit of $a$ under $f$ is made 
of fixed points of $g$ which accumulate on $b$ and $c$,  so $g$ fixes $b$ and $c$ as well. By Kopell's 
lemma, this implies that $g$ is the identity on $[b,c]$. 

Now let $[b',c']$ be the maximal interval containing $[b,c]$ on which $g$ is the identity. Assume by contradiction 
that $[b',c']$ is not all of $[0,1]$. For example, assume $c' < 1$, the case where $b' > 0$ being analogous. 
We claim that $c'$ is a global fixed point of the action. Otherwise, there would exist 
$f$ in the image group such that $f(c') > c'$, and Kopell's lemma again would yield that 
$g$ is the identity on an open neighborhood of the point $c'$, which contradicts its definition.
 
By hypothesis, $c'$ is not a parabolic global fixed point, hence there exists $\tilde{f}$ in the image group 
that is hyperbolically repelling at $c'$. By definition of $c'$, the diffeomorphism $g$ must act nontrivially 
on arbitrarily small (right) neighborhoods of $c'$. 
By the previous lemma, the element $g$ must hyperbolically fix $c'$. But since $g$ pointwise fixes $[b',c']$, 
it is parabolic at $c'$. This contradiction completes the proof.
\end{proof}
\medskip
  
\begin{proof}[Proof of Proposition \ref{p:noncyclic}] 
Let $\rho$ be a $C^{1+bv}$ action of $\Z^d$ on $[0,1]$ with no interior parabolic global fixed point 
and with noncyclic image. According to the previous lemma, all the nontrivial elements of the image 
group have exactly the same fixed points. Let $f$ be such an element. 
By Proposition \ref{p:Szek}, for every connected component $(a,b)$ of $[0,1]\setminus\Fix(\rho)=[0,1]\setminus \Fix(f)$, 
the map $f$ is the time-$1$ map of a $C^1$ flow $(f^t_{[a,b)})$ on $[a,b)$ and $(f^t_{(a,b]})$ on $(a,b]$. 
Moreover, according to Kopell's lemma (see also Remark \ref{r:folk}) for any $g$ in the image group there exist 
numbers $\alpha$ and $\beta$ such that the restriction of $g$ to $[a,b)$ (resp. $(a,b]$) is equal to $f^\alpha_{[a,b)}$ 
(resp. $f^\beta_{(a,b]}$). 
In fact, these numbers are equal. A simple way of seeing this is the following: given $c\in(a,b)$, the maps 
$\phi:t\mapsto f^t_{[a,b)}(c)$ and $\psi:t\mapsto f^t_{(a,b]}(c)$ define $C^1$ diffeomorphisms from $\R$ to $(a,b)$ 
which conjugate $f$ to the translation by $1$ on $\R$ and $g$ to the translations by $\alpha$ and $\beta$, 
respectively. Hence, these two translations are conjugated by $\phi^{-1}\psi$, which commutes with the unit translation. 
This implies that $\alpha=\beta$ because of the following elementary fact whose proof is left to the reader: 
if two translations of the real line are conjugated by an orientation-preserving homeomorphism 
that commutes with a nontrivial translation, then they coincide. We call $\alpha$ the {\em translation number} of $g$ 
with respect to $f$. 

Since there is no interior parabolic global fixed point and $\log Df^\alpha_{[a,b)}(a)=\alpha\log Df(a)$, 
the value of $\alpha$ is determined by the derivative of $g$ at $a$. As a consequence, 
if $(c,a)$ and $(a,b)$ are two neighbor connected components of $[0,1]\setminus\Fix(\rho)$, 
the translation number of $g$ with respect to $f$ is the same on both components. 
Therefore, this number is constant on $[0,1]\setminus\Fix(\rho)$. 

If all the relative translation numbers of the generators of $\rho$ were rational, the action would be cyclic, 
which is contrary to our hypothesis. So there is some generator $g$ whose relative translation number $\alpha$ is 
irrational. In this case, for every connected component $(a,b)$ of $[0,1]\setminus\Fix(\rho)$, the restrictions 
$f^t_{[a,b)}$ and $f^t_{(a,b]}$ coincide for every $t\in\Z+\alpha\Z$. By continuity, this actually holds for every $t\in\R$. 

Finally, for each $t\in\R$, let us define $f^t$ by letting  $f^t := \id$ on $\Fix(f)$ and $f^t := f^t_{[a,b)}=f^t_{(a,b]}$ 
in restriction to every connected component $(a,b)$ of $[0,1]\setminus\Fix(\rho)$. 
We thus obtain a flow whose restriction to each component of $[0,1]\setminus\Fix(\rho)$ is $C^1$. Moreover, 
everything glues up nicely at each isolated global fixed point of the action, so the flow is $C^1$ on $(0,1)$, 
and the image group embeds into this flow. 
The last thing to check is that the flow is $C^1$ when one includes the endpoints $0$ and $1$. 
The only problematic case is when some of them is accumulated by interior fixed points. 
But then the control given in Proposition \ref{p:Szek} guarantees that the derivatives with 
respect to $x$ of $(t,x)\mapsto f^t(x)$ converge to $1$ as $x$ goes to the boundary, which concludes the proof.
\end{proof}

In the following ``$C^r$ version'' of Proposition \ref{p:noncyclic}, we say that a global fixed point of an action is 
{\em $r$-parabolic} if every diffeomorphism of the image group is $C^r$-tangent to the identity at that point.

\begin{prop}
\label{p:noncyclicC2}
Let $r\ge2$ be an integer. If a $C^{r}$ action $\rho$ of $\Z^d$ on $[0,1]$ with a noncyclic image has no interior 
$r$-parabolic global fixed point, 
then it embeds into a $C^{1}$ flow $(f^t)$ that is $C^{r-1}$ away from the possible $r$-parabolic fixed points at the 
boundary.
\end{prop}

The proof of this statement follows the same lines as that of Proposition \ref{p:noncyclic}, with an additional ingredient.
 Namely, in this context, if $b$ is an interior fixed point, then it is isolated, since non-$2$-parabolic. 
 Thus, letting $(a,b)$ and $(b,c)$ the corresponding neighboring connected components of $[0,1]\setminus \Fix(\rho)$, 
 then, according to \cite[Appendice IV]{Yo}, the flows $(f^t_{(a,b]})$ and $(f^t_{[b,c)})$ 
arise from $C^{r-1}$ vector fields on $(a,b]$ and $[b,c)$ that glue as a $C^{r-1}$ vector field on $(a,c)$. 
Moreover, the non-$r$-parabolicity also implies that relative translation numbers do not change between two 
neighbor components. Details are left to the reader; see also \cite{EyNa23}.

%%%%%%%%%%%%%%%%%%%%%%%%%%%%%%%%%%%%%%%%%%%%%%%%%%%%%%%%%%%%%%%%%%%%%%

\subsection{Deformation of flowable actions}
\label{ss:reg}

In this section, we prove Proposition \ref{p:flot}. To do this, we use the next subtle and somewhat miraculous lemma, 
which is actually the second new key ingredient of our strategy. The argument of proof is a mixture of \cite{Ha}, 
\cite{Na14} and \cite[Prop.~8.4]{EN1}. 

\begin{reg}
\label{l:lissage} 
If $(f^t)$ is a $C^1$ flow of diffeomorphisms of $[0,1]$ whose time-$1$ map $f = f^1$ is of class $C^{r}$, 
with $r \in \{1,1+bv,1+ac\}$, then  
it is conjugated by a $C^1$ diffeomorphism to the flow of a $C^r$ vector field $\tilde{X}$ satisfying  
$\var(D\tilde X) = \var(\log Df)$.
\end{reg}

\begin{proof}[Proof]  
Let $X$ be the generating vector field of the flow, and let $\varphi: [0,1] \to \R$ be the map defined by 
$\varphi (0):= 0$ and 
\begin{equation}
\label{e:defh}
\log D\varphi (x) := \int_0^{1} \log Df^s(x) ds - c, \quad\text{with} \quad c:= \log\left(\int_0^1\exp\left(\int_0^{1} \log Df^s(x) ds\right)dx\right).
\end{equation}
Then $\varphi$ is of class $C^{1}$, has positive derivative everywhere, and satisfies $\int_0^1D\varphi(x)dx=1$. 
As a consequence, it is a  $C^1$ diffeomorphism of $[0,1]$.

Let us next check that $\varphi$ conjugates $(f^t)$ to the flow generated by a $C^{r}$ vector field, 
i.e. that the vector field $\tilde X := \varphi_*X$ is of class $C^{r}$, and that this vector field satisfies the announced 
equality $\var(D\tilde X)=\var(\log Df)$.
To this end, first consider an arbitrary point $x \in [0,1]\setminus X^{-1}(\{0\})$. We have 
$(\tilde X \circ \varphi )(x) = X(x) \cdot D\varphi (x) \neq 0$. 
Hence, 
\begin{align*}
\log(\tilde X\circ \varphi)(x) 
&= \log X(x)+\log D\varphi (x) 
=\log X (x)+\int_0^1 \log Df^s(x) \, ds \, -c\\
&=\int_0^1 \log(X\cdot Df^s(x)) \, ds \, -c
=\int_0^1 \log(X\circ f^s)(x) \, ds \, -c \\
&=\int_0^1 \frac{\log(X\circ f^s)(x)}{(X\circ f^s)(x)}\cdot \tfrac d {ds} f^s(x) \, ds \, -c 
=\int_x^{f(x)}\frac{\log X(y)}{X(y)} \, dy \, -c.
\end{align*}
Thus, $\log(\tilde X\circ \varphi)$ is differentiable at $x$, with derivative
$$\frac{\log(X\circ f) (x)}{(X\circ f) (x)}\cdot Df (x) - \frac{\log X (x)}{X(x)} =\frac{\log Df (x)}{X (x)},$$
where the equality follows from applying twice the relation $X \circ f = Df \cdot X$. 
As a consequence, $\tilde{X} \circ \varphi$ is also differentiable at $x$, and
$$D ( \tilde{X} \circ \varphi ) (x) 
= D \log (\tilde{X} \circ \varphi) (x) \cdot ( \tilde{X} \circ \varphi ) (x) 
= \frac{\log Df (x)}{X (x)} \cdot X(x) \cdot D\varphi (x) = \log Df (x) \cdot D\varphi(x).$$
This shows that $\tilde X \circ \varphi$ is differentiable on $[0,1]\setminus X^{-1}(\{0\})$, 
and its derivative therein coincides with the \emph{continuous} map \, $\log Df\cdot D\varphi$. 

We now want to show that this actually holds on all of $[0,1]$, i.e. also at every zero $a$ of $X$. 
We distinguish two cases according to whether $a$ is an isolated zero or not.

In the first case, there exists an open interval $U$ of $[0,1]$ containing $a$ such that $a$ is the only 
zero of $X$ in $U$. 
By the above discussion, $\tilde X\circ \varphi$ is $C^1$ on $U\setminus\{a\}$ and its derivative on this domain 
extends to the continuous function $\log Df\cdot D\varphi$ on $U$. 
By a standard extension theorem for $C^1$ maps, this implies that $\tilde X\circ \varphi$ is $C^1$ on 
all of $U$ and its derivative there coincides with $\log Df\cdot D\varphi$.  

In the second case, $a$ is necessarily a global parabolic fixed point of the flow, which implies that 
$(\log Df\cdot D\varphi)(a)=0$. We are thus reduced to showing that $\tilde X\circ \varphi$ is differentiable 
with derivative equal to $0$ at $a$. Now, since $\tilde X\circ \varphi = X\cdot D\varphi$, for all $x\neq a$ we have 
$$\frac{(\tilde{X} \circ \varphi )(x) - (\tilde{X} \circ \varphi) (a)}{x-a} 
= \frac{(X\cdot D\varphi)(x)-(X\cdot D\varphi)(a)}{x-a} 
= \frac{X(x)}{x-a}\cdot D\varphi (x).$$
Since $D\varphi (x)$ has a finite limit at $a$, we are left with checking that $\frac{X(x)}{x-a}$ tends to $0$ as 
$x$ goes to $a$. If $x$ is a zero of $X$ (or equivalently a fixed point of $f$), the latter quantity equals $0$. 
In the opposite (more complicated) case, 
\begin{align}
\label{e:derivX}
\frac{X(x)}{x-a} = \frac{X(x)}{f(x)-x}\cdot \frac{f(x)-x}{x-a}.
\end{align}
For the first term of this product, we have
$$\frac{f(x)-x}{X(x)} = \frac{\int_0^1\frac{d}{ds}f^s(x) ds}{X(x)} 
= \int_0^1 \frac{(X\circ f^s)(x)}{X(x)}ds = \int_0^1 Df^s(x) ds.
$$
This expression converges to $1$ when $x$ goes to $a$ because $a$ is a parabolic fixed point for every $f^s$ 
and $(s,x)\mapsto Df^s(x)$ is uniformly continuous on $[0,1]^2$. 
Concerning the second term of the product in \eqref{e:derivX}, using the parabolicity at $a$, we obtain 
$$f(x)-x = \underbrace{f(a)-a}_0 + \underbrace{D (f-\id)(a)}_0 \cdot \, (x-a)+o(x-a) = o (x-a).$$
Therefore, the expression in  \eqref{e:derivX} tends to $0$ as $x$ goes to $a$, as announced. 
This completes the proof of the $C^1$ regularity of $\tilde X\circ \varphi$.

Since $\varphi$ is a $C^1$ diffeomorphism, we deduce that $\tilde X$ itself is $C^1$ 
and that its derivative satisfies $D\tilde X \circ \varphi \cdot D\varphi = D (\tilde{X} \circ \f) =\log Df\cdot D\varphi$. 
Hence $D \tilde{X} = \log Df \circ \varphi^{-1}$, so $\tilde X$ is of class $C^r$ and $\var(D\tilde X)=\var(\log Df)$, 
as announced.
\end{proof}

\vspace{0.1cm}

\begin{rem} 
All along our proof of the Main Theorem \ref{t:C1ac} for the case of the interval, Proposition~\ref{p:flot} is 
only used for the restriction of a flow to the connected components of the complement of the set of 
parabolic global fixed points. On the interior of such a component, every zero of the generating vector field is 
hyperbolic, hence isolated. 
Thus, we could have restricted to this context the statements of Proposition~\ref{p:flot} and its main proof ingredient, 
the Regularization Lemma~\ref{l:lissage} above, and avoided dealing with non-isolated zeroes of the vector field, 
which are the most difficult to treat. However, we have included here the general statements with proof in order to 
show that no ``devil's staircase phenomenon" arises in this setting. 
\end{rem}
\medskip

\begin{rem} 
\label{r:stillholds} 
Part of Lemma \ref{l:lissage} still holds on the interval $[0,1)$. More precisely, if $(f^t)$ is a $C^1$ flow 
of diffeomorphisms of $[0,1)$ whose time-$1$ map $f = f^1$ is of class $C^{r}$, with $r \in \{1,1+bv,1+ac\}$, 
then it is still true that $(f^t)$ is conjugated by a $C^1$ diffeomorphism to the flow of a $C^r$ vector field. 
The proof is the same as the one above when the involved ``normalization constant'' $c$ in~\eqref{e:defh} is 
finite. If it is infinite, we just define $\tilde{\varphi}$ by $\tilde{\varphi} (0) = 0$ and 
$$\log D \tilde{\varphi}(x) := \int_0^{1} \log Df^s(x) ds.$$
Then $\tilde{\varphi}$ is a $C^r$ diffeomorphism from $[0,1)$ to $\mathbb{R}_+ := [0,\infty)$. The desired conjugacy is 
$\varphi := \varphi_* \circ \tilde{\varphi}$, where $\varphi_*$ is any $C^1$ diffeomorphism from $\mathbb{R}_+$ to 
$[0,1)$. 
Details are left to the reader.\\ 
\end{rem}

Before passing to the proof of Proposition \ref{p:flot}, we need to make a final observation. When $r$ is a positive 
integer, it is standard that the flow of a $C^r$ vector field on $[0,1]$ is a $C^r$ flow (that is, a continuous group 
morphism from $(\R,+)$ to $\Diff^r_+([0,1])$). Below, we make sure that this still holds when $r=1+ac$.

\begin{lem}
\label{l:unsurprising}
The flow $(f^t)$ of a $C^{1+ac}$ vector field $X$ on $[0,1]$ is a $C^{1+ac}$ flow whose maps satisfy 
$\var(\log Df^t)\le |t|\cdot \var(DX)$.
\end{lem}

\begin{proof}
Let $X$ be a $C^{1+ac}$ vector field on $[0,1]$. Since $X$ is in particular $C^1$, 
its flow $(f^t)$ is a $C^1$ flow, as already mentioned. Given $t\in\R$, using once 
the equality $D f^t = \frac{X\circ f^t}{X}$ on $[0,1]\setminus X^{-1}(\{0\})$ 
(which shows that $f^t$ is $C^2$ therein), we get
$$D\log Df^t 
= D\log(X\circ f^t)-D\log X 
= (D\log X)\circ f^t \cdot D f^t - D \log X  
= \tfrac{DX\circ f^t\cdot D f^t}{X\circ f^t}-\tfrac{DX}X.$$ 
Using the same inequality again, we obtain
\begin{equation}
\label{e:nec}
D\log D f^t = \frac{DX\circ f^t - DX}{X}.
\end{equation}
Thus, for every connected component $I$ of $[0,1]\setminus X^{-1}(\{0\})$ and all $a$ in $I$,
\begin{small}
\begin{align*}
\| D\log Df^t\|_{L^1(I)}
&= \! \int_I  \left|\tfrac{DX\circ f^t - DX}{X}\right| \\
&= \! \sum_{n=-\infty}^{+\infty}  \int_{(f^t)^{n+1}(a)}^{(f^t)^n(a)}\left|\tfrac{DX\circ f^t - DX}{X}\right|\\
&\le \!  \sum_{n=-\infty}^{+\infty} \max_{[(f^t)^{n+1}(a),(f^t)^n(a)]}\left|DX \! \circ \! f^t - DX\right| \underbrace{\int_{(f^t)^{n+1}(a)}^{(f^t)^n(a)}\left|\tfrac{1}{X}\right|}_{|t|} 
\le |t| \cdot \var(DX;I).
\end{align*}
\end{small}(reversing the orientation of the fundamental intervals of $f^t$ if necessary). This implies that $f^t$ is $C^{1+ac}$ 
and that $\var(\log Df^t)\le |t|\cdot \var(DX)$ as announced. In particular, $f^t$ goes to the identity in the 
$C^{1+ac}$ topology when $t$ goes to $0$, and this implies that the group morphism $t\mapsto f^t$ from $(\R,+)$ to 
$\Diff^{1+ac}_+([0,1])$ is continuous, which completes the proof.
\end{proof}

We are finally in position of proving Proposition \ref{p:flot}. To this end, we will use the Interpolation 
Lemma~\ref{l:chem} together with the Regularization Lemma \ref{l:lissage} and Lemma \ref{l:unsurprising} above 
to reduce to the case of a well-controlled $C^{1+ac}$ flow, which is easy to deform to the trivial one.\\

 \begin{proof}[{Proof of Proposition \ref{p:flot}}] 
By assumption, the generators of the image group of the (nontrivial) action $\rho$ we want to deform are flow maps 
$f^{\tau_1}$, \dots, $f^{\tau_d}$ of a $C^1$ flow $(f^t)$. 
Up to permutation of the generators and time rescaling of the flow, one can assume that $\tau_1=1$ and 
$\tau_i\in[-1,1]$ for every~$i \in [\![1,d]\!]$. 
According to Lemma~\ref{l:lissage}, this flow is $C^1$ conjugated to the flow of a $C^{1+ac}$ vector field $X$ 
satisfying $\var(DX)=\var(\log Df)$ and whose time-$s$ maps $g^s$ are tangent to the identity at the endpoints 
if the original maps are. 
According to Lemma \ref{l:unsurprising}, these maps are $C^{1+ac}$ and, for every $s$ in $[-1,1]$ 
(in particular, for $s=\tau_i$, $1\le i\le d)$,
$$d_{1+ac}^* ( g^{s}, id ) \leq | s | \cdot \var(DX) 
=|s | \cdot \var(\log f) 
= \, | s | \cdot d_{1+ac}^* ( f, id ) 
\leq d_{1+ac}^* ( \rho, id ) .$$ 
Therefore, according to the Interpolation Lemma \ref{l:chem}, it suffices to deal with this new more regular action. 
To this end, we simply consider the continuous path of $C^{1+ac}$ actions defined by 
$$t\in[0,1]\mapsto(g^{t\tau_1},\dots,g^{t\tau_d})$$ 
(continuity follows again from Lemma \ref{l:unsurprising}). By the above estimates, 
$d_{1+ac}^* ( g^{t\tau_i}, id )\le d_{1+ac}^* (\rho,\id)$, 
so this path connects the regularized action to the trivial one satisfying the desired control on the $C^{1+ac}$ distance 
from the generators to the identity. 
\end{proof}

We close this section with a lemma that gives extra information in the context of the Regularization Lemma 
\ref{l:lissage} that will be useful when dealing with $\mathbb{Z}^d$ actions on the circle. 

\begin{lem} 
\label{lem-complement}
Under the hypothesis of the Regularization Lemma \ref{l:lissage}, if $g$ is a $C^r$ diffeomorphism commuting with 
every $f^t$, 
then the conjugate of $g$ by $\varphi$ is still of class $C^r$. 
\end{lem}

\begin{proof} 
If $r=1$, then $\tilde g := \varphi g \varphi^{-1}$ is automatically $C^{1}$. Let us now assume that $r=1+bv$. 
By the chain rule, 
$$\log (D\tilde g)\circ \varphi = \log (Dh)\circ g  + \log Dg - \log D\varphi.$$
By the definition of $\varphi$, this gives
$$\log (D\tilde g)\circ \varphi  = \int_0^1(\log (Df^s)\circ g +\log Dg- \log Df^s ) \, ds.$$
Again, the chain rule applied to the equality $f^s \circ g = g \circ f^s$ gives 
$$\log (D\tilde g)\circ \varphi  = \int_0^1\log (Dg)\circ f^s ds$$
As a consequence, $\var ( \log D\tilde g ) \le \var (\log Dg)$, hence $\tilde{g}$ belongs to $ \Diff^{1+bv}_+([0,1])$.  
Finally, in the case $r=1+ac$,  
derivation under the integral in the above equality shows that $\tilde g$ belongs to $ \Diff^{1+ac}_+ ([0,1])$. 
\end{proof}

\begin{rem} 
\label{rem-unpleasant} 
Both the Regularization Lemma \ref{l:lissage} and the complementary Lemma \ref{lem-complement} above hold 
(with the very same proofs) for actions on the circle. 
We didn't state them in this setting because it is not a natural one, as actions on the circle arising from $C^1$ flows are 
``rare''. (This is contrary to the case of the interval, for which Szekeres' theorem always provides $C^1$ flows on some 
half-closed intervals for $C^2$ actions; see also Proposition \ref{p:noncyclic}.) 
Actually, although we will use variations of them to study connectedness of $C^2$ actions of $\mathbb{Z}^d$ 
on the circle, they will only apply under very restrictive hypothesis (cf. Proposition \ref{t:circleC2}).
\end{rem}

%%%%%%%%%%%%%%%%%%%%%%%%%%%%%%%%%%%%%%%%%%%%%%%%%%%%%%%%%%%%%%%%%%%%%%%

\subsection{Deformation of ``non-flowable'' actions}
\label{ss:rat}

In this section, we give a proof of Proposition \ref{p:cyclic}. Let us start by explaining what is the main difficulty, 
since it is somehow hidden in the argument below. 
Let $f_1,\ldots,f_d$ be the images of the canonical generators for a $C^r$ action of $\mathbb{Z}^d$ on the interval 
with a cyclic image, and let $f_0$ be a generator of the cyclic group $\langle f_1,\dots,f_d\rangle$. 
Since each $f_i$ is an integer power of $f_0$, deforming the action would reduce to just deforming $f_0$.  
However, there is a deep problem, namely, it could happen that $f_1,\dots,f_d$ are $C^r$ close to $\id$ 
but $f_0$ is not. In this case, if we deform $f_0$ and extend the deformation to its powers, the absence of control on 
$f_0$ might lead to a path that goes very far from the trivial action.

In the $C^{1+ac}$ case, we will manage to avoid this problem by first conjugating the action so as to obtain 
the desired control on $f_0$.  
To do this, we will need a crucial tool, namely the \emph{asymptotic variation}, a notion that was introduced in 
\cite{Na23}. (Originally, this was called ``asymptotic distortion'', but the terminology was changed afterwards.) 
For a $C^{1+bv}$ diffeomorphism $f$ of a compact 1-manifold $M$, this is 
$$V_{\infty}(f):=\lim_{n\to+\infty}\frac{\var (\log Df^n)}n.$$
This number is well-defined and less than or equal to $\var (\log Df)$. 
Moreover, it is homogeneous: for every $m \in \mathbb{Z}$,
\begin{equation}\label{eq-hom}
V_\infty(f^m) = |m| \, \dist (f).
\end{equation}
An important result of \cite[Corollary 2]{EN1} establishes that, for an interval diffeomorphism~$f$, 
the value of $V_{\infty} (f)$ vanishes if and only if $f$ has no hyperbolic fixed point and is $C^1$-flowable. 
This result is extended to the $C^{1+bv}$ setting in Appendix \ref{a:varflow} (cf. Corollary \ref{c:varflow2}).\\

\begin{proof}[Proof of Proposition \ref{p:cyclic}.] 
Let $\rho_0$ be the action with a cyclic image we are trying to deform, let $f_1,\dots,f_d$ be the images of the 
canonical generators of $\Z^d$, and let $f_0$ be a generator of the cyclic group $\langle f_1,\dots,f_d\rangle$. 
For every $n \geq 1$, define a diffeomorphism $\varphi_n$ of $[0,1]$ by letting $\varphi_n (0) =0$ and 
\begin{equation}\label{eq:def-g_n}
D\varphi_n (x) := 
\frac{\left[ \prod_{0 \leq n_j < n} D(f_0^{n_0} \cdots f_d^{n_d}) (x) \right]^{\frac{1}{n^{d+1}}}}
{\int_0^1  \left[ \prod_{0 \leq n_j < n} D(f_0^{n_0} \cdots f_d^{n_d}) (y) \right]^{\frac{1}{n^{d+1}}} \, dy}
\end{equation}
(note that the right-side expression above is a positive continuous function of integral equal to $1$, so it is indeed 
the derivative of a $C^1$ diffeomorphism of $[0,1]$ which is actually of class $C^{1+ac}$). 
At first glance, this definition may look strange, but it is very natural in a 
much more general setting: see Appendix \ref{a:Banach}. Compare also (\ref{eq-conj-top}).

Now we compute.
Using commutativity and the chain rule, we obtain 
$$D (\varphi_n \circ f_i \circ \varphi_n^{-1}) (\f_n (x)) 
= \frac{D\varphi_n (f_i (x))}{D\varphi_n (x)} \cdot Df_i (x).$$
This gives
\begin{eqnarray*}
D (\varphi_n \circ f_i \circ \varphi_n^{-1}) (\varphi_n (x)) 
&=& \frac{\left[ \prod_{0 \leq n_j < n} D(f_0^{n_0} \cdots f_d^{n_d}) (f_i(x)) \right]^{\frac{1}{n^{d+1}}}}
{\left[ \prod_{0 \leq n_j < n} D(f_0^{n_0} \cdots f_d^{n_d}) (x) \right]^{\frac{1}{n^{d+1}}}} \cdot Df_i (x) \\
&=& \frac{\left[ \prod_{0 \leq n_j < n} D(f_0^{n_0} \cdots f_d^{n_d}) (f_i(x)) \cdot Df_i (x)\right]^{\frac{1}{n^{d+1}}}}
{\left[ \prod_{0 \leq n_j < n} D(f_0^{n_0} \cdots f_d^{n_d}) (x) \right]^{\frac{1}{n^{d+1}}}} \\
&=& \frac{\left[ \prod_{0 \leq n_j < n} D(f_0^{n_0} \cdots f_i^{1+n_i} \cdots f_d^{n_d}) (x) \right]^{\frac{1}{n^{d+1}}}}
{\left[ \prod_{0 \leq n_j < n} D(f_0^{n_0} \cdots f_i^{n_i} \cdots f_d^{n_d}) (x) \right]^{\frac{1}{n^{d+1}}}} \\
&=& \frac{\left[ \prod_{0 \leq n_j < n; \, j \neq i} D(f_i^{n} f_0^{n_0} \cdots f_{i-1}^{n_{i-1}} f_{i+1}^{n_{i+1}} \cdots f_d^{n_d}) (x) \right]^{\frac{1}{n^{d+1}}}}
{\left[ \prod_{0 \leq n_j < n; \, j \neq i} D(f_0^{n_0} \cdots f_{i-1}^{n_{i-1}} f_{i+1}^{n_{i+1}} \cdots f_d^{n_d}) (x) \right]^{\frac{1}{n^{d+1}}}} \\
\end{eqnarray*}
Thus,
$$D (\varphi_n \circ f_i \circ \varphi_n^{-1}) (\varphi_n (x)) 
=\left[ \prod_{0 \leq n_j < n; j \neq i} D(f_i^{n}) (f_0^{n_0} \cdots f_{i-1}^{n_{i-1}} f_{i+1}^{n_{i+1}} \cdots f_d^{n_d} (x) ) \right]^{\frac{1}{n^{d+1}}}
,$$
and therefore
$$\log \big( D (\varphi_n \circ f_i \circ \varphi_n^{-1}) (\varphi_n (x)) \big) = 
\frac{1}{n^{d+1}} \sum_{\substack{0 \leq n_j < n; \\ j \neq i}} 
\log (Df_i^n) (f_0^{n_0} \cdots f_{i-1}^{n_{i-1}} f_{i+1}^{n_{i+1}} \cdots f_d^{n_d} (x) ).$$
Since the total variation is invariant under coordinate changes, a triangle inequality yields 
\begin{equation}
\label{e:var-conj}
\var \big( \log D (\varphi_n \circ f_i \circ \varphi_n^{-1})  \big) 
\leq \frac{1}{n^{d+1}} \sum_{\substack{0 \leq n_j < n \\ j \neq i}} \var (\log Df_i^n) 
= \frac{\var (\log Df_i^n)}{n}.
\end{equation}
At this point, note that, by Corollary \ref{c:varflow2}, if $\dist(f_0)=0$ then $f_0$ embeds into a $C^1$ flow. 
Since $f_1,\ldots,f_d$ are all powers of $f_0$, Proposition \ref{p:flot} provides in this case the desired path 
connecting the action to the trivial one (see also Remark \ref{r:vanish} below). 
From now on, we thus assume that $\dist(f_0)>0$, which implies that $\dist(f_i)>0$ for every nontrivial $f_i$ 
by the homogeneity of $\dist$ (see~\eqref{eq-hom}).  
In this case, \eqref{e:var-conj} implies that, for $N$ big enough and all $i \in [\![0,d]\!]$, 
$$\var \big( \log D (\varphi_N \circ f_i \circ \varphi_N^{-1})  \big) 
\leq  2 \, V_\infty(f_i).$$
Using the Interpolation Lemma \ref{l:chem}, one can connect the initial action to the one conjugated by $\varphi_N$ 
without increasing the $C^{1+ac}$-distance to the identity along the deformation. 
Let $g_0,\dots, g_d$ be the conjugates of $f_0,\dots, f_d$ by $\varphi_N$, and let $\tilde\rho$ be the new action 
of $\Z^d$ sending the canonical generators to $g_1,\dots,g_d$. 
For each $i \in [\![1,d]\!]$, there exists an integer $m_i$ such that $f_i = f_0^{m_i} $. 
Using the homogeneity of $V_\infty$ (see~\eqref{eq-hom}), for those $i$ for which $m_i \neq 0$ we obtain
\begin{equation}
\label{eq:est-F}
\var (\log Dg_0) \leq 2 \, \dist (f_0) =
\frac{2\, \dist (f_0^{m_i})}{|m_i|} = \frac{2\, \dist(f_i)}{|m_i|}.
\end{equation} 
We let $g_0^t$ be the homotopy from $g_0$ to the identity that is linear along $\log D (\cdot)$.  
More precisely, we let $g_0^t$ be such that $g_0^t (0) := 0$ and 
\begin{equation}
\label{e:otra-def}
Dg_0^t (x) := \frac{ e^{(1-t) \log Dg_0(x)}} {\int_0^1 e^{(1-t) \log Dg_0 (y)} dy}.
\end{equation}
Now consider the path $(\tilde\rho_t)$ from $\tilde\rho$ to the trivial action made of the actions $\tilde\rho_t$ 
that associate to the $i^{th}$ generator of $\mathbb{Z}^d$ the map $(g_0^t)^{m_i}$. 
Since, for a certain constant $c$,  
$$\log Dg_0^t = (1-t) \log Dg_0 + c,$$
we have 
 $\var (\log Dg_0^t) = (1-t) \, \var (\log Dg_0).$ 
Therefore, by (\ref{eq:est-F}),
\begin{eqnarray*}
\var (\log D(g_0^t)^{m_i}) 
&\leq& | m_i | \, \var (\log Dg_0^t) 
\,\,\, = \,\,\, (1-t) \, | m_i | \, \var (\log Dg_0) \\
&\leq& 2 \, (1-t) \, \dist (f_i) 
\,\,\, \leq \,\,\, 2 \, (1-t) \, \var (\log Df_i),
\end{eqnarray*}
which shows the desired estimates for the path.  
\end{proof}

\begin{rem}
\label{r:vanish}
Given any $C^{1+ac}$ action of $\Z^d$ on $[0,1]$ (without the hypothesis of a cyclic image), 
if the asymptotic variation of the images of the generators of $\Z^d$ all vanish, 
then the Main Theorem of \cite{Na23} establishes that the action can be deformed to the trivial 
one through a path of $C^{1+ac}$-conjugated actions. We could have used this result in the case where 
$\dist(f_0)=0$ above. However, we preferred to reduce it to the flowable case using Appendix \ref{a:varflow}. 
\end{rem}

\begin{rem}
\label{rem-specific}
Let us stress that, in the $C^2$ setting, we do not know how to perform a conjugacy trick as the one above 
involving the asymptotic distortion that keeps a good control in $C^2$ topology. 
\end{rem}

%%%%%%%%%%%%%%%%%%%%%%%%%%%%%%%%%%%%%%%%%%%%%%%%%%%%%%%%%%%%%%%%%%%%%%%%%%%%%

\section{Path-connectedness of $C^{1+ac}$ actions on the circle}
\label{s:C1ac-circle}

In this section, we prove our Main Theorem \ref{t:C1ac} for the case of the circle. 
Note that it suffices to prove that any $C^{1+ac}$ action of $\Z^d$ on the circle can be connected to another one 
that is conjugated to an action by rotations, for the latter can be connected to the trivial action just 
by moving the angles. 
So let $\rho$ be such an action, and denote 
$G := \rho(\Z^d) \subset \Diff^{1+ac}_+(\T^1)$. 
The rotation number function induces a group morphism $\rot \!: G \to \T^1$ whose kernel $G_0$ coincides with 
the set of elements having fixed points. 
Moreover, each of these elements fixes all the points in the support of every invariant probability measure; 
in particular, they share a global fixed point $a_0$ (that we identify to $0 \in \mathbb{T}^1$). 
See \cite[Chapter 2]{N1} for all of this.

We proceed differently depending on whether $\rot(G)$ is finite or not. 
Roughly, in the first case, we adjust and extend the results about the interval from the preceding section, 
while in the second case, we slightly simplify the proof of the Main theorem of \cite{Na23}, 
which claims that the action can be connected to an action by rotations through a path of conjugates.

%%%%%%%%%%%%%%%%%%%%%%%%%%%%%%%%%%%%%%%%%%%%%%%%%%%%%%%%%%%%%%%%%%%%%%%%%%%%%

\subsection{Finite rotation number set}
\label{ss:finite}

Here, we assume that $\rot(G)$ is a finite subgroup of $\T^1$, hence of the form $\mathbb{Z} / n \mathbb{Z}$. 
The case where $n=1$ has been essentially settled. 
Indeed, this corresponds to the situation where there is a global fixed point for the action, 
and the reader can readily check that our proof of the Main Theorem \ref{t:C1ac} for the interval works verbatim 
for $C^{1+ac}$ actions of $\Z^d$ on the circle with global fixed points. 
For example, the conjugating map $\varphi$ introduced in \eqref{e:defh} is a genuine circle diffeomorphism 
if the original maps are. (Recall that we are identifying $a_0 \sim 0$.)

If $n > 1$ and $G_0$ is trivial, then $G$ is finite and conjugated to a finite (hence cyclic) group of rotations.  
We assume henceforth that $n>1$ and $G_0$ is nontrivial. In this case, we know how to deform 
the restriction of the action to the subgroup $\rho^{-1}(G_0)=\rho^{-1}(\ker(\rot))$, 
but we need to ensure that this deformation can be extended to the whole action of $\Z^d$, 
and this requires a few adjustments. 
Note that $G$ is generated by $\{ g,f_1,\dots,f_k \}$ for any $g$ satisfying $\rot(g)=\frac1n$ 
and any $\{ f_1,\dots,f_k \}$ generating $G_0$. 

Suppose first that one can find such a $g$ with $g^n=\id$. Then one can conjugate the action so that $g$ becomes the 
rational rotation $R_{\frac1n}$. 
One can then check that the deformation of $(f_1,\dots, f_k)$ to $(\id,\dots,\id)$ provided by the proof of the Main 
Theorem \ref{t:C1ac} for actions on the circle with a global fixed point is $R_{\frac1n}$-equivariant. 
Thus, $(g,f_1,\dots,f_k)$ can be deformed through commuting diffeomorphisms to $(R_{\frac1n},\id,\dots,\id)$, 
which settles this case.

Suppose now, more generally, that there exists a $C^{1+ac}$ circle diffeomorphism $f$ fixing $a_0$ 
that is an $n^{th}$ root of $g^n$.   
We claim that $f$ commutes with all the elements of $G$. 
Assuming this for a while, we let $\tilde g := gf^{-1}$  to obtain that  $\{ \tilde g, f, f_1,\dots,f_k \}$ generates 
an Abelian group of $C^{1+ac}$ circle diffeomorphisms containing $G$ and $\tilde g^n=\id$, so that we are reduced 
to the previous settled case. 
Now, to see that $f$ commutes with $G$, we first check that it commutes with $g$. 
To do this, note that $\Fix (f^n) = \Fix (f)$,  and on each connected component $I$ of 
$\clo\setminus\Fix(g^n)=\clo\setminus\Fix(f)$, 
the $n^{th}$ root of $g^n$ is unique among $C^1$ diffeomorphisms (see Remark \ref{r:folk}). Since $gfg^{-1}$ satisfies 
$(gfg^{-1})^n = gf^ng^{-1} =gg^ng^{-1} = g^n$, this forces the equality $gfg^{-1} = f$ on each such $I$. Moreover, 
the maps $f$ and $g$ trivially commute on $\Fix (g^n) = \Fix (f)$. 
It remains to prove that $f$ commutes with every element $h$ of $G_0$. 
To do this, first note that, if we open the circle at $a_0$, then $g^n$ and $h$ become commuting diffeomorphisms 
of the interval. 
It then follows from Kopell's lemma that every connected component of $\clo\setminus\Fix(g^n)=\clo\setminus\Fix(f)$ 
is fixed by $h$. Now, a folklore result (see again Remark \ref{r:folk}) shows that on the closure of any such component
 $I$, the $C^1$ centralizer of the restriction of $g^n$ is Abelian, so the restrictions of $f$ and $h$ commute there. 
 Finally, since $\Fix(g^n) =  \Fix(f)$ is contained in $\Fix (h)$, the maps $f$ and $h$ trivially commute on $\Fix (g^n)$.

A major problem with the strategy above is that, in general, a diffeomorphism of the interval 
(as is $g^n$ if we open the circle at $a_0$) does not have an $n^{th}$ root, even if it does not fix any other point. 
Actually, Kopell proved in her thesis \cite{Ko} that generic $C^2$ diffeomorphisms have a trivial $C^1$ centralizer 
(i.e. a centralizer reduced to its powers). 
And even in the optimistic situation, elements in the $C^1$ centralizer of a $C^{1+ac}$ diffeomorphism may fail to be 
of class $C^{1+ac}$, as it is the case for the time-1 map of the flows built in Appendix~\ref{a:examples}. 
All that follows is a way to circumvent this misfortune.

%%%%%%%%%%%%%%%%%%%%%%%%%%%%%%%%%%%%
\subsubsection{The case where $G_0$ has a global parabolic fixed point}

Let us first assume that $a_0$ is a parabolic global fixed point for $G_0$. 
Adding an element to the generating set $\{f_1,\dots,f_k\}$ of $G_0$ if necessary, 
we can assume that $g^n$, which belongs to $G_0$, is equal to $f_1$. 
Up to $C^{\infty}$ conjugacy, we can also assume that the orbit of $a_0 \sim 0$ under $g$ is 
$\{0,\frac1n,\dots,\frac{n-1}n\}$. 

We claim that, up to a $C^{1+ac}$ conjugacy now, one can assume that
\begin{equation}
\label{e:R1n}
g = R_{\frac1n} \text{ on } [0,\tfrac{n-1}n]\quad \text{and}\quad g 
= g^n\circ R_{\frac1n}
= f_1\circ R_{\frac1n}\text{ on }[\tfrac{n-1}n,1].
\end{equation} 
Indeed, any smooth diffeomorphism $\psi$ of $[0,\frac1n]$ parabolic at $0$ and satisfying $D\psi(\frac1n) = Dg(0)$  
extends in a unique way to a conjugacy $\f$ with the desired behavior. 
The requirement that $\f^{-1}g\f=R_{\frac1n}$ on $[0,\frac{n-1}n]$, 
or equivalently that $\f = g \f R_{-\frac1n}$ on $[\frac1n,1]$, 
determines $\f$ completely (step by step) on $[\frac1n,1]$ from its restriction $\psi$ to $[0,\frac1n]$ 
(namely, for $k\in[\![0,n-1]\!]$, $\f= g^k\psi R_{-\frac k n}$ on $[\frac k n, \frac{k+1}n]$). 
In particular, if for each $f \in G$ one denotes by $\tilde f$ the conjugate $\f^{-1}f\f$, on $[\frac{n-1} n, 1]$ we have 
$$ \tilde g = \f^{-1} g \f 
= \psi^{-1} g ( g^{n-1}\psi R_{-\frac{n-1} n}) 
= \f^{-1} g^{n}\f R_{\frac1n} 
= (\f^{-1}f_1\f )R_{\frac1n} 
=\tilde f_1 R_{\frac1n}$$
as required.

We still need to check that $\f$ is a $C^{1+ac}$ diffeomorphism. 
To do this, first note that the boundary condition on $\psi$ ensures that the restrictions of $\f$ to $[0,\frac1n]$ 
and $[\frac1n,\frac2n]$ glue in a $C^1$ way at $\frac1n$. 
Indeed, if $D_-$ and $D_+$ denote the left and right derivatives, then 
$$D_+\f(\tfrac1n)=D_+(g \f R_{-\tfrac1n})(\tfrac1n)=D_+g(0) \, D\f(0)=D_{-}\f(\tfrac1n).$$
If $2\le k\le n-1$, near $\frac k n$ one has  $\f = g^{k-1}\f R_{-\frac{k-1}n}$. Since we just checked that $\f$ is $C^1$ 
near $\frac1n = R_{-\frac{k-1}n}(\frac kn)$, this equality shows that $\f$ is also $C^1$ near $\frac k n$. Finally, to check 
differentiability at $a_0 = 0$, we use the initial hypothesis of parabolicity of $g^n$ at~0: 
\begin{align*}
D_-\f(0)
= D_-(g^{n-1}\psi R_{-\tfrac{n-1} n})(\tfrac n n)
&=(Dg^{n-1})(\tfrac1n) \, D_-\psi(\tfrac1n)\\
&=\underbrace{(Dg^n)(0)}_1 \, Dg^{-1}(\tfrac1n) \, D_-\psi(\tfrac1n)
=(Dg(0))^{-1} \, D_-\psi(\tfrac1n)
=1.
\end{align*}
This shows that $\varphi$ is a $C^1$ diffeomorphism, and the absolute continuity of $\log D\f$ is immediate 
from that of $\log Dg$.

\medskip
\vspace{0.1cm}

From now on, we assume that \eqref{e:R1n} is satisfied. We claim that, under this hypothesis, 
all the elements of $G_0$ commute with $R_{\frac1n}$.  
Indeed, every $f\in G_0$ fixes $[0,\frac{n-1}n]$, and $g$ equals $R_{\frac1n}$ therein. Therefore, on this interval, 
$f\circ R_{\frac1n}=f\circ g=g\circ f=R_{\frac1n}\circ f$. 
On $[\frac{n-1}n,1]$ we have
$$f\circ R_{\frac1n}
= f\circ g^{-n}\circ (g^n\circ R_{\frac1n})
= f\circ g^{-n}\circ g = f\circ (g^{n-1})^{-1}
= (g^{n-1})^{-1}\circ f 
= (R_{\frac1n}^{n-1})^{-1}\circ f 
= R_{\frac1n}\circ f,$$
where the second equality follows from the second part of \eqref{e:R1n} and the second-but-last one 
from the first part of \eqref{e:R1n}.

As was explained at the beginning of this section, knowing that all elements of 
$G_0 = \langle f_1, \ldots, f_k \rangle$ commute with $R_{\frac1n}$ easily allows to connect $(f_1,\dots,f_k)$ to 
$(\id,\dots,\id)$ through a continuous path $(f_{1,t},\dots,f_{k,t})$ of commuting $C^{1+ac}$ diffeomorphisms, 
all of which commute with $R_{\frac1n}$. 
We let $g_t$ be the rotation by $\frac1n$ everywhere except on $[\frac{n-1}n,1]$, where it is defined 
as $f_{1,t}\circ R_{\frac1n} $. 
We claim that $(g_t,f_{1,t},\dots,f_{k,t})$ defines a deformation of the initial action to $(R_{\frac1n},\id,\dots,\id)$.
Indeed, 
\begin{itemize}
\item the regularity of $g_t$ for each $t$ follows from the parabolicity of $f_{1,t}$ at the orbit of $0$;  
\item the continuity of the paths of diffeomorphisms is then immediate;
\item the path starts and ends where stated (note that $g_0=g$, according to \eqref{e:R1n});
\item for a given $t$, the $f_{i,t}$'s commute by definition, and $g_t$ commutes with them thanks to the analogue of \eqref{e:R1n} 
with ``$t$'' subscripts. More precisely, on $[0,\frac{n-1}n]$ and $[\frac{n-1}n,1]$, which are both fixed by $f_{i,t}$ 
for every $i$, the map $g_t$ is equal to (the restriction of) a diffeomorphism which commutes with $f_{i,t}$ 
(namely, $R_{\frac1n}$ or $f_{1,t}\circ  R_{\frac1n}$, respectively).
\end{itemize}

\medskip

This concludes the argument in the case where $G_0$ has a global parabolic fixed point.

\medskip

The next two paragraphs deal with the case where all globally fixed points of $G_0$ are hyperbolic. 
A straightforward extension of Proposition \ref{p:noncyclic} to this setting shows that either $G_0$ 
is cyclic or it embeds into a $C^1$ flow, so we analyze these cases separately. Recall that we are 
still assuming that $n > 1$. 

%%%%%%%%%%%%%%%%%%%%%%%%%%%%%%%%%%%%%%%%%%%%%%%%%%%%%%%%%%%%%

\subsubsection{The case where $G_0$ has no global parabolic fixed point and is infinite cyclic}
\label{section-pourri}

In this case, if we denote $f$ a generator of $G_0$, then the group $G$ is generated by $g$ and $f$ (the generator $f$ 
can be avoided if $g^n$ already generates $G_0$). The deformation we aim requires a previous structure result. In the 
statement and the proof below, $\mathbb{Z} / 1 \, \mathbb{Z}$ has to be understood as the trivial group, so that 
$\mathbb{Z} \times \mathbb{Z} / 1 \, \mathbb{Z} \sim \mathbb{Z}$.

\begin{lem} 
\label{lem-structure}
If $n > 1$ and $G_0$ is infinite cyclic, then $G$ is isomorphic to $\mathbb{Z} \times \mathbb{Z} / k \mathbb{Z}$ 
for some $k \ge 1$. 
\end{lem}

Before proving this elementary lemma, let us explain how we deform the action based on it. 
We first conjugate the action so that the generator of the $\Z/k\Z$ factor becomes a rotation $R$ of order $k$, 
and we denote by $h$ the generator of the $\Z$ factor after conjugacy (which is not necessarily in $G_0$, 
cf. Remark \ref{rem-st}).

One can deform $h$ to a rotation along diffeomorphisms commuting with $R$, 
for example by linearly interpolating the logarithmic derivative (compare (\ref{e:otra-def})). 
In concrete terms, if one lets, for each $t \in [0,1]$, 
\begin{equation}
\label{eq-deform}
h_t (x) = h(0) + \frac{\int_0^x (Dh (y))^{1-t} \, dy}{\int_{\mathbb{T}^1} (Dh (y))^{1-t} \, dy},
\end{equation}
one can readily check that, since $h$ commutes with the rotation of order $k$, the same holds for  
$h_t$ for each $t \in [0,1]$. 
(Note that this argument proves that the space of circle diffeomorphisms commuting with a given 
rational rotation is path-connected.) 
Now keeping the order-$k$ rotation untouched and deforming $h$ via $h_t$, we obtain 
a path of actions ending (at $t=1$) at an action by rotations, as desired.

\begin{proof}[Proof of Lemma \ref{lem-structure}] 
Since $G$ is Abelian, generated by two elements and infinite, it has a torsion decomposition with 
at most two factors (cf. Remark \ref{rem-st} below), at least one of them being $\Z$. 
There cannot be two factors equal to $\mathbb{Z}$ since any two elements have positive powers in 
$G_0 = \langle f \rangle$, and hence have nontrivial powers that are equal. The only possibilities are thus 
$\Z $ and $\Z\times \Z/k\Z$ for some $k>1$.
\end{proof}

\medskip

\begin{rem}
\label{rem-st} 
Let us relate the algebraic structure of $G$ to the relation between the initial generators $f$ and $g$. 
First note that the map $\pi:(i,j)\in\Z^2\mapsto g^if^j\in G$ is a surjection, which is not injective since $g^n$ 
belongs to $G_0$ and hence equals a (possibly trivial) power $f^m$ of the fixed generator $f$ of $G_0$. 
The structure of $G$ then depends on the arithmetic nature of the pair $(m,n)$.  
In concrete terms, we claim that $G\simeq \Z\times\Z/k\Z$, with $k=gcd(m,n)$ if $m\neq0$, and $k=n$ if $m=0$.  

Indeed, the kernel $N$ of $\pi$ contains $(n,-m)$. Actually, it is generated by it, since it cannot be of rank $2$ 
(otherwise $G$ would be finite), and no relation between $g$ and $f$ can involve a power of $g$ smaller than $n$ 
(because $\rot(g)=\frac1n$). 
Hence, the structure of $G$ depends on whether $(n,-m)$ can be completed to a basis of $\Z^2$ 
(equivalently, whether $m$ and $n$ are relatively prime) 
or not. In the first case, a change of basis shows that $G\simeq\Z^2/N\simeq(\Z\times\Z)/(\Z\times\{0\})\simeq \Z$. 
In the second case, if $k=gcd(m,n)$ (or $k=n$ if $m=0$), up to a change of basis of $\Z^2$, we have 
$N= \{0\} \times k\Z $. Thus, 
$G\simeq\Z^2/N\simeq \Z \times \Z/k\Z$, as announced.

The description above can be made more concrete. 
For example, when $m$ is a multiple $q\, n$ of $n$ (in particular, when $m=0$), 
in the isomorphism $G\simeq \Z\times\Z/n\Z$, the $\Z/n\Z$ factor is generated by  $\tilde g = gf^{-q}$ 
and the $\Z$ factor by $f$.  
In the totally opposite case, namely when $gcd (m,n) = 1$, letting $i,j$ be integers such that $i m + j n=1$, 
we have that  $\tilde{f} := g^i f^j$ generates $G \sim \mathbb{Z}$. 
Indeed, 
$$\tilde{f}^n = g^{in} f^{jn} = f^{im} f^{jn} = f^{im+jn} = f,$$
hence $f \in \langle \tilde{f} \rangle$, and 
$$\tilde{f}^{m} 
= g^{im} f^{jm} 
= g^{1-jn} f^{jm}
= g\, g^{-jn}   f^{ j m}
= g \, f^{-jm} f^{jm}
=g,$$
hence $g \in \langle \tilde{f} \rangle $. 
\end{rem}

%%%%%%%%%%%%%%%%%%%%%%%%%%%%%%%%%%%%%%%%%%%%%%%%%%%%%%%%%%%%%%%%%%%%%%%%%%%%%%%%%

\subsubsection{The case where $G_0$ has no  global parabolic fixed point and is flowable}

Here we assume that $G_0=\langle f_1,\dots,f_k\rangle$ embeds into a $C^1$ flow $(f^t)$ fixing $a_0$. 
By the Regularization Lemma \ref{l:lissage}, there exists a $C^1$ diffeomorphism $\varphi$ conjugating $(f^t)$ 
to the flow $(\tilde f^t)$ of a $C^{1+ac}$ vector field on $[0,1]$. Moreover, by looking into the proof, one 
readily checks that $\varphi$ perfectly matches at the endpoints, so this is actually a flow on the circle. 
Furthermore, the proof of Lemma \ref{lem-complement} applies to $g$, thus showing that $\varphi$ 
conjugates $g$ to a $C^{1+ac}$ circle diffeomorphism $\tilde g$ commuting with the new flow. By the 
Interpolation Lemma \ref{l:chem}, the initial action can be connected to the new one (conjugated by $\varphi$) 
by a path of $C^{1+ac}$ actions. Now, since $\tilde g^n\in\langle \tilde f_1,\dots,\tilde f_k\rangle$, 
there exists $\tau\in\R$ such that $\tilde g^n = \tilde f^\tau$. But then $\tilde f^{\tau/n}$ is an $n$-th 
root of $\tilde g^n$ fixing $a_0 \sim 0$ and commuting with $G_0$, so the discussion 
of the beginning of Section \ref{ss:finite} allows to conclude.

%%%%%%%%%%%%%%%%%%%%%%%%%%%%%%%%%%%%%%%%%%%%%%%%%%%%%%%%%%%%%%%%%%%%%%%%%%%%%%%%%%

\subsection{Infinite rotation number set}
\label{ss:infinite}

Let us now assume that $\rot(G)$ is an infinite subgroup of $\T^1$, which means that at least one element of 
the image group has an irrational rotation number. 
Then it follows from \cite[Main Theorem]{Na23} that there exists a path of $C^{1+ac}$-conjugated actions 
ending at an action by rotations, and this last action can then be easily deformed to the trivial one. 
In concrete terms, this works as follows.

First, \cite[Theorem 2]{Na23} establishes that $C^{1+ac}$ circle diffeomorphisms of irrational rotation number have 
vanishing \emph{asymptotic variation}
\footnote{Recall that \cite{Na23} also shows that this is false for $C^{1+bv}$ circle diffeomorphisms 
of irrational rotation number. 
This is the second place in the article where we need regularity $C^{1+ac}$ rather than $C^{1+bv}$ for making our proof work.}  
(cf. Section \ref{ss:rat} for the definition).
For obvious reasons, this is also true for finite-order diffeomorphisms. 
Now, it readily follows from Denjoy's theorem that for a $C^{1+bv}$ action of $\Z^d$ on the circle 
whose image has an element with irrational rotation number, elements of rational rotation number are forced to have 
finite order. 
Therefore, in the present situation, all the elements of the image group have vanishing asymptotic variation. 
Fix a system of generators $\{f_1,\dots,f_d\}$. 
As in Section \ref{ss:rat}, we can define a sequence of conjugacies~$\f_n$ by (\ref{eq:def-g_n}), 
and the very same estimate \eqref{e:var-conj} shows that $\var (\log D (\f_n f_i \f_n^{-1}))$ 
converges to $0$ as $n$ goes to infinity for every $i \in [\![1,d]\!]$. 
Since $\rot (\f_n f_i \f_n^{-1})$ remains constant on $n$, this implies that $\f_n f_i \f_n^{-1}$ 
converges in the $C^{1+ac}$ sense to the rotation of angle $\rot(f_i)$. 
Using the Interpolation Lemma \ref{l:chem}, we can upgrade this discrete deformation to a continuous path 
that connects the initial action to an action by rotations, which completes the proof.

\medskip

Let us conclude this section with a somewhat unifying remark. 
The parallel between the strategies for the case of irrational rotation numbers on the circle and for the case of 
cyclic components on the interval is clear, since we used the same conjugacies in order to reduce to a more 
manageable action. 
But a parallel can also be made with the strategy for the flowable components, or rather the noncyclic components, 
as explained below.

On the one hand, let us consider the circle case with elements with irrational rotation number. 
Let $f_1,\dots,f_d$ be the generators of the image group. 
Let us choose a lift $\F_i$ of each $f_i$ to the real line. 
An easy application of Denjoy's theorem shows that the group generated by $\F_1, \ldots, \F_d$ 
embeds into a $C^0$ flow $(\F^t)$, where $\F^1$ is the unit translation and $\F_i$ equals $\F^{\alpha_i}$ for some 
$\alpha_i \in \mathbb{R}$ 
(the value of $\alpha_i$, reduced modulo $\mathbb{Z}$, equals $\rot (f_i)$). 
With this notation, let us again consider the conjugacy $\f_n$ described by (\ref{eq:def-g_n}) in Section~\ref{ss:rat}, 
and let $\Phi_n$ be its lift that satisfies $\Phi_n (0) = 0$. Then we have an equality of the form
\begin{equation}
\label{e:hn}
\log D \Phi_n = \frac1{n^d}\sum_{k\in[\![0,n-1]\!]^d}\log D \F^{\rep{k\cdot\alpha}}-c_n,
\end{equation}
where $c_n\in\R$ is such that $\int_0^1D \Phi_n=1$ and $\{ k\cdot\alpha \}$ denotes the fractional part of 
the scalar product $k \cdot \alpha$ between the $d$-tuples $(k_1,\dots,k_d)$ and $(\alpha_1,\dots,\alpha_d)$. 
Moreover, each $\Phi_n \F_i \Phi_n^{-1}$ converges in the $C^{1+ac}$ topology towards the translation 
by $\alpha_i$, so the limit action on the real line embeds into the flow of the unit vector field.

On the other hand, let us consider a family $f_1,\dots,f_d$ of (nontrivial) commuting diffeomorphisms 
of the interval embedding into a $C^1$ flow $(f^t)$. Such a flow can be 
 taken so that $f_1=f^1$. For each $i \in [\![1,d]\!]$, let $\alpha_i$ be such that $f_i = f^{\alpha_i}$. 
Let $(\f_n)$ be a sequence of conjugacies defined similarly to (\ref{e:hn}), that is, by letting $\f_n (0) := 0$ and  
$$\log D \f_n := \frac1{n^d}\sum_{k\in[\![0,n-1]\!]^d}\log D f^{\rep{k\cdot\alpha}}-c_n,$$
where $c_n\in\R$ is such that $\int_0^1D\f_n=1$.  
It turns out that, if some $\alpha_i\in\R$ is irrational, then these conjugacies $\f_n$
converge in the $C^1$ topology to the conjugacy $\f$ of the proof of the Regularization Lemma \ref{l:lissage}; 
see (\ref{e:defh}). 
Moreover, the conjugates $\f_n f_i \f_n^{-1}$ converge in the $C^{1+ac}$ topology towards the time-$\alpha_i$ 
maps of the flow of the $C^{1+ac}$ vector field $\tilde{X}$ of the same proof. 
Furthermore, these conjugates can be connected by a path using the Interpolation Lemma \ref{l:chem}. 
Note that this argument yields an alternative proof of Proposition \ref{p:flot} in the noncyclic case. Details are 
left to the reader. 

%%%%%%%%%%%%%%%%%%%%%%%%%%%%%%%%%%%%%%%%%%%%%%%%%%%%%%%%%%%%%%%%%%%%%%%%%%%%%

\section{Studying the connectedness in regularity $C^2$ and higher}
\label{s:C2}

In this section, we prove the Main Theorem \ref{t:C2}. We also give a partial analog for actions on the circle, 
and we discuss the situation in higher regularity.

%%%%%%%%%%%%%%%%%%%%%%%%%%%%%%%%%%%%%%%%%%%%%%%%%%%%%%%%%%%%%%%%%%%%%%%%%%%%%

\subsection{Connectedness of $C^{2}$ actions of $\Z^d$ on the interval}
\label{ss:C2-int}

Theorem \ref{t:C2} follows from the next result, which positively answers a question from \cite{EN2}.

\begin{thm}
\label{t:dense}
The path-connected component of the trivial action is dense in the space of all $C^2$ actions of $\Z^d$ on $[0,1]$. 
As a consequence, this space is connected. 
\end{thm}

\begin{proof}
Consider a $C^2$ action $\rho$ of $\Z^d$ on $[0,1]$, and let $P$ denote this time the set of its 2-parabolic 
global fixed points, that is, the set of global fixed points of $\rho$ at which every diffeomorphism 
of the action is $C^2$-tangent to the identity. 
For any subinterval $I$ of $[0,1]$, we abusively denote by $\id$ the trivial action of $\Z^d$ on $I$. 
Given $\eps>0$, we can define a new $C^2$ action $\tilde \rho$ that is $\eps$-close to $\rho$ in the $C^2$ metric 
as follows: 
for every connected component $I$ of $[0,1]\setminus P$, we let 
$\tilde \rho\res{I} := \id$  if $d_2(\rho\res{I},\id)\le \eps$, 
and $\tilde \rho\res{I} := \rho\res{I}$ otherwise. 
Note that there are only finitely many connected components of $[0,1]\setminus P$ where the new action 
$\tilde{\rho}$ is nontrivial. 
We claim that, on these connected components, $\tilde{\rho}$ can be continuously deformed to the trivial action 
preserving the $C^2$-tangency to the identity at the $C^2$-parabolic global fixed points. 
Indeed, if the restricted action has a cyclic image, it suffices to deform its generator 
(no control is needed here since we only deal with a finite number of components).  
If not, according to Proposition \ref{p:noncyclicC2}, this restriction is $C^1$-flowable, and one can apply 
Proposition \ref{p:flotC2} below, which is the analog of Proposition \ref{p:flot} for the $C^2$ setting. 
So $\tilde\rho$ belongs to the path-connected component of the trivial action. 
Since $\varepsilon > 0$ was arbitrary, this shows that $\rho$ belongs to the closure of this path-connected component, as required. 
\end{proof}

\begin{rem} 
\label{r:pcC2}
We would not need to ``crash'' the dynamics on the components where $d_2(\rho,\id)\le \eps$ if we knew how to 
deform a $C^2$ action with cyclic image \emph{while keeping a control on the $C^2$-distance from the canonical 
generators to the identity}  (as we did in the $C^{1+ac}$ case). 
This would lead to the path-connectedness of the space of $C^2$ actions of $\Z^d$ on the interval 
(rather than the density of the path-connected component of the trivial action), 
because Proposition \ref{p:flotC2} below provides a good control for components where the restriction of the action 
is noncyclic.
\end{rem}

\medskip

We are left with proving the following $C^2$ version of Proposition \ref{p:flot}. 
The proof would be much simpler if we were not interested in the control of the size of the deformation, 
which we did not use in the proof of Theorem \ref{t:dense}. 
However, because of Remark \ref{r:pcC2} above, we prefer to include these estimates here. 
It is worth pointing out that in the statements and the proofs below, the notation $O$ is used in regard to 
small values of the entries, not large ones.

\begin{prop}
\label{p:flotC2}
Every $C^{2}$ action $\rho_0$ of $\Z^d$ on $[0,1]$ that embeds into a $C^1$ flow can be 
deformed to the trivial action~$\rho_1=\id$ through a path $(\rho_t)_{t \in [0,1]}$ satisfying 
$\max_{t\in[0,1]}d^*_2(\rho_t,\id)=O(d^*_2(\rho_0,\id))$. 
Moreover, along this path, all diffeomorphisms remain $C^2$-tangent to the identity at an endpoint 
if they originally were.
\end{prop}

This follows from the next two Lemmas \ref{l:chemC2} and \ref{l:lissageC2} much like 
Proposition \ref{p:flot} followed from the Interpolation Lemma \ref{l:chem} and the Regularization Lemma \ref{l:lissage} 
(of which the statements \ref{l:chemC2} and \ref{l:lissageC2} are $C^2$ analogs). 
We will nevertheless give a complete proof at the end of the section for completeness.

\begin{lem}
\label{l:chemC2}
Let $M$ denote either the interval $[0,1]$ or the circle. 
If two $C^2$ actions $\rho_0$ and $\rho_1$ of a finitely generated group $G$ on $M$ are conjugated 
by a $C^1$ diffeomorphism $\f$, 
then they are connected by a continuous path $(\rho_t)_{t\in[0,1]}$ of $C^1$-conjugated $C^2$ actions such that, 
all along the path, 
$$d_2^* (\rho_t,\id) \le \max \{ d_2^* (\rho_0,\id), d_2^* (\rho_1,\id) \} \cdot (1+O(d_1^* (\f,\id))).$$ 
Moreover, in the case of the interval, if the images of all group elements under $\rho_0$ and $\rho_1$ are 
$C^2$-tangent to the identity at an endpoint, so are their images under $\rho_t$ for every $t$. 
\end{lem}

\begin{proof} 
As it was pointed out in Remark \ref{r:higher}, if we proceed as in the proof of the Interpolation Lemma \ref{l:chem},  
we obtain a family $(\rho_t)$ of $C^2$ actions of $G$. Equality~(\ref{e:chem}) then yields
\begin{equation}
\label{e:C2tangency}
D\log Df_{i,t} = \big(  t \, D (\log Dg_i) \circ \f \cdot D\f + (1-t) \, D (\log Df_i) \big) \circ \varphi_t^{-1} \cdot D\f_t^{-1}.
\end{equation}
Since $t \mapsto \f_t$ is continuous in the $C^1$ topology, this easily implies that $t\mapsto f_{i,t}$ is continuous 
in the $C^2$ topology. Moreover, 
$$\|D\log Df_{i,t} \|_\infty \le \max(\|D\log Df_i\|_\infty,\|D\log Dg_i\|_\infty) \cdot 
( 1 + \|D\f - 1\|_{\infty}) \cdot (1+\|D\f_t^{-1}-1\|_\infty).$$
Clearly, $\|D\f - 1\|_{\infty} = O (\| \log D \f \|_{\infty}) = O (d^*_1 (\f, id))$ and 
$$\|D\f_t^{-1}-1\|_\infty 
= O (\| D \f_t - 1 \|_{\infty}) 
= O (\| \log D\f_t \|_{\infty}).$$
Furthermore, the definition 
$$\log D\varphi_t(x) := t\log D\varphi(x) - c_t, \qquad c_t := \log\int_0^1 (D\varphi(y))^tdy,$$ 
gives
$$\|\log D\f_t \|_{\infty} \leq t \|\log D\f \| + |c_t| \leq  2 t \,\|\log D\f \|_{\infty} = O (d_1^*(\f,id)),$$
where the last inequality follows from 
$$ -t\, \| \log D\f \|_\infty
= \log \int_0^1 e^{- t \| \log D \f \|_{\infty}} \, dy
\leq 
c_t 
\leq \log \int_0^1 e^{ t \| \log D \f \|_{\infty}} \,dy
= t \,\|\log D\f\|_\infty.$$
This proves the desired estimate for $d_2^* (\rho_t,id)$. 

Finally, note that (\ref{e:C2tangency}) directly implies that if $f_i$ and $g_i$ are $C^2$-tangent to the identity at an 
endpoint for every $i \in [\![1,d]\!]$, so are the $f_{i,t}$ for every $t \in \R$, as announced.  
\end{proof}

\medskip

\begin{lem}
\label{l:lissageC2}
If $(f^t)$ is a $C^1$ flow of diffeomorphisms of $[0,1]$ whose time-$1$ map $f$ is of class $C^2$, then it is 
conjugated by a $C^1$ diffeomorphism $\f$ to the flow of a $C^{2}$ vector field $\tilde X$ in such a way that
$$d_1^* (\f,\id) = O\bigl( d_1^* ((f^t),\id) \bigr) 
\qquad \text{and}\qquad 
\|D^2\tilde X\|_{\infty}\le  d_2^* (f,\id)  \cdot \big( 1+ O \big( d_1^* ((f^t),\id)\big) \big) .$$ 
Moreover, if $f$ is $C^2$-tangent to the identity at an endpoint, so are the maps of the new flow.  Furthermore, 
if $g$ is a $C^2$ diffeomorphism commuting with every $f^t$, then the conjugate of $g$ by $\f$ is still of class $C^2$.
\end{lem}

\begin{proof}
According to the proof of the Regularization Lemma \ref{l:lissage}, the $C^1$ diffeomorphism $\varphi$ defined by 
$\f(0)=0$ and
\begin{equation*}
\log D\varphi (x) := \int_0^{1} \log Df^s(x) ds - c, 
\quad\text{with} \quad 
c:= \log\left(\int_0^1\exp\left(\int_0^{1} \log Df^s(x) ds\right)dx\right)
\end{equation*}
sends the generating vector field $X$ of $(f^t)$ to the $C^1$ vector field $\tilde X:=\f_*X$ satisfying 
$D\tilde X = \log Df\circ \f^{-1}$. 
The control for $\f$ directly follows from its definition. Moreover, since $\log Df$ and $\f^{-1}$ are $C^1$, 
the latter equality implies 
$$ D^2 \tilde{X} = (D\log Df)\circ \f^{-1} \cdot D\f^{-1}.$$
In particular, 
$$\|D^2\tilde X\|_\infty \le d_2^* (f,\id) \cdot \|D \f^{-1}\|_\infty.$$
which yields the desired control on $\tilde X$ given the control on $\f$. 

Moreover, if $f$ is $C^2$-tangent to the identity at an endpoint $a$, then   
$$D\tilde X (a) = (\log D f \circ \f^{-1}) (a) = \log Df (a) = 0$$
and
$$ D^2 \tilde{X} (a) = ((D\log Df)\circ \f^{-1} \cdot D \f^{-1}) (a) = D \log Df (a) \cdot D \f^{-1} (a) = 0.$$
This shows that the elements of the new flow are all $C^2$-tangent to the identity at $a$. 
Finally, the proof that the conjugate of $g$ as in the statement by $\f$ is of class $C^2$ works verbatim as in 
Lemma~\ref{lem-complement}.
\end{proof}

\begin{rem}
\label{rem-circle-flow-C2} 
As it was the case of the Regularization Lemma \ref{l:lissage} and its complementary Lem\-ma~\ref{lem-complement},  
the preceding statement still holds in the case of actions on the circle, with the very same proof. The hypothesis 
of $C^1$ flowability in this setting is, however, unnatural; compare Remark~\ref{l:lissage}. 
\end{rem}

\medskip

 \begin{proof}[Proof of Proposition \ref{p:flotC2}.] 
We proceed as it was done for Proposition~\ref{p:flot}. We generically denote by $C$ 
some universal positive constant (not depending on $\rho_0$ for small values of 
distances from $\rho_0$ to $id$) that may vary along the proof.

Up to time-rescaling of the flow and permutation of the generators of $\Z^d$, their images under 
$\rho_0$ are maps $f^{\tau_1}$, \dots, $f^{\tau_d}$ of a $C^1$ flow satisfying 
$\tau_1=1$ and $\tau_i\in[-1,1]$ for every~$i$. 
According to Proposition \ref{p:Szek} applied to the restriction of $f^1$ to the closure of each connected component 
of $[0,1]\setminus\Fix(\rho_0) =[0,1]\setminus\Fix(f^1)$, we have 
\begin{equation}
\label{e:logDf}
d_1^* ((f^t),\id) 
= \max_{t\in[-1,1]}\|\log Df^t\|_\infty \le C\,  d_2 (f^1,\id) \le C\,  d_2^* (f^1,\id) \le C \, d_2^* (\rho_0,\id).
\end{equation}
\noindent By Lemma~\ref{l:lissageC2}, this flow is conjugated by a $C^1$ diffeomorphism $\varphi$ satisfying 
\begin{equation}
\label{e:d1h}
d_1^* (\varphi,\id)\,\le C \, d_1^*((f^t),\id) \le C \, d_2^* (\rho_0,\id)
\end{equation}
 to the flow of a $C^{2}$ vector field $X$ satisfying 
\begin{equation}
\label{e:D-2}
\|D^2X\|_{\infty} \, \le d_2^* (f,\id)  \cdot \big( 1+ O \big( d_1^* ((f^t),\id)\big) \big)
\le  \, d_2^* (\rho_0,\id) \cdot \big( 1 + C \, d_2^* (\rho_0,id) \big).
\end{equation} 
Moreover, the time-$s$ maps $g^s$ of this new flow are $C^2$-tangent to the identity at an endpoint 
if the original map $f^1$ is. 
Using the equality $D\log D g^s = \tfrac{DX\circ g^s - DX}{X}$ (see (\ref{e:nec})), we get
\begin{align}
\label{e:DlogDgs}
\|D\log D g^s\|_{\infty} 
=  \left\|\tfrac{DX\circ g^s - DX}{X}\right\|_\infty\le \|D^2X\|_\infty \cdot \left\|\tfrac{g^s - \id}{X}\right\|_\infty.
\end{align}
Now
$$ \frac{g^s(x)-x}{X(x)} 
 = \int_0^s\frac{\frac{d}{dt} g^t (x)}{X(x)} \, dt 
 = \int_0^s\frac{X(g^t(x))}{X(x)} \, dt 
 = \int_0^sD g^t(x) \, dt 
,$$
so
$$ \left\|\tfrac{g^s - \id}{X}\right\|_\infty  
 \le C \, |s| \, \big(1+d_1^* ((g^t),\id))\big) 
 \le C \, |s| \, \big( 1 + 2\, d_1^* (\f,\id)+d_1^* ((f^t),\id) \big) 
\le C \, |s| \, \big( 1 + d_2^* (\rho_0,\id)\big), 
$$
where the second equality follows from the chain rule and the third one from \eqref{e:logDf} and \eqref{e:d1h}.
  
Together with (\ref{e:D-2}) and \eqref{e:DlogDgs}, this shows that the $C^{2}$-distances to $\id$ 
of the new generators $g^{\tau_i}$ are bounded from above by 
$$C  \, d_2^* (\rho_0,\id)\cdot (1+  \, d_2^* (\rho_0,id))^2.$$
Thus, using Lemma \ref{l:chemC2}, one can connect the initial action to the new more regular one 
with the desired control in the $C^2$ topology. Let us finally consider the continuous path of $C^{2}$-actions 
defined by $t\in[0,1]\mapsto(g^{t\tau_1},\dots,g^{t\tau_d})$. 
By the above computations, this path connects the regular action to the trivial one keeping the desired control 
on the $C^{2}$ distances to the identity of the generators. 
Finally, the desired $C^2$-tangency property follows from the construction.
\end{proof}

%%%%%%%%%%%%%%%%%%%%%%%%%%%%%%%%%%%%%%%%%%%%%%%%%%%%%%%%%%%%%%%%%%%%%%%%%%%

\subsection{Connectedness of a subspace of $C^{2}$ actions of $\Z^d$ on the circle}
\label{ss:C2-circle}

As a corollary of the previous results and their proofs, we obtain the following result.

\begin{prop}
\label{t:circleC2} 
The path-connected component of the trivial action is dense in the space of $C^2$ actions of $\Z^d$ 
on the circle which either have a rotation number set of rank less than or equal to~$1$ or embed into a $C^1$ flow.
\end{prop}

\begin{proof} 
Let $\rho$ be a $C^2$ action of $\Z^d$ on the circle and let $G:=\rho(\Z^d)$. 
If $\rot(G)$ is finite, combining the arguments of Section \ref{ss:finite} and Theorem \ref{t:dense}, 
one can show that $\rho$ belongs to the closure of the path-connected component of the trivial action. 
If $\rot(G)$ has rank $1$, then $G$ is generated by some $f$ with irrational rotation number and possibly some 
$g$ of finite order. 
In the last case, $g$ is $C^{2}$ conjugated to the corresponding rotation, and as in Section \ref{section-pourri} 
one can deform $f$ to the identity along diffeomorphisms commuting with $g$ (leaving $g$ unchanged) 
and then deform the pair $(\id,g)$ to the trivial pair. 
This shows that $\rho$ lies in the path-connected component of the trivial action. 
Finally, this is also the case if $\rho$ embeds into a $C^1$ flow, as can be easily established by combining 
Lemmas \ref{l:lissageC2} and \ref{l:chemC2}, 
both of which hold in the circle case (see Remark \ref{rem-circle-flow-C2} for the first one). 
Indeed, the former claims that $\rho$ is $C^1$ conjugated to a $C^2$ flowable action $\tilde \rho$ 
(which lies in the path-connected component of the trivial action), and the latter that $\rho$ and $\tilde \rho$ are 
connected by a path of $C^2$ actions. 
(This combination was the core idea of the proof of Proposition \ref{p:flotC2}, leaving aside the matter of 
controlling the size of the deformation.)
\end{proof}

%%%%%%%%%%%%%%%%%%%%%%%%%%%%%%%%%%%%%%%%%%%%%%%%%%%%%%%%%%%%%%%%%%%%%%%%%%%%
\subsection{A discussion concerning higher regularity on the interval}
\label{ss:Cr}

It is very tempting to try to give a version of Theorem \ref{t:dense} for $C^r$ actions 
when $r > 2$. Certainly, given an action $\rho: \Z^d \to \Diff^r_+([0,1])$ and 
$\eps>0$, we can again crash the dynamics on the components of $[0,1] \setminus \Fix (\rho)$ 
where $d_r(\rho,\id)\le \eps$, and look for a deformation to the trivial action on each remaining component. 
Those where the image group is cyclic can be managed just by deforming the generator. 
However, we face a huge problem when dealing with $C^1$ flowable components where the restriction of the action 
has no cyclic image. 
Indeed, although there might be some $C^r$ version of the Regularization Lemma \ref{l:lissageC2}, 
the work of Sergeraert \cite{Se} implies that the conjugacy $\varphi$ produced to improve the regularity of the flow 
will be no more regular than $C^1$ in general, even if $r=\infty$. 
Then, if we apply the strategy of the Interpolation Lemma \ref{l:chemC2}, we will only get a path of $C^2$ actions 
rather than $C^r$ actions. New ideas are hence needed to deal with these components in higher regularity.

Nevertheless, the tools presented in the previous sections yield the following very partial result. 
For the statement, recall that an {\em $r$-parabolic global fixed point} is a global fixed point at which every element 
is $C^r$-tangent to the identity. 
For the proof, note that, if we remove the explicit estimates on derivatives from their statements, Lemmas 
\ref{l:chemC2} and \ref{l:lissageC2}, and thus Proposition~\ref{p:flotC2}, remain true 
if one replaces $C^1$ and $C^2$ by $C^{r-1}$ and $C^r$, respectively (we leave the verification of this to the reader). 

\begin{prop}
\label{p:partial-Cr}
Given $r\ge2$, the subspace of the space of $C^r$ actions of $\Z^d$ on $[0,1]$ made of the actions which either 
are trivial or have no $r$-parabolic global fixed point is path-connected.
\end{prop}

\begin{proof}
Let $\rho$ be a nontrivial action having no $r$-parabolic global fixed point. 
On the one hand, if it has a cyclic image, it can easily be deformed to the trivial action along actions 
without $r$-parabolic fixed points, just by deforming the generator the same way we did before 
(note that this procedure of linearly interpolating the logarithmic derivative preserves the non-$r$-parabolicity 
condition). 
On the other hand, if it has a noncyclic image, then, by Proposition \ref{p:noncyclicC2}, it embeds into a 
$C^{r-1}$ flow. By the extension of Lemma \ref{l:lissageC2},  $\rho$ is $C^{r-1}$ conjugated to a $C^r$-flowable action. 
By the extension of Lemma \ref{l:chemC2}, $\rho$ can be connected to this new $C^r$-flowable action  
by a path of $C^r$ actions. 
Finally, the $C^r$-flowable action can be deformed to the trivial one just by reparametrization of time. 
One can check that, by construction, the actions along the path thus built have no $r$-parabolic fixed point.
\end{proof}

%%%%%%%%%%%%%%%%%%%%%%%%%%%%%%%%%%%%%%%%%%

\section{Path-connectedness for real-analytic actions}
\label{s:analytic}

In this section, we prove Theorem C. Let $\Diff_+^\omega ([0,1])$ denote the group 
of (orientation-preserving) real-analytic diffeomorphisms of the interval. We see this 
as a subspace of the space of analytic functions $C^\omega([0,1])$, which is endowed with a natural topology described as follows: 
% the topology we put on this space. A classical one \textcolor{blue}{nom? ref?}, 
%which is stronger than the $C^\infty$-topology consists in saying that 
a sequence $(u_n)$ converges to $u$ if there exists some neighborhood (strip) of the interval in the complex plane 
such that the $u_n$ extend to analytic maps that uniformly converge to $u$ on this strip. We list below three basic properties of this topology:

\begin{itemize}

\item For every $f\in \Diff_+^\omega[0,1]$, the path $t\in[0,1]\mapsto (1-t) f + t \, \id$ is continuous and connects $f$ to the identity. In particular, 
$\Diff_+^\omega([0,1])$ is path-connected. 

\item If $(f_n)$ is a sequence in $\Diff_+^\omega ([0,1])$ converging to some $f\in\Diff_+^\omega([0,1])$ and $k$ is an integer number, 
then the sequence  $(f_n^k)$ converges to $f^k$.

\item If for a real-analytic vector field $X$ on $[0,1]$ we denote by  $f_X^t$ its (analytic) flow maps, then $t\mapsto  f_X^t$ is continuous.

\end{itemize}

The first two claims are rather immediate, whereas the third one follows from 
%(the first two claims are straightforward exercises, the third comes from 
standard results on (analytic) ordinary differential equations (namely, continuity of solutions on parameters, noting that $f_X^t = f^1_{tX}$). 
See \cite{Gh93} and references therein for more on this topology. 

Again, we identify real-analytic actions of $\Z^d$ on $[0,1]$ with $d$-tuples of commuting real-analytic diffeomorphisms of $[0,1]$. 
This yields a subspace of $(\Diff_+^\omega ([0,1]))^d$ (endowed with the induced topology). With this terminology, 
Theorem C becomes a corollary of the following fact. 
%which is itself an easy consequence of old classical results:}

\vspace{0.1cm}

\begin{prop}
\label{p:dichotomy}
If $(f_1,\dots,f_d)$ is a $d$-tuple of commuting elements of $\Diff_+^\omega[0,1]$, then it embeds in the flow of a 
real-analytic vector field or in an infinite cyclic subgroup of $\Diff_+^\omega ([0,1])$.
\end{prop}

\vspace{0.1cm}

Again, in the first case, we say that the $d$-tuple is \emph{flowable}, and in the second case, that it is \emph{cyclic}. 
(Note that a $d$-tuple can be both flowable and cyclic.)

\begin{proof}[Proof of Theorem C from Proposition \ref{p:dichotomy}] 
It suffices to show that every $d$-tuple $(f_1,\ldots,f_d)$ of commuting elements in $\Diff_+^\omega ([0,1])$ 
can be connected to the $d$-tuple $(\id,\dots,\id)$ through a continuous path of such $d$-tuples. 
%Let $(f_1,\dots,f_d)$ be such a $d$-tuple.  
The argument splits into two cases, according to Proposition \ref{p:dichotomy}.
%, it is either flowable or cyclic. }

If $(f_1,\ldots,f_d)$ is flowable, then there exists a real-analytic vector field $X$ and real numbers $\tau_1,\dots,\tau_d$ such that 
$f_i=f^{\tau_i}_X$ for every $i\in[\![1,d]\!]$. Then $s\in[0,1]\mapsto (f^{s\tau_1},\dots,f^{s\tau_d})$ gives the desired path. 

If $(f_1,\ldots,f_d)$ is cyclic, then there exist $f\in \Diff_+^\omega ([0,1])$ and integers $k_1,\dots,k_d$ such that $f_i=f^{k_i}$ for every $i\in[\![1,d]\!]$. 
If we let  \, $f_s := (1-s)f+s\, \id$ \,  for every $s\in[0,1]$, then $s\in[0,1]\mapsto(f_s^{k_1},\dots f_s^{k_d})$ gives the desired path.
\end{proof}
 
\begin{proof}[Proof of Proposition \ref{p:dichotomy}] Since the trivial $d$-tuple $(\id,\dots,\id)$ is both flowable and cyclic, we only need to deal with 
a nontrivial one $(f_1,\ldots,f_d)$. Let us assume that it is not cyclic, and let us argue that it is flowable. Up to permuting the indices, we can suppose 
that $f_1\neq \id$. By analyticity, $f_1$ has only finitely many fixed points, all of them fixed by the other $f_i$'s,  
%(por qué añadiste eso? viene tambien de la proposition siguiente
%\marginpar{eso a este nivel clarifica el panorama; corregí más abajo}
and it is not infinitely tangent to the identity 
at these points. By Proposition \ref{p:noncyclicC2}, all the $f_i$'s belong to the flow of a  $C^\infty$ vector field $X$, which is unique if one imposes 
that $f_1$ is its time-$1$ map. We need to show that $X$ is real-analytic.  To do this, note that 
%since the $d$-tuple is assumed to be non cyclic, the analytic centralizer of $f_1$ is non cyclic. In particular, 
about each of the fixed points, the germs of the $f_i$'s 
%and of the elements of its analytic centralizer 
extend to a non cyclic family of commuting holomorphic germs. 
By classical works of \'Ecalle \cite{Ec81} and Voronin \cite{Vo81}, 
at a neighborhood of each fixed point, they arise from the flow of a holomorphic vector field. By the uniqueness above, up to rescaling, 
this vector field must coincide with $X$ in restriction to the intersection of this neighborhood with the interval. Therefore, $X$ is real-analytic 
in a neighborhood of each fixed point, and since it is invariant under $f_1$, 
%they actually form a whole (complex) flow of a germ of holomorphic vector field, which restricts to a germ of analytic vector field on $\R$. 
%By uniqueness of the generating vector field for a germ of diffeomorphism of $\R$ at a fixed point (reference corollaire de Kopell ou autre?), 
%$X$ is analytic near each fixed point of $f_1$, and since it is invariant under $f_1$, it is analytic everywhere, 
it is real-analytic on the whole interval.
\end{proof}

%%%%%%%%%%%%%%%%%%%%%%%%%%%%%%%%%%%%%

\section{Appendix 1: some variations and extensions of classical results by Szekeres and Sergeraert}
\label{ss:Szek}

Let us start by recalling two classical results used throughout the article. 
The former is due to Szekeres \cite{Sz}, with  slight improvements by Sergeraert \cite{Se} and Yoccoz \cite{Yo}. 
The latter is a famous lemma in Kopell's thesis \cite{Ko}, slightly extended (from $C^2$ to $C^{1+bv}$ regularity) 
in \cite{N1} (see also \cite{Na06} for a folklore straightened version in the real-analytic setting). 
For the statements, we say that a homeomorphism $f$ of $[0,1)$ is a \emph{contraction} if 
$f(x)<x$ for every $x\in(0,1)$, and that a flow of homeomorphisms of $[0,1)$ is \emph{contracting} if $f^t$ 
is a contraction of $[0,1)$ for every~$t\ > 0$. 

\medskip

\begin{propsn} 
Every $C^2$ contraction of $[0,1)$ is the time-$1$ map of the flow generated by a $C^1$ vector field.
\end{propsn}

\begin{kop}
Given a $C^{1+bv}$ contraction $f$ of $[0,1)$, the only $g \in \Diff_+^1([0,1))$ commuting with $f$ and 
having an interior fixed point is the identity. 
\end{kop}

\begin{rem}
\label{r:folk}
A folklore consequence of these two statements, that has been used in numerous works, is that the $C^1$ centralizer 
of a $C^2$ diffeomorphism of $[0,1)$ fixing only $0$ is a $C^1$ flow. 
\emph{A fortiori}, the $C^1$ centralizer of a diffeomorphism of $[0,1]$ without interior fixed points is isomorphic to 
a subgroup of $\R$ (and, in particular, is Abelian). 
However, in general, it is much smaller than~$\R$. (Another result of Kopell \cite{Ko} claims that this centralizer is generically 
cyclic; see also \cite{BCW09} and references therein for general versions and variations of this phenomenon in arbitrary 
dimension.) All of this extends to diffeomorphisms of class $C^{1+bv}$ thanks to the variation of Szekeres' theorem below.
\end{rem}

\medskip

We devote the next subsection \ref{a:szekeres}.to a \emph{variation} of Szekeres' theorem in the $C^{1+bv}$ setting, 
which is crucially used in the core of the article. 
We next show in Subsection \ref{a:examples} that the $C^1$ regularity of the flow provided by Szekeres' theorem 
and its variation is ``optimal'' in the sense that, in general, a $C^{1+bv}$ (or even $C^\infty$) contraction does not 
embed into a $C^{1+bv}$ flow.

%%%%%%%%%%%%%%%%%%%%%%%
\subsection{A variation of Szekeres' theorem} 
\label{a:szekeres}

The starting point of the proof of the Main Theorem \ref{t:C1ac} for the interval was the dichotomy between 
cyclic and flowable components (cf. Proposition \ref{p:noncyclic}), for which the following version of 
Szekeres' theorem in the $C^{1+bv}$ setting is a crucial ingredient. 
Here again, the notation $O$ is used in regard to small values of the entries, not large ones. Moreover, given a $C^1$ 
function $u$ on a closed interval (which will be clear from the context), we let 
$$\|u\|_1:= \|u\|_\infty+\|Du\|_\infty.$$

 \begin{prop}[A variation of Szekeres' theorem]
 \label{p:Szek}
Let $f$ be a $C^{1+bv}$ contraction of~$[0,1)$. 
Then $f$ is the time-$1$ map of a unique $C^1$ contracting flow~$(f^t)$. 
Moreover, if $f$ is of class $C^{1+bv}$ on $[0,1]$, then this flow satisfies
$$\|\tfrac{d}{dt}_{|t=0}f^t\|_\infty+\max_{t\in[-1,1]}\|\log Df^t\|_\infty= O \big( \|f-\id\|_{1}+\var (\log Df) \big).$$
\end{prop}

We refer to the previous statement as a ``variation'' rather than an ``extension'' of Szekeres' theorem 
for a subtle reason. 
Though it applies to a wider class of diffeomorphisms, what it provides is \emph{a $C^1$ flow} 
(i.e. a continuous homomorphism from $(\R,+)$ into $\Diff^1_+ ([0,1))$) 
rather than \emph{the flow of a $C^1$ vector field}. 
The latter is slightly stronger, but not so much: 
a result of \cite{Ha} implies that a $C^1$ flow on a one-dimensional manifold is always $C^1$ conjugated 
to the flow of a $C^1$ vector field (see also the Regularization Lemma \ref{l:lissage} and Remark \ref{r:stillholds} 
for a proof of this).  

\medskip

\begin{rem} 
By a classical result from \cite{St} and \cite[Appendice 4]{Yo} (see also \cite{MW}), 
if $f$ is of class $C^r$ and hyperbolic at $0$, with $r > 1$, then the (unique) $C^1$ flow into which it embeds is 
actually $C^r$. 
This is no longer true for $r=1+bv$ or even for $r=1+ac$, as we discuss in Appendix \ref{a:examples}. 
For integers $r \geq 2$, Yoccoz \cite[Appendice 3]{Yo} also proved that, 
for $C^r$ contractions that are parabolic at the origin but not $r$-tangent to the identity, 
the flow is generated by a $C^{r-1}$ vector field, hence it is a $C^{r-1}$ flow. 
In particular, for non-flat parabolic $C^{\infty}$ contractions, the flow is $C^{\infty}$ 
(this latter result was originally established by Takens \cite{Ta}; see \cite{EyNa23} for 
versions in finite regularity of this result). 
For parabolic $C^{\infty}$ contractions that are infinitely tangent to the identity at $0$, 
Sergeraert has shown in \cite{Se} that the flow may fail to be $C^2$. 
Actually, in Appendix \ref{a:examples} we will adapt his construction to obtain a flow in which 
some elements are not even $C^{1+bv}$. 
\end{rem}

\begin{proof}[Proof of Proposition \ref{p:Szek}]
The fact that $f$ is the time-$1$ map of an explicit $C^1$ flow $(f^t)$ on $[0,1)$ 
was almost proved in \cite[Appendix I]{EN2} as a nontrivial generalization of Szekeres' result. 
Strictly speaking, \cite{EN2} only claims that $f^t$ is $C^1$ for every $t$ 
because this is what we were interested in at that time, but all the arguments needed to show that 
$(t,x)\mapsto f^t(x)$ is $C^1$ are there. 
We rewrite the precise adapted statement and its proof (see Proposition \ref{p:Szek2} below), and we add the 
announced $C^1$ estimate for the flow. Finally, note that the uniqueness of the flow follows from Kopell's Lemma.
\end{proof}

\begin{prop}
\label{p:Szek2}
Given a $C^{1+bv}$ contraction $f$ of $[0,1)$,  
let $\Delta(x) := f(x)-x $, and let
$$c_0 (f) := 
\begin{cases}
\frac{\log Df (0)}{Df (0) - 1} \quad \, \mathrm{ if } \quad Df (0) \neq 0,\\
1 \quad \quad \quad  \quad \mathrm{otherwise}.
 \end{cases}$$
For every $n \geq 0$, 
 let $X_n := c_0(f)(f^n)^* (\Delta) = c_0 (f) \tfrac{\Delta \circ f^{n} }{ Df^{n}}$. 
 Then the sequence of vector fields $X_n$ uniformly converges on every compact subset of~$[0,1)$ 
 and its limit $\X$ is the generating vector field of a $C^1$ flow whose time-1 map is $f$. 
 Moreover, for all $a > 0$, 
$$\max_{[0,a]}|\X|+\max_{(t,x)\in[0,1]\times[0,a]}|\log Df^t(x)| = O \Big( \| (f-\id)|_{[0,a]} \|_{1}+\var(\log Df; [0,a]) \Big),$$ 
where the notation $O$ is used with respect to small values of the entries and is uniform in $a$.
 \end{prop}
 
\begin{proof} 
We divide the proof in five steps. First, we prove the convergence of the sequence $(X_n)$ and, 
almost simultaneously, we obtain the announced control for the limit $X$. 
Next, we check that $X$ generates a $C^0$ flow ($C^1$ away from $0$) of which $f$ is the time-$1$ map. 
Then, we prove the announced control on the derivatives of the flow-maps away from $0$. 
Finally, we extend this to $0$ by showing that these derivatives are continuous at $0$.

 \medskip
 
\noindent \underline{Step 1:} 
For every $x\in(0,1)$, let 
  $$\theta(x)
  := \log \left( \frac{ f^2(x)-f(x) }{ Df(x) \, (f(x) - x) }\right).$$ 
 Then, for every $k\in\N$, one has $ \log \frac{ \X_{k+1} }{ \X_k }  = \theta \circ f^{k}$. 
 Note that $\theta$ may be also expressed as  
$$\theta (x) = \log \left(\int_0^1 Df\bigl( x + s (f(x)-x) \bigr)\,ds \right) 
- \log(Df(x)).$$
In particular, $\theta$ extends to $[0,1)$ as a continuous map, with $\theta(0)=0$. 
By the Mean Value Theorem, for each $x\in[0,1)$ there exists  $ y_x \in [f(x),x]$ such that 
$$ \int_0^1 Df \bigl( x + s (f(x)-x) \bigr) \, ds = Df(y_x) .$$
Therefore, given $a \in [0,1)$, for every $x\in [0,a]$ and every $0\le i <  j$, we have
$$\sum_{k=i}^{j-1} \left| \theta \circ f^{k}(x) \right| 
=  \sum_{k=i}^{j-1} |\log Df(y_{f^{k}(x)})-\log Df(f^{k}(x))| 
\le \var(\log Df ; [0,f^{i}(a)])\xrightarrow[i\to+\infty]{}0.$$
As a consequence, $\sum_k \theta\circ f^{k}$ converges absolutely and uniformly on $[0,a]$. 
Denote by $\Sigma$ the limit sum, which is thus a continuous function on $[0,1)$ that satisfies $\Sigma (0) = 0$ and 
\begin{equation}
\label{eq:sum}
\max_{[0,a]}|\Sigma|\le \var(\log Df ; [0,a]),
\end{equation} 
according to the above computation. Then $(\X_n)$ converges uniformly on every compact subset of $[0,1)$ 
towards $\X:= c_0(f)\Delta \, e^\Sigma$. 
In particular, $\X$ vanishes only at $0$, and is strictly negative at all points of $(0,1)$. \medskip

\noindent\underline{Step 2:} 
Denote by $c(\lambda)$ the number $\frac{\log\lambda}{\lambda-1}$ if $\lambda\in(0,1)$ and $1$ if 
$\lambda=1$, so that $c$ is continuous at $1$. From the equality $\X = c_0(f) \Delta \, e^\Sigma$, one gets 
$$\max_{[0,a]}|\X|\le | c \, (Df(0)) | \,\cdot \| (f-\id)|_{[0,a]}\|_{\infty}\cdot e^{\var(\log Df,[0,a])}.$$
This gives the desired type of control for $\X$.

\medskip

\noindent\underline{Step 3:} 
Since, by definition, $\X_{n+1}=f^*\X_n$, in the limit we obtain $\X=f^*\X$, 
 that is, $\X=\frac{\X\circ f}{Df}$. This implies that the derivative of $x\mapsto \int_x^{f(x)}\frac{du}{\X(u)}$ is identically 
 $0$ on $(0,1)$, so this map is constant. The proof that this constant equals 1 
 in the ``usual'' Szekeres' theorem (i.e. when $f$ is assumed~$C^2$) reproduced in
 \cite[Proposition 4.1.14]{N1} works without any change in the present setting. 
 
Let us now show that $\X$ is the generating vector field of a $C^1$ flow with the desired $C^1$ control. Let
\begin{equation}\label{eq:P}
\tau_{\X} = \tau \!: x \in (0,1) \mapsto \int_1^x\frac{du}{\X(u)}.
\end{equation}
Then $\tau$ is of class $C^1$, with positive derivative, and $\tau (f^n(1))=n$ for every $n\in \mathbb{Z}$. 
Therefore, $\tau$ defines a $C^1$ diffeomorphism between $(0,1)$ and $\R$, which satisfies $f = \tau^{-1} T_1 \tau$, 
where $T_t$ denotes the translation by $t$ on $\R$. 
For every $t\in\R$, define $f^t$ as $\tau^{-1}T_t \tau$ on $(0,1)$ and $f^t(0)=0$ (so that, in particular, $f^1=f$). 
Then $(t,x)\in\R\times [0,1) \mapsto  f^t(x)$ 
defines a $C^0$ flow and is differentiable in the $t$ variable everywhere, with $\frac{d}{dt}f^t(x)=X(f^t(x))$. 
Moreover, this map restricted to $\R\times(0,1)$ is $C^1$. We need to check that it is actually $C^1$ on 
$\R\times[0,1)$ and prove the announced control on the derivatives. \medskip

\noindent\underline{Step 4:} 
Let us prove the desired control on $[0,1]\times (0,1)$. (The case of $(t_0,0)$ with $t_0 \in [0,1]$ will follow from this 
once we prove the continuity of $(t,x) \mapsto Df^t(x)$ at $x=0$ below.) For every $x>0$,
\begin{equation}
\label{e:C1control}
\log Df^t(x) 
= \log \left( \frac{\X(f^t(x))}{\X(x)} \right) 
= \log \left( \frac{\Delta(f^t(x))}{\Delta(x)} \right) + \Sigma(f^t(x)) - \Sigma(x).
\end{equation}
By (\ref{eq:sum}), for every $a\in[0,1)$ and all $(t,x)\in[0,1]\times [0,a]$,
\begin{equation}
\label{eq:control2}
\big| \Sigma(f^t(x)) - \Sigma(x) \big| \le 2 \max_{[0,a]} |\Sigma|\le 2\,\var(\log Df,[0,a]).
\end{equation}
Moreover, given $(t,x)\in[0,1]\times(0,a]$, there exists $z_x \in [f^t(x),x]$ such that
\begin{equation*}
\left|\frac{\Delta(f^t(x))}{\Delta(x)}-1\right| 
= \frac{|D\Delta(z_x)|\times|f^t(x)-x|} {|\Delta(x)|}
\le  \max_{z \in [f^t(x),x]}|D\Delta (z)|
\leq \| (f-\id)|_{[0,a]}\|_{1}.
\end{equation*}
This implies
\begin{equation}
\label{eq:control3}
\max_{(t,x)\in[0,1]\times(0,a]}\left| \log \left( \frac{\Delta(f^t(x))}{\Delta(x)} \right)\right|=O(\| (f-\id)|_{[0,a]}\|_{1}).
\end{equation}
Putting (\ref{e:C1control}), (\ref{eq:control2}) and (\ref{eq:control3}) together, we 
get the announced control on $\max_{(t,x)\in[0,1]\times(0,a]}|\log Df^t(x)|$.

\medskip

\noindent\underline{Step 5:} 
Finally, we need to prove that $\log Df^t(x)$ has a limit, namely $t_0\log Df(0)$, when $(t,x)$ 
goes to $(t_0,0)$, for every $t_0\in[0,1]$. 
Since $\Sigma$ is continuous at $0$ and vanishes at this point, 
in view of~\eqref{e:C1control}, this is equivalent to proving that 
$$\lim_{x\to0\atop{t\to t_0}} \log \left( \frac{\Delta(f^t(x))}{\Delta(x)} \right) = t_0\log Df(0).$$
If $Df(0)=1$, then the computation above gives 
\begin{align*}
\left|\frac{\Delta(f^t(x))}{\Delta(x)}-1\right|
& \le \max_{z \in [f^t(x),x]}|D\Delta (z)|\xrightarrow[x\to0\atop{t\to t_0}]{}|D\Delta(0)|=0.
\end{align*}
If $\lambda_0 := Df(0)<1$, then $D\Delta(0)=\lambda_0 - 1<0$, so 
$\frac{\Delta(y)}{y}\xrightarrow[y\to0]{}\lambda_0 - 1\neq0$. 
This implies $\frac{\Delta(f^t(x))}{\Delta(x)}\sim_{x\to0} \frac{f^t(x)}{x}$ uniformly in $t\in[-1,2]$.  
Thus, we need to prove that $\log ( \frac{f^t(x)}{x} ) \xrightarrow[x\to0\atop{t\to t_0}]{}t_0\log \lambda_0$. 
To do this, first observe that 
$$\X(x)=\frac{\log\lambda_0}{\lambda_0 - 1} \, \Delta(x) \, e^{\Sigma(x)}\underset{x\to0}{\sim}(\log\lambda_0) \, x.$$ 
Thus, given $\eps>0$, we may let $\delta>0$ be such that 
$$\frac{1-\eps}{(\log\lambda_0)\, u}<\frac1{\X(u)}<\frac{1+\eps}{(\log\lambda_0)\, u}$$
for every $u\in(0,\delta]$. Assuming that $f^{-1}(x)$ (and thus $f^{t}(x)$ for every $t\ge-1$) is in this interval, we obtain
$$\int_x^{f^t(x)}\frac{1-\eps}{(\log\lambda_0) \, u} \, du
\le \int_x^{f^t(x)}\frac{du}{\X(u)}\le  \int_x^{f^t(x)}\frac{1+\eps}{(\log\lambda_0) \, u} \, du.$$ 
Hence,
$$\frac{1-\eps}{\log\lambda_0}\log \left( \frac{f^t(x)}{x} \right) 
\le t \le \frac{1+\eps}{\log\lambda_0}\log \left(\frac{f^t(x)}{x} \right),$$
and therefore
$$\frac{t\log\lambda_0}{1+\eps}\le \log \left( \frac{f^t(x)}{x} \right) \le \frac{t\log\lambda_0}{1-\eps}.$$
Letting $\varepsilon \to 0$, this gives the desired limit for $\log (\frac{f^t(x)}{x})$, and thus completes the proof.
\end{proof}

%%%%%%%%%%%%%%%%%%%%%%%%%%%%%%%%%%%%%%%%%%%%%%%%%%%%%%
\subsection{Some contractions that do not embed into a $C^{1+bv}$ flow}
\label{a:examples}

Here, we give two examples of $C^{1+ac}$ contractions that do not embed into a $C^{1+bv}$ flow. 
The first one is hyperbolic at $0$ (thus necessarily non $C^2$, because of Sternberg's linearization result) 
and, actually, $C^1$ conjugate to a linear map. The second one is 
$C^\infty$ and infinitely tangent to the identity at $0$. For simplicity, we prefer to directly 
work on the half-line $\R_+$, which is diffeomorphic to $[0,1)$.

\medskip

Let us start with the hyperbolic example. 
It is worth pointing out that the phenomenon of failure of Sternberg's linearization result in regularity $C^{1+ac}$ 
we are about to describe is further and deeply investigated in the separate publication \cite{EN4}. 

Consider a sequence $(\psi_n)$ of $C^2$ diffeomorphisms of $[1,2]$ 
with pairwise disjoint supports lying inside $[\sqrt{2},2]$ 
and satisfying $\var (\log D\psi_n )\sim \tfrac1{n^2}$. 
This condition implies $\|\log D\psi_n\|_\infty = O(\frac1{n^2})$. 
For every $n \geq 1$, let $\phi_n := \prod_{k=n}^{+\infty}\psi_k$ 
(this is well-defined since the $\psi_k$ are disjointly supported). Then
 \begin{equation}
 \label{e:rig}
 \|\log D\phi_n\|_\infty=O \left(\frac{1}{n^2} \right) \to 0
 \qquad\text{and} \qquad
 \var (\log D\phi_n) \sim \sum_{k=n}^{+\infty}\frac1{k^2}\sim \frac1n.
 \end{equation} 
 Let $f^t$ be the homothety of ratio~$2^{-t}$, and $f=f^1$. Define $\f$ (the future $C^1$ conjugacy)  
 by letting 
 $$\f \res{f^n([1,2])} := f\circ (f^n \phi_n f^{-n})$$
 for all $n \geq 1$, and $\f := f$ elsewhere. On the one hand, 
 according to the left-side estimate of \eqref{e:rig}, this defines a $C^1$ diffeomorphism of $\R_+$. 
 On the other hand, by the right-side estimate of \eqref{e:rig}, for all $n\geq 1$ one has 
 $$\var(\log D\f,f^n([1,2]))\sim\frac1n;$$ 
since $\sum\frac1n$ diverges, this implies that $\f$ is not $C^{1+bv}$ on any neighborhood of the origin.   
 Let us now check that $g^t = \f^{-1} f^t \f$ is of class $C^{1+ac}$ for $t=1$ but not for $t=\frac12$. 
 By construction, for all $n\geq 1$, 
 $$g\res{f^n([1,2])} = f\circ f^n\psi_nf^{-n}.$$ 
 Therefore, $g$ is $C^2$ on $(0,+\infty)$ and 
 $$\var(\log Dg, f^n([1,2]))= \var (\log D\psi_n) \sim \tfrac1{n^2},$$ 
 which shows that $\var (\log Dg)$ is finite, and hence that $g$ is $C^{1+ac}$. 
Finally, for all $n \geq 1$, one can readily check that
 $$g^{\frac12}\res{f^n([1,2])} = f^{\frac12}\circ f^n\phi_nf^{-n}.$$
 Thus, 
 $$\var ( \log Dg^{\frac12}, f^n([1,2]))= \var (\log D\phi_n),$$
 which is not summable. Therefore, $\var (\log Dg^{\frac12})$ is infinite, 
thus concluding the construction.

\begin{rem}
Let us stress that \emph{no} conjugacy between $f$ and $g$ (even different from $\f$)  can be $C^{1+bv}$. 
Indeed, from Kopell's Lemma, one easily deduces that any $C^1$ conjugacy from $f$ to $g$ 
conjugates their flows, but $g^{\frac12}$ is not $C^{1+bv}$ though $f^{\frac12}$ is.
\end{rem}

\bigskip

Let us now move on to the smooth ``infinitely parabolic'' example.

\begin{prop}
\label{p:Sergeraert}
There exists a $C^1$ contracting vector field on $[0,+\infty)$ whose time-$1$ map is $C^\infty$ but whose 
time-$\frac12$ map is not $C^{1+bv}$ on $[0,\eps]$ for any $\eps>0$.
\end{prop}

\begin{proof} 
The following construction is mostly borrowed from \cite{Se}, which provides an example of a smooth contraction 
(necessarily embedded into a $C^1$ flow according to Szekeres' theorem) 
that does not admit a $C^2$ square root. 
One could probably show that the very same example does not admit a $C^{1+bv}$ square root either. 
However, following a clever idea of Bonatti, we prefer to slightly modify the construction in order to simplify the 
computations. 
For an explanation of the idea behind the construction, see \cite{Ey19}.

Let us start with a $C^\infty$ vector field $X_0$ on $[0,+\infty)$ with the following shape:
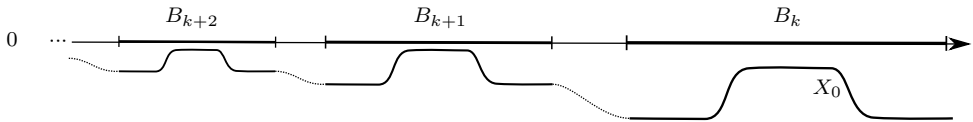
\begin{figure}[htbp]
\centering
\includegraphics[width=12cm]{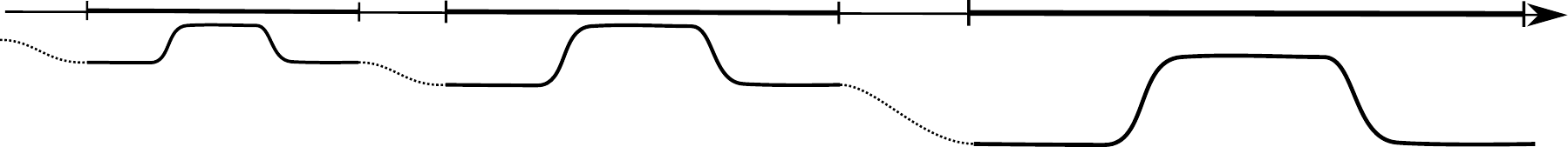}
\put(-60,10){$\scriptstyle{X_0}$}
\put(-75,37){$\scriptstyle{B_{k}}$}
\put(-211,37){$\scriptstyle{B_{k+1}}$}
\put(-305,37){$\scriptstyle{B_{k+2}}$}
\put(-365,28){$\scriptstyle{0\quad\cdots}$}
\caption{The initial vector field $X_0$.}
\label{f:briques}
\end{figure}

More precisely, $X_0$ is made of ``bricks'' of the form described in Figure \ref{f:brique} 
defined on pairwise disjoint segments $B_k$, $k \geq 1$, which get smaller and closer to $0$ as $k$ goes to infinity. 
These bricks are joined together by some simple interpolation. 
For example, $B_k$ and the interpolation zone between $B_k$ and $B_{k+1}$ each have a width equal 
to $2^{-k-1}$, and $X_0$ is constant equal to $-2^{-k^2}$ on the left and right sides of $B_k$ 
(that we denote by $L_k$ and $L_k'$ respectively), 
and to $-2^{-k^3}$ in the central part $C_k$. 
We take the intervals $L_k$, $L'_k$ and $C_k$ and the ``interpolation regions'' between them of 
width $\frac{2^{-k-1}}{5}$. 
One can check that the smoothness of $X_0$ follows from the fact that any power of $2^{-k^2}$ and $2^{-k^3}$ is 
negligible with respect to $2^{-k}$. 

\begin{figure}[h!]
\centering
\includegraphics[width=12cm]{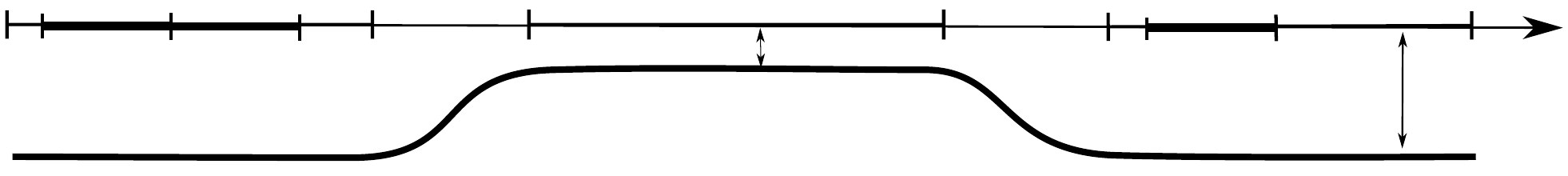}
\put(-170,25){$\scriptscriptstyle{2^{-k^{_3}}}$}
\put(-100,35){$\overbrace{\hspace{2.7cm}}^{L'_k}$}
\put(-339,35){$\overbrace{\hspace{2.7cm}}^{L_k}$}
\put(-184,37){$\scriptstyle{C_k}$}
\put(-30,13){$\scriptstyle{2^{-k^2}}$}
\put(-335,-7){$\scriptstyle{X_0}$}
\put(-323,24){$\scriptstyle{J_k}$}
\put(-293,24){$\scriptstyle{I_k}$}
\put(-80,24){$\scriptstyle{I'_k}$}
\caption{One brick.} 
\label{f:brique}
\end{figure}
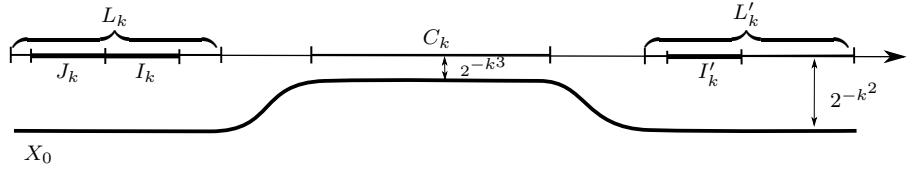

Let $f_0^t$ denote the time-$t$ map of $X_0$, and $f_0:=f_0^1$. We will perturb $f_0$ to a new 
$C^\infty$ contraction~$f$, and the desired vector field $X$ will be the generating vector field of $f$. 
More precisely, $f$ will coincide with $f_0$ except on two fundamental intervals 
of $f_0$ inside each $B_k$, located in the ``low regions'' $L_k$ and $L_k'$. 
Let us fix an orbit $(a_j)$ of $f_0$. One can check that, for $k\ge3$, the widths of the intervals $L_k$, $L'_k$ and 
$C_k$ are at least three times the (constant) size of $X_0$ there. 
This easily implies that they each contain three consecutive $a_j$'s. 
For all $k\ge3$, we denote by $J_k$ the leftmost fundamental interval of the form $[a_{j(k)},a_{j(k)-1}]$ lying in $L_k$, 
and we let $I_k:=f_0^{-1}(J_k)$ its neighbor, which also lies in $L_k$. 
Both $I_k$ and $J_k$ have length $2^{-k^2}$, and 
one easily checks that the restriction of $f_0$ to $I_k$ is the translation by $-2^{-k^2}$. 
Furthermore, there exist $m=m_k$ and $l=l_k$ such that $f_0^{-m}(I_k)\subset C_k$ 
and $f_0^{-l}(J_k\cup I_k)\subset L'_k$. 
The perturbation on the $k^{\mathrm{th}}$ brick will take place in $I_k$ and $I'_k:=f_0^{-l}(I_k)$. 
To define it, we need a smooth ``model'' function $\delta$ like this:
\begin{figure}[h!]
\centering
\includegraphics[width=6cm]{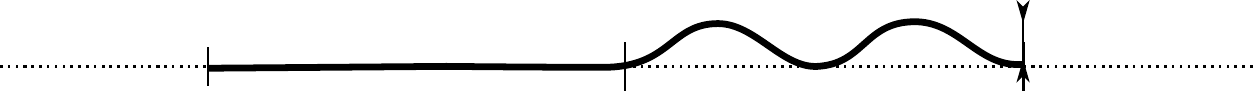}
\put(-89,-8){$\scriptstyle{\frac12}$}
\put(-142,-5){$\scriptstyle{0}$}
\put(-31,-7){$\scriptstyle{1}$}
\caption{The function $\delta$.}
\label{fig:gammak}
\end{figure}

\noindent (The specific form of $\delta$ on $[\frac12,1]$ does not matter here, whereas it did in Sergeraert's 
original argument. 
What does matter is that $\delta$ vanishes on $[0,\frac12]$ and not on $[\frac12,1]$.) We rescale it horizontally and 
vertically to obtain a function $\delta_k$ supported in $I_k = [ p_k,q_k ]$: 
$$\delta_k(x) := 2^{-k^3} \delta(2^{k^2}(x-p_k)).$$
Finally, we define: 
\begin{itemize}
\item a diffeomorphism $\f_k$ of $I_k$ by $\f_k:=\id+\delta_k$, 
\item a diffeomorphism $\psi_k$ of $I'_k$ by $\psi_k := (f_0^{l})^{-1}\circ \f_k^{-1}\circ f_0^{l}$, 
\item a contraction $f$ by
$$ f:=
\begin{cases}
f_0\circ \f_k & \text{on $I_k$ for every $k\ge3$,}\\
f_0\circ \psi_k& \text{on $I'_k$ for every $k\ge3$,}\\
f_0 & \text{elsewhere.}
\end{cases}
$$
\end{itemize}
Note that, on $I'_k$, 
one has $Df_0^l = \frac{X_0\circ f_0^l}{X_0}=\frac{-2^{-k^2}}{-2^{-k^2}}=1$. 
Thus, $f_0^l$ restricted to $I'_k$ is a translation that conjugates $\psi_k$ to $\f_k$. 
As a consequence, for every $r\in\N$, 
$$\|f_k-f_0\|_r=\max(\|\f_k-\id\|_r,\|\f_k^{-1}-\id\|_r)=O(\|\f_k-\id\|_r)=O(2^{-k^3+rk^2}\|\de\|_r),$$
which implies that $f$ is still $C^\infty$. Note also that, by construction, $(a_j)$ is still an orbit of $f$.

We now need to understand the time-$\frac12$ map $f^{\frac12}$ of its generating vector field, 
especially on the regions $C_k$. One way to do this would be to show that $f$ is conjugated to $f_0$ by some explicit 
$C^1$ diffeomorphism (and thus the same holds for their generating flows), but we choose to give a more elementary 
(though less enlightening) argument that we present as a series of claims. 
In the end, just note that the conclusion of Claim 5 below directly implies that $f^{\frac{1}{2}}$ is not $C^{1+bv}$ 
on any interval $[0,\varepsilon]$ with $\varepsilon > 0$, as announced. 

\medskip

\noindent{\underline{Claim 1:}} The generating vector field $X$ of $f$ coincides with $X_0$ on each $J_k$. 

\medskip

This comes from the fact that the two perturbations applied to $f_0$ on each $B_k$ are designed to 
``compensate each other'' in the sense that, 
if $f_0^{n_k}$ is the iterate of $f_0$ sending $J_{k}$ to $J_{k+1}$, then $f^{n_k}=f_0^{n_k}$ on $J_{k}$ 
(while if some $f^n$ sends $J_{k}$ somewhere between $I_{k+1}$ and $I'_{k+1}$, then $f^n\neq f_0^n$ on~$J_k$). 
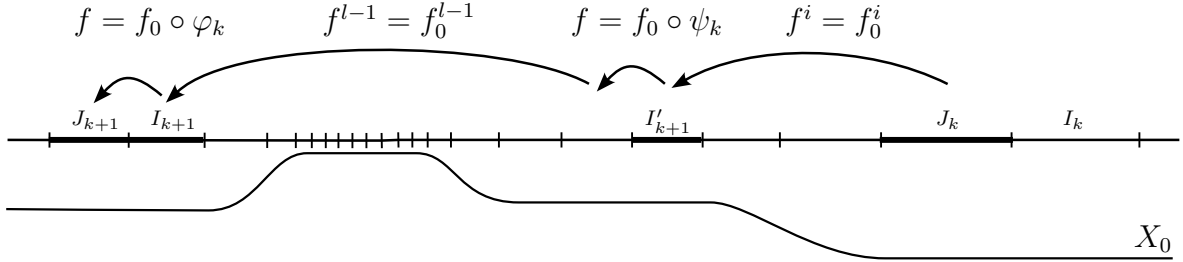
\begin{figure}[h!]
\centering
\includegraphics[width=16cm]{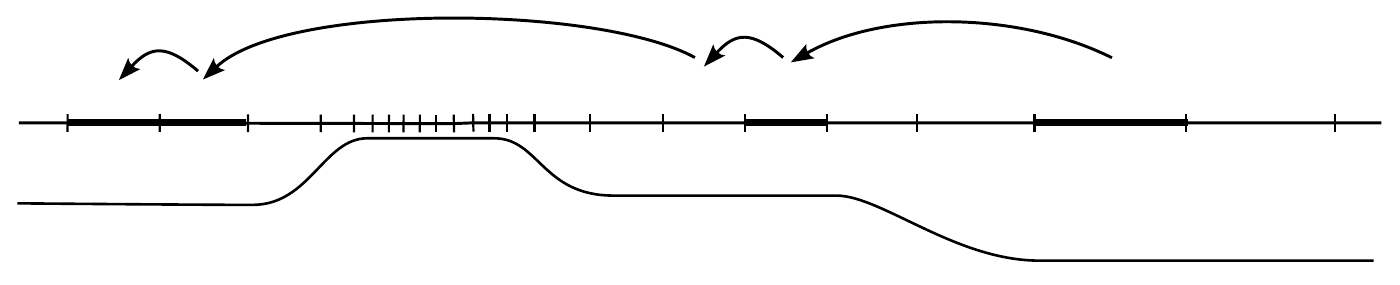}
\put(-330,94){${f^{l-1}=f_0^{l-1}}$}
\put(-155,94){${f^i=f_0^i}$}
\put(-425,58){$\scriptstyle{J_{k+1}}$}
\put(-395,58){$\scriptstyle{I_{k+1}}$}
\put(-208,58){$\scriptstyle{I_{k+1}'}$}
\put(-98,58){$\scriptstyle{J_k}$}
\put(-50,58){$\scriptstyle{I_k}$}
\put(-236,94){$f=f_0\circ \psi_k$}
\put(-424,94){$f=f_0\circ \f_k$}
\put(-23,12){$X_0$}
\caption{Checking that $f^{n_k}= f_0^{n_k}$ from $J_k$ to $J_{k+1}$.}
\label{f:fnk}
\end{figure}

\noindent Indeed, if $l$ and $i$ are positive integers such that $f_0^{i}(J_k)=I'_{k+1}$ and 
$f_0^{l}(I'_{k+1})=I_{k+1}$, then, since $f_0^1(I_{k+1})=J_{k+1}$, one has $n_k=1+l+i$ and 
\begin{align*}
f^{n_k}\res{J_k}
&=\left(f\res{I_{k+1}}\right)\circ \left(f^{l-1}\res{f(I'_{k+1})}\right)\circ \left(f\res{I'_{k+1}}\right)\circ \left(f^i\res{J_k}\right)\\
&=\left(f_0\circ \f_{k+1}\right)\circ f_0^{l-1}\circ \bigl(f_0\circ\underbrace{\psi_{k+1}}_{f_0^{-l}\circ\f_{k+1}^{-1}\circ f_0^{l}}\bigr)\circ f_0^i\\
&=f_0\circ f_0^l\circ f_0^i\\ 
&= f_0^{n_k}
\end{align*} 
(cf. Figure \ref{f:fnk}, where the vertical dashes represent the common orbit $(a_j)$ of $f$ and $f_0$). 
Now let $\Delta_0(x):=(f_0(x)-x)\partial_x$ and $\Delta(x):=(f(x)-x)\partial_x$. Note that these coincide on every $J_k$. 
On such an interval, according to Proposition \ref{p:Szek2}, the generating vector field $X$ of $f$ satisfies
$$ X = \lim_{n\to+\infty}(f^n)^*\Delta = \lim_{K\to+\infty}(f^{n_K+\dots+n_k})^*\Delta
= \lim_{K\to+\infty}(f_0^{n_K+\dots+n_k})^*\Delta_0=X_0,$$
as announced.

For the next claims, given an interval $K$, we denote $K^+$ (resp. $K^-$) its half subinterval to the right (resp. left). 
Also, given $t \in \mathbb{R}$, we let $T_t$ denote the translation of amplitude $t$. 

\medskip

\noindent{\underline{Claim 2:}} The restriction of $f^{\frac{1}{2}}$ to $J_k^+$ coincides with that of the translation $T_{t_k}$ of amplitude 
$t_k := -2^{-k^2-1}$.

\medskip

Indeed, by the previous claim, it we let $g^t(x)=x-t\cdot 2^{-k^2}$ then, for every  $x\in J_k^+$ and $t\in[0,\frac12]$, 
the point $g^t(x)$ belongs to $J_k$, and $\frac{d}{dt}g^t(x)=-2^{-k^2}=X(g^t(x))$. 
By definition of the flow of $X$, this implies that $g^t$ coincides with 
$f^t$ on $J_k^+$ for $t\in[0,\frac12]$. Taking $t=\frac12$, this proves the claim.

\medskip

\noindent{\underline{Claim 3:}} The restriction of $f^{\frac{1}{2}}$ to $I_k^+$ coincides with that of $T_{t_k} \circ \f_k$.

\medskip

Indeed, given that the restriction of $f$ to $I_k$ equals $f_0\circ \f_k$, that $\f_k$ fixes $I_k^+$, that 
$f_0$ sends $I_k^+$ to $J_k^+$, and that, according to the previous claim, 
$f^{\frac12}(J_k^+)=T_{t_k}(J_k^+)=J_k^-$, on $I_k^+$ we have:
$$f^{\frac12}=f^{-1}\circ f^{\frac12}\circ f = (f_0\f_k)^{-1}T_{t_k} (f_0\f_k) = \f_k^{-1} (f_0^{-1} T_{t_k} f_0) \f_k.$$
Now, $f_0$ is the translation by $-2^{-k^2}$ on $\f_k(I_k^+)=I_k^+$, so $(T_{t_k}f_0\f_k)(I_k^+)=J_k^-$. 
On this interval, 
$f_0^{-1}$ is the translation by $2^{-k^2}$, so the equality above reduces to 
$$f^{\frac{1}{2}} = \f_k^{-1} T_{t_{t_k}} \f_k.$$
Finally, since $T_{t_{t_k}} \f_k$ sends $I_k^+$ to $I_k^{-}$ and $\f_k$ is trivial on the latter interval, 
this finally gives the equality 
$f^{\frac{1}{2}} = T_{t_k} \circ \f_k$ on $I_k^+$.

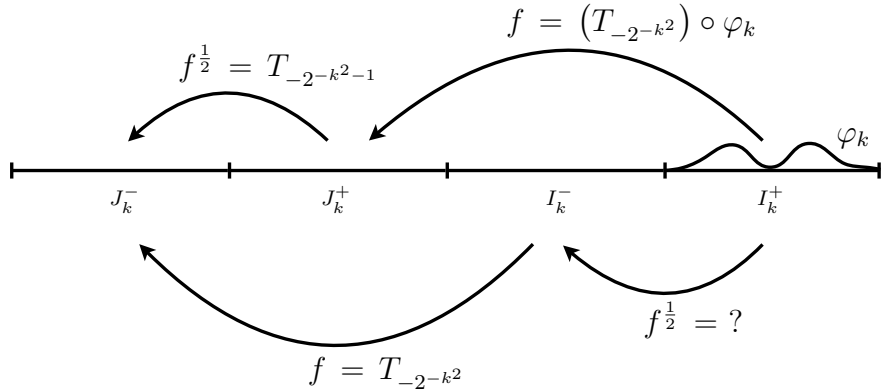
\begin{figure}[h!]
\label{f:f12}
\centering
\includegraphics[width=12cm]{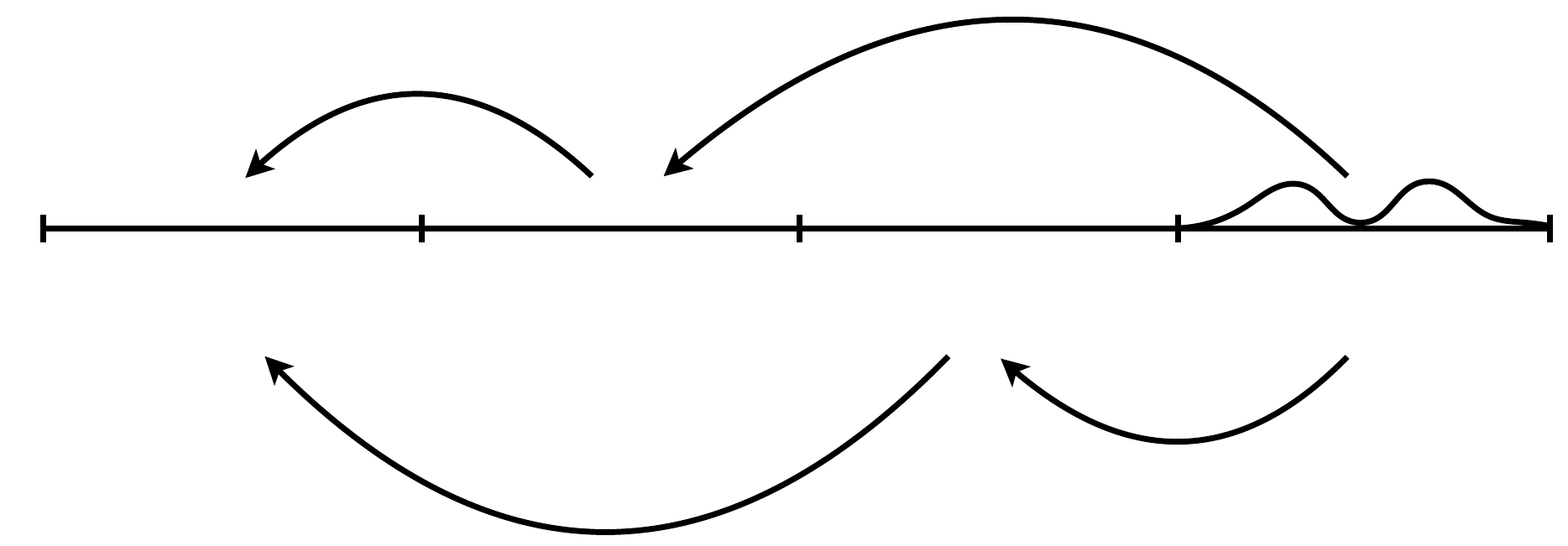}
\put(-215,56){$\scriptstyle{J_k^+}$}
\put(-220,-10){${f\,=\,T_{-2^{-k^2}}}$}
\put(-295,56){$\scriptstyle{J_k^-}$}
\put(-269,105){${f^{\frac12}\,=\,T_{-2^{-k^2-1}}}$}
\put(-130,56){$\scriptstyle{I_k^-}$}
\put(-145,122){${f\,=\,\left(T_{-2^{-k^2}}\right)\circ \f_k}$}
\put(-50,56){$\scriptstyle{I_k^+}$}
\put(-93,8){${f^{\frac12}\,=\,\, ?}$}
\put(-20,80){${\f_k}$}
\caption{Checking that $f^{\frac12}= T_{-2^{-k^2-1}}\circ\f_k$ on $I_k^+$.}
\end{figure}

\medskip

\noindent{\underline{Claim 4:}} There exists a universal constant $c > 0$ such that, for every $k\ge3$ and  
every fundamental interval $I$ lying in the central plateau zone $C_k$ and sent to $I_k$ by some iterate $f^n$ of $f$, 
one has  
$$\mathrm{var} ( \log Df^{\frac12} |_{I^+}) 
\geq c \, 2^{-k^3 + 2k^2}.$$

\medskip

Indeed, on such an $I$ one has the equality $f^n = f_0^n$. Since $D (f_0^n) \cdot X_0 = X_0 \circ f_0^n$, 
the map $f^n$ restricted to $I$ is a homothety $H$ of ratio 
$\alpha_k := \frac{X_0 \circ f_0^n|_I}{X_0|_I} = \frac{2^{-k^2}}{2^{-k^3}} = 2^{k^3 - k^2}$. Thus, on $I^+$, 
$$f^{\frac12} 
= f^{-n} f^{\frac{1}{2}} f^{n} 
= H^{-1} f^{\frac12} H
= H^{-1} T_{t_k} \f_k H,$$
where the last equality follows from Claim 3.  Therefore, still on $I^+$, abbreviating $T_{t_k}$ by $T$,
\begin{align*}
\log Df^{\frac12} 
&= \log D(H^{-1}T\f_k H)\\
&=\underbrace{(\log DH^{-1})\circ (T\f_k H)}_{-\log\alpha_k}
+\underbrace{(\log DT)\circ (\f_k H)}_{0}+(\log D\f_k)\circ H
+\underbrace{\log DH}_{\log\alpha_k}\\
&=(\log D\f_k)\circ H.
\end{align*}
As a consequence,
\begin{eqnarray*}
\mathrm{var} ( \log Df^{\frac12} |_{I^+}) 
&=& \var( \log D\f_k |_{I_k^+}) 
\,\,\, = \,\,\,\| D \log D \f_k  \|_{L^1(I_k^+)} \\
&=& \|\tfrac{D^2\delta_k}{1+D\delta_k}\|_{L^1(I_k^+)} 
\,\,\, \sim \,\,\, \| D^2 \delta_k \|_{L^1 (I_k^+)}
\,\,\, = \,\,\, C \,2^{-k^3 + 2k^2},
\end{eqnarray*}
with $C:=\|D^2\delta\|_{L^1([0,1])}$. 

\medskip

\noindent \underline{Claim 5:} There exists a constant $c' > 0$ such that, 
for every $k \geq 3$,  
$$\var (\log Df^{\frac12} |_{C_k} )
\geq c' \, 2^{2k^2-k}.$$

\medskip

Indeed, since $C_k$ has width $\frac{2^{-k-1}}5$ and $X_0$ is constant equal to $-2^{-k^3}$ there, 
a simple computation shows that $C_k$ contains at least $\frac{2^{-k-1}}{5\cdot 2^{-k^3}}-1$ consecutive points 
of the orbit $(a_j)$ of $f_0$, which is also an orbit of $f$. This implies that the number of intervals $I$ satisfying the 
condition of Claim 4 is at least $\frac{2^{k^3-k-1}}{5}-2$, which is bigger than $c''\,2^{k^3-k}$ for some $c''$ 
independent of $k$. Therefore, by the same claim, 
$$
\mathrm{var} ( \log Df^{\frac12} |_{C_k}) 
\geq  (c''\,2^{k^3-k})(c\,  2^{-k^3 + 2 k^2})
=c'\, 2^{2 k^2 - k}
$$
with $c':=c\,c''$.
\end{proof}

\begin{rem} 
Most probably, in the flow constructed above, only the integer-time maps are of class $C^{1+bv}$. 
Compare \cite{eynard-thesis}, where it was shown that, in Sergeraert's original example, only the integer-time maps 
are of class $C^2$. 
The techniques of \cite{Ey11,Ey19}, which combine Sergeraert's construction to Anosov-Katok's method of 
deformation by conjugation, would probably allow to build a contracting flow in which rationally independent times 
yield smooth maps while other times do not yield $C^{1+bv}$ diffeomorphisms. 
\end{rem}

%%%%%%%%%%%%%%%%%%%%%%%%%%%%%%%%

\section{Appendix 2: asymptotic variation, flowability, and a fundamental inequality revisited} 
\label{a:varflow}

In this appendix, we go a little further in our study of the flow associated to a $C^{1+bv}$ contraction of $[0,1)$ 
given by Proposition \ref{p:Szek2}.
This allows us to extend to the $C^{1+bv}$ setting the relation established in \cite{EN1} between asymptotic variation 
(cf. Section \ref{ss:rat}) and flowability for $C^2$ diffeomorphisms of the closed interval without interior fixed point. 
This relation (cf.~Corollary~\ref{c:varflow} below) was crucially used in Section \ref{ss:rat}.

In order to state it precisely, we need to recall the definition of the \emph{Mather Invariant} of such a 
diffeomorphism~$f$. We will denote by $\Diff_+^{r,\Delta}([0,1])$ 
the space of orientation preserving $C^r$ diffeomorphisms of $[0,1]$ without interior fixed points, 
where $r$ can be some positive integer, but also $1+bv$ or $1+ac$. 
Given $f\in \Diff_+^{2,\Delta}([0,1])$, Szekeres' theorem and Kopell's Lemma imply that 
$f\res{[0,1)}$ and $f\res{(0,1]}$ each embed into a unique $C^1$ flow (cf. Remark \ref{r:folk}), and we denote them 
by $(f^t)$ and $(f_t)$ respectively. Generically, these flows do not coincide, and the Mather invariant of $f$ 
measures this lack of coincidence (see \cite[Chap. V]{Yo} for a detailed discussion of all of this).

\medskip
\begin{defprop}[Mather Invariant]
In the above context, given $a,b$ in $(0,1)$, the maps $\psi_0:t\mapsto f^t(a)$ and $\psi_1:t\mapsto f_t(b)$ define 
$C^1$ diffeomorphisms from $\R$ to $(0,1)$ that conjugate $f=f^1=f_1$ to the translation by $1$. 
As a result, $\psi_1^{-1}\psi_0$ is a $C^1$ diffeomorphism of $\R$ commuting with this translation, 
hence it induces a circle diffeomorphism. 
Changing $a$ or $b$ boils down to pre- or post-composing by a rotation, and the \emph{Mather Invariant} $M_f$ 
of $f$ is precisely the class of circle diffeomorphisms thus obtained. 
\end{defprop}

We will say that the Mather Invariant of $f$ is \emph{trivial} if it coincides with the class of rotations. 
This corresponds to the case where the flows $(f^t)$ and $(f_t)$ coincide, i.e. where $f$ embeds into a 
$C^1$ flow on the whole closed interval. 
We will from now on also denote by $M_f$ the diffeomorphism $\psi_1^{-1}\psi_0$ obtained by taking $a=b=\frac12$ 
in the definition above. We will use the same notation for the circle diffeomorphism it induces, 
and we will abusively refer to it as \emph{the Mather Invariant of~$f$}. 

Note that, despite the Mather Invariant is not a genuine diffeomorphism, the total variation of the logarithm of its 
derivative is well defined, as pre- or post-composition by a rotation doesn't change this value. 
A crucial relation between  $\var(\log DM_f)$ and the asymptotic variation $V_\infty(f)$ was 
established in \cite{EN1} for every $f\in\Diff_+^{2,\Delta}([0,1])$. Below we establish a $C^{1+bv}$ version of this 
thanks to Proposition \ref{p:Szek2}, which allows extending the definition of $M_f$ verbatim to $C^{1+bv}$ 
diffeomorphisms, and to the complementary statement given by Proposition \ref{p:Szek3} below. 

\medskip

\begin{thm} 
\label{thm-general}
For every $f \in \Diff^{1+bv,\Delta}_+ ([0,1])$, the diffeomorphism $M_f$ is of class $C^{1+bv}$, and
\begin{equation}\label{e:thm}
\big| \var (\log DM_f) - \dist (f) \big| \leq | \log Df(0) | + | \log  Df(1 ) | .
\end{equation}
\end{thm}

\medskip

What we actually used in the body of this article were the following corollaries:

\medskip

\begin{cor} \label{c:varflow}
For every $f \in \Diff^{1 + bv,\Delta}_+ ([0,1])$, the equality $V_\infty(f)=0$ holds 
if and only if $f$ embeds into a $C^1$ flow and $Df(0)=Df(1)=0$.
\end{cor}

\begin{cor} \label{c:varflow2}
For every $f \in \Diff^{1 + bv}_+ ([0,1])$, the equality $V_\infty(f)=0$ holds 
if and only if $f$ embeds into a $C^1$ flow and has no hyperbolic fixed point.
\end{cor}

\medskip

\begin{proof}[Proof of Corollary \ref{c:varflow}] 
Recall that $f \in \Diff^{1 + bv}_+ ([0,1])$ embeds into a $C^1$ flow if and only if it has 
a trivial Mather invariant. If  this is the case and, moreover, $Df(0)=Df(1)=0$, then inequality~(\ref{e:thm}) 
directly implies that $V_\infty(f)=0$.

For the reverse implication, we first claim that, for all $f\in \Diff^{1 + bv,\Delta}_+ ([0,1])$, 
\begin{equation}\label{e:indeed}
\dist (f)  \geq | \log  Df(0) | + | \log  Df(1) |.
\end{equation}
Indeed, for every $n \geq 1$, 
$$\var (\log Df^n) 
\geq | \log Df^n (1) - \log Df^n (0) | 
= n \big[ | \log Df(1) | + | \log Df (0) | \big].$$
Dividing by $n$ both sides of this inequality and letting $n \to \infty$ yields the announced inequality. Now, if $V_\infty(f)=0$, 
then inequality (\ref{e:indeed}) implies that $Df (0) = Df (1) = 1$, and inequality (\ref{e:thm}) then gives $|\var(\log DM_f)|\le 0$. Thus, 
$\log DM_f$ is constant, and since $M_f$ is a circle diffeomorphism, this constant has to be $0$, so the Mather Invariant of $f$ is trivial.
\end{proof}

\begin{proof}[Proof of Corollary \ref{c:varflow2}] Let us denote by $\cI$ the set of connected components of $[0,1]\setminus\Fix(f)$. 
According to \cite[Lemma 1.1]{EN1},
$$\dist (f) = \sum_{I \in \mathcal{I}} \dist (f\res{\bar I}).$$
It follows from this equality and Corollary \ref{c:varflow} that $\dist(f)=0$ holds if and only if $f$ has no hyperbolic fixed point and 
$f\res{\bar I}$ embeds into a $C^1$ flow for every $I\in \cI$. It remains to show that
all these local flows pasted together give a $C^1$ flow on the whole interval~$[0,1]$. Clearly, they glue up nicely near isolated fixed 
points of $f$, for the derivative of $f$ at a fixed point determines the derivatives of the elements in a flow containing it; 
see the proof of Proposition \ref{p:Szek2}. The case of non-isolated fixed points can be treated by using the estimate 
of the same proposition. Details are left to the reader.
\end{proof}

Before passing to the proof of Theorem \ref{thm-general}, we show the following complement of Proposition~\ref{p:Szek2}, 
which will be fundamental for our argument.

\medskip

\begin{prop}
\label{p:Szek3}
With the notations of Proposition \ref{p:Szek2}, for every $c > 0$, the function $\log \X_n$ converges to $\log \X$ in the $C^{bv}$ topology 
on the fundamental interval $[f(c),c]$, and
$$\left|\var(\log \X;[f(c),c])- |\log Df(0)|\right|\le \var(\log Df ; [0,c]).$$
\end{prop}

\begin{proof} 
By Proposition \ref{p:Szek2}, $X_n$ converges uniformly towards $X$ on $[f(c),c]$. Moreover, since $\X$ is bounded away from zero on 
this interval, $\log \X_n$ uniformly converges towards $\log \X$ therein. Now, for all integers $n > m$, by the definition of $X_n$ and $X_m$, 
\begin{align}\label{eq:est-n}
\var \big( \log \X_n - \log \X_m ; [f(c),c] \big) \le \var \big(&\log (\Delta\circ f^{n})-\log (\Delta\circ f^{m}) ; [f(c),c] \big)\\&+ 
\var \big( \log Df^{n}-\log Df^{m} ; [f(c),c] \big).\notag
\end{align}
Concerning the last term, 
\begin{eqnarray*}
\var(\log Df^{n}-\log Df^{m} ; [f(c),c]) 
&=& \var \Big( \sum_{i=m}^{n-1}\log (Df\circ f^{i}); [f(c),c] \Big)\\
&\le& \sum_{i=m}^{n-1} \var( \log (Df\circ f^{i}); [f(c),c])\\
&=& \sum_{i=m}^{n-1} \var( \log Df; [f^{i+1}(c),f^{i}(c)])\\
&\le& \var( \log Df; [f^{n}(c),f^{m}(c)])\xrightarrow[m \to +\infty]{}0.
\end{eqnarray*}
Concerning the previous term in (\ref{eq:est-n}), 
\begin{align*}
\var(\log (\Delta\circ f^{n})-\log (\Delta\circ f^{m}) ; [f(c),c])
&=\| D\log (\Delta\circ f^{n})-D\log (\Delta\circ f^{m})\|_{L^1([f(c),c])}\\
&=c_0(f) \left\| \tfrac{D\Delta\circ f^{n}}{\X_n}-\tfrac{D\Delta\circ f^{m}}{\X_m}\right\|_{L^1([f(c),c])}\xrightarrow[m \to+\infty]{}0,
\end{align*}
where the last convergence follows from that 
$\tfrac{D\Delta\circ f^{n}}{\X_n}$ uniformly converges  towards $\tfrac{Df(0)-1}{\X}$ on $[f(c),c]$.
By the completeness of the space of $C^{bv}$ functions on $[f(c),c]$, 
we obtain that \, $\log \X$ \, belongs to this space. Observe moreover that 
\begin{align*}
\var \big( \log (\Delta\circ f^{n}) ; [f(c),c] \big) 
= c_0(f) \int_{f(c)}^c\left|\tfrac{D\Delta\circ f^{n}}{\X_n}\right| \xrightarrow[n \to+\infty]{}c_0(f)\int_{f(c)}^c\left|\tfrac{D\Delta(0)}{\X}\right|=|\log Df(0)|,
\end{align*}
and that (because of the previous estimate with $m=0$)
\begin{align*}
\var(\log Df^{n}; [f(c),c])\le  \var( \log Df; [f^{n}(c),c]).
\end{align*}
Therefore, letting $n$ go to infinity in 
$$\var(\log \X_n ;[f(c),c])\le \var( \log (\Delta\circ f^{n}) ; [f(c),c])+ 
\var(\log Df^{n}; [f(c),c])$$
and 
$$\var (\log (\Delta\circ f^{n}) ; [f(c),c])\le\var(\log \X_n;[f(c),c])+ 
\var(\log Df^{n}; [f(c),c])$$
(which both follow from the definition of $\X_n$),
we get
\begin{equation*}\label{est-one-direction}
\var(\log \X;[f(c),c])\le |\log Df(0)|+\var(\log Df ; [0,c])
\end{equation*}
and
\begin{equation*}\label{est-second-direction}
| \log Df(0)|\le \var(\log \X;[f(c),c])+\var(\log Df ; [0,c]),
\end{equation*}
which yield the desired estimate.
\end{proof}

\medskip

We are finally in position to prove Proposition \ref{thm-general}.

\begin{proof}[Proof of Theorem \ref{thm-general}] Assume that $f (x) > x$ 
for $x \in (0,1)$ to fix ideas (otherwise, just observe that $\dist(f)=\dist(f^{-1})$ and that, by definition, 
the Mather invariant of $f^{-1}$ equals that of $f$ up to a reflexion.)  Denote by $X$ and $Y$ the generating vector fields 
of the flows $(f^t)$ and $(f_t)$ associated to $f\res{[0,1)}$ and $f\res{(0,1]}$ respectively by Proposition \ref{p:Szek2}. 
Since $f$ is the time-1 map of the flow of both 
$\X$ and $\Y$, the maps $\psiX:t\mapsto f^t(\frac12)$ and $\psiY: t\mapsto f_t(\frac12)$ satisfy $\psi_1 \circ T = f \circ \psi_0$ for $T := T_1$, 
the translation by $1$. Therefore, for each positive $m,n$ we have, letting $k:= m+n$: 
\begin{equation}\label{ren}
M_f = T_{-m} \circ ( \psiY )^{-1} \circ f^k \circ \psiX \circ T_{-n}.
\end{equation}
This gives
\begin{eqnarray*}
DM_f (t) 
&=& \frac{D \psiX (t-n)}{D \psiY \big( (\psiY)^{-1} f^k \, \psiX (t-n) \big) } \cdot Df^k (\psiX (t-n)) \\
&=& \frac{\X (\psiX (t-n)) }{\Y (f^k\psi_0(t-n))} \cdot Df^k (\psiX (t-n)).
\end{eqnarray*}
This easily implies that 
$$\left| \var ( \log DM_f  ) - \var (\log Df^k; [f^{-n}(a), f^{-n+1}(a)])\right|$$
is bounded from above by 
$$\var \big( \log \X ; [f^{-n}(a), f^{-n+1}(a)] \big) + \var \big( \log \Y ; [f^m(a),f^{m+1}(a)] \big).$$
By Proposition \ref{p:Szek3}, the latter expression is smaller than or equal to
$$\big| \log Df(0) \big| + \big| \log Df (1) \big| + \mathrm{var} (\log Df; [f^{-n}(a), f^{-n+1}(a)]) + \mathrm{var} (\log Df ; [f^m (a), f^{m+1}(a)]).$$
Letting $m=n = N \to \infty$, the last two terms above converge to $0$, and Proposition 5.1 of \cite{EN1} yields 
\begin{equation}\label{ren-dist}
\var (\log Df^k; [f^{-n}(a), f^{-n+1}(a)]) = \var (\log Df^{2N} ; [f^{-N}(a), f^{-N+1}(a)] )\to \dist (f).
\end{equation}
Putting everything together, we finally obtain 
$$\big| \var ( \log  DM_f ) - \dist (f) \big| \leq \big| \log Df(0) \big| + \big| \log Df (1) \big|,$$
as announced. 
\end{proof}

%%%%%%%%%%%%%%%%%%%%%%%%%%%%%%%%%%%%%%%%%%%%%%%%%%%%%%%%%%%%%%%%%%

\section{Appendix 3: A final on drift of cocycles in Banach spaces}
\label{a:Banach}

The procedure employed in Section \ref{ss:rat} is strongly inspired from \cite[Lemma 2.1]{Na23}, which deals with isometric actions on Banach 
spaces and cocycles with zero drift. Indeed, this framework is perfectly suited for $C^{1+bv}$ actions of $\mathbb{Z}^2$ on one-dimensional 
manifolds for which the generators have vanishing asymptotic variation. Here we extend the discussion to cocycles with nonzero drift and 
relate this to Section \ref{ss:rat}, notably to the ``miraculous'' conjugating sequence (\ref{eq:def-g_n}).
 
Let $U$ be a linear isometric action of a group $\Gamma$ on a Banach space $\mathbb{B}$. A {\em cocycle} for $U$ is a map 
$c \!: \Gamma \to \mathbb{B}$ that, for all $g_1,g_2 $ in $\Gamma$, satisfies the relation
$$c (g_1 g_2) = c(g_2) + U(g_2) (c(g_1)).$$ 
For each $f \in \Gamma$, we define {\em the drift of $c$ at $f$ as}
$$\mathrm{drift}_c (f) := \lim_{n \to \infty} \frac{\| c(f^n) \|_{\mathbb{B}}}{n}.$$ 
The limit above exists because the sequence of norms $\| c(f^n) \|_{\mathbb{B}}$ is subadditive:
$$\| c(f^{m+n}) \|_{\mathbb{B}} = \| c(f^n) + U(f^n) (c(f^m)) \|_{\mathbb{B}} \leq \|c(f^n)\|_{\mathbb{B}} 
+ \|U(f^n)(c(f^m))\|_{\mathbb{B}} = \|c(f^n)\|_{\mathbb{B}} + \|c(f^m)\|_{\mathbb{B}}.$$

The next lemma should be compared to \cite{CTV07}, and is suitable for applications in wide contexts. 

\begin{lem} \label{lem-Banach}
Let $U$ be a linear isometric action of a finitely generated Abelian group $\Gamma$ on 
a Banach space $\mathbb{B}$, and let $c \!: \Gamma \to \mathbb{B}$ be a cocycle. Then there exists a 
sequence of vectors $\psi_n \in \mathbb{B}$ such that, for all $f \in \Gamma$, the coboundary defect
$$\big\| c(f) - \big( \psi_n - U(f) (\psi_n) \big) \big\|_{\mathbb{B}}$$ 
converges to $\mathrm{drift}_c (f)$ as $n$ goes to infinity. Actually, 
\begin{equation}\label{eq-drift-cob}
\mathrm{drift}_c (f) = \inf_{\psi \in \mathbb{B} }\big\| c(f) - \big( \psi - U(f) (\psi) \big) \big\|_{\mathbb{B}}.
\end{equation}

\end{lem}

\begin{proof} We number the elements of $\Gamma$ as $f_1,f_2, \ldots$ 
Let us denote 
$$B (n) := \{ f_{1}^{m_1} f_2^{m_2} \cdots f_n^{m_{n}}: \,\, 0 \leq m_i < n \}$$
For each $n \geq 1$, define
\begin{equation}
\psi_n := \frac{1}{n^n} \sum_{g \in B(n)} c(g).
\label{int-zero}
\end{equation}
Each $f \in \Gamma$ equals $f_i$ for a certain index $i$. Then, for each $n \geq i$, 
\begin{eqnarray*}
U(f) (\psi_n) \!\!\!\!\!\!\!
&=& \!\!\!\!\!\!\! \frac{1}{n^n} \sum_{g \in B(n)} U(f) ( c(g) )
\quad = \quad
 \frac{1}{n^n} \sum_{g \in B(n)} [c(g f) - c(f)] \qquad \hspace{1cm} \\
 &\qquad \qquad = \,\,\, & - c(f) + \frac{1}{ n^n } \sum_{g \in B(n)} c(g f) 
\quad = \quad  - c(f) + \frac{1}{ n^n } \sum_{g \in B(n)} c(f g).
\end{eqnarray*}
Therefore, 
$$\Big\| c ( f ) - ( \psi_n - U(f) (\psi_n) ) \Big\|_{\mathbb{B}} \leq \frac{1}{n^n} \Bigl\| \sum_{g \in B(n)} [ c (f g) - c(g) ] \Bigr\|_{\mathbb{B}},$$
and the last expression equals
$$\frac{1}{n^n} \left\| \sum_{\substack{0 \leq m_j < n \\ j \neq i}} \big[ c (f^n f_1^{m_1} \cdots f_{i-1}^{m_{i-1}} f_{i+1}^{m_{i+1}} \cdots f_n^{m_n}) 
- c (f_1^{m_1} \cdots f_{i-1}^{m_{i-1}} f_{i+1}^{m_{i+1}} \cdots f_n^{m_n}) \big] \right\|_{\mathbb{B}} \! \! .$$
By the cocycle relation, this reduces to
$$\frac{1}{n^n} \left\| \sum_{\substack{0 \leq m_j < n \\ j \neq i}} 
U (f_1^{m_1} \cdots f_{i-1}^{m_{i-1}} f_{i+1}^{m_{i+1}} \cdots f_n^{m_n}) (c (f^n) ) \right\|_{\mathbb{B}},$$
which, by the triangular inequality, is smaller than or equal to 
$$\frac{1}{n^n} \sum_{\substack{0 \leq m_j < n \\ j \neq i}} 
\left\| U (f_1^{m_1} \cdots f_{i-1}^{m_{i-1}} f_{i+1}^{m_{i+1}} \cdots f_n^{m_n}) (c (f^n) ) \right\|_{\mathbb{B}} 
= \frac{1}{n^n} \sum_{\substack{0 \leq m_j < n \\ j \neq i}} \| c(f^n) \|_{\mathbb{B}}.$$
The last expression is equal to
$$\frac{1}{n^n} \, n^{n-1} \| c (f^n) \|_{\mathbb{B}} = \frac{\| c (f^n) \|_{\mathbb{B}}}{n}.$$
By definition, this converges to $\mathrm{drift}_c (f)$ as $n \to \infty$. Therefore, the lim sup of
$$\| c ( f ) - ( \psi_n - U( f ) (\psi_n) ) \|_{\mathbb{B}} $$
is at most $\mathrm{drift}_{c} (f)$. 

To complete the proof of (\ref{eq-drift-cob}), for $f \in \Gamma$ and $\psi \in \mathbb{B}$, we let 
\, $C := \| c(f) - (\psi - U(f) (\psi)) \|_{\mathbb{B}}$. \, Then, for each $i \geq 1$, we have
$$C = \| U (f^i) (c(f)) - (U(f^i)(\psi) - U(f^{i+1})(\psi)) \|_{\mathbb{B}}.$$
The triangular inequality and the cocycle relation (which in particular implies that $c (id) = 0$) then yield
\begin{eqnarray*}
n \, C 
&=& \sum_{i=0}^{n-1}  \big\| U (f^i) (c(f)) - (U(f^i)(\psi) - U(f^{i+1})(\psi)) \big\|_{\mathbb{B}} \\
&\geq& \left\|  \sum_{i=0}^{n-1} U (f^i) (c(f)) -  \big( U(f^i)(\psi) - U(f^{i+1})(\psi) \big) \right\|_{\mathbb{B}} \\
&=& \left\| \sum_{i=0}^{n-1} [c(f^{i+1})-c(f^i)] - \big( \psi - U(f^n) (\psi) \big) \right\|_{\mathbb{B}} \\
&=& \big\| c(f^n) - ( \psi - U(f^n)( \psi )) \big\|_{\mathbb{B}} \\
&\geq& \|c(f^n)\|_{\mathbb{B}} - \|\psi\|_{\mathbb{B}} - \|U(f^n) (\psi) \big) \|_{\mathbb{B}}.
\end{eqnarray*}
Therefore,
$$C \geq \frac{\|c(f^n)\|_{\mathbb{B}}}{n}  - 2 \frac{\|\psi\|_{\mathbb{B}}}{n}.$$
Passing to the limit in $n$ this yields $C \geq \mathrm{drift}_c (f).$
\end{proof}

\medskip

If $\Gamma$ is an Abelian group of $C^{1+\mathrm{ac}}$ diffeomorphisms of a compact one-manifold 
$M$, then $U \!: (f,\varphi) \mapsto ( \varphi \circ f) \cdot Df$ is an isometric action on $\mathbb{B} = L^1(M)$, 
and the affine derivative $c(f) := \frac{D^2 f}{D f}$ is a cocycle for this action. The drift of this cocycle at $f$ is 
nothing but the asymptotic variation of $f$. There is also a close similarity between (\ref{eq:def-g_n}) and 
(\ref{int-zero}). No miracle is happening here.

%%%%%%%%%%%%%%%%%%%%%%%%%%%%%%%%%%%%%%%%%%%%%%%%%%%%%%%%%%%%%%%%%%
\section{Appendix 4: 
On the topology of the space of bounded variation functions, 
in collaboration with Th\'eo Virot}
\label{a:theo}
 
In this last appendix, we give a counterexample to Lemma \ref{l:astuce} for a function $u$ that has bounded variation but is not absolutely 
continuous. As we will see, this is closely related to the fact that $\Diff^{1+bv}_+ ([0,1])$ is not a topological group when endowed with the 
topology induced by the metric $d_{1+bv}$ from Section \ref{notations}. (Specifically, left multiplication is not continuous.) It is worth 
mentioning that, after working out the examples below, we discovered that the results they lead to are not new. Actually, they seem 
to be known to the specialists, though the only reference we found on this is the beautiful recent article \cite{Co20} of Cohen, 
where many related results are given. (The most striking one is that $\Diff^{1+bv}_+([0,1])$ admits no (Hausdorff) Polish 
topological group structure.) However, proofs from \cite{Co20} are not explicit and rest upon very subtle considerations 
from functional analysis. The example we present below is very concrete and has a dynamical flavor. 
 
\medskip
 
Recall that we look for:  
\vspace{0.2cm}
\begin{itemize}

\item a bounded variation function $u: [0,1] \to \mathbb{R}$ (which is not absolutely continuous) and  

\item a sequence of $C^1$ diffeomorphisms $\phi_n \!: [0,1] \to [0,1]$ converging to the identity 
(in $C^1$ topology)  

\end{itemize}
such that 
\begin{equation}\label{e:but}
\mathrm{var} (u \circ \phi_n - u) \not\longrightarrow 0.
\end{equation}
%D'apr\`es le Lemme 3.3 de [EN], pour un tel exemple, la fonction $E$ ne peut pas \^etre absolument continue. 
We will manage to do this with $u$ a devil staircase. However, for technical reasons, we will need to replace 
the classical ternary Cantor set by another one which is more adapted to our needs.  
%Nous chercherons \`a construire l'exemple avec $E$ un escalier de Cantor, mais pas sur le Cantor 
%triadique mais plut\^ot sur un ensemble adapt\'e \`a nos besoins. 
%De cette fa\c{c}on, on obtiendra une  convergence $\phi_n \to id$ m\^eme en topologie $C^{\infty}$.

%%%%%%%%%%%%%%%%%%%%%%%%%%%%%%%
%\section{Une  construction g\'en\'erale}

\vspace{0.35cm}

\noindent\underbar{Step 1:} We start with any closed interval  $I := [a,b] \subset (0,1)$ and we choose 
$I_0 := [a_0,b_0] \subset (0,a)$ and $[a_1,b_1] \subset (b,1)$. Then we proceed by induction: we assume 
that for a given $n \geq 1$ we have chosen $I_{\iota}$ for all $\iota \in \{0,1\}^k$, with $k \leq n$. We fix 
$\iota \in \{ 0 , 1\}^n$, we let $\iota,0$ (resp. $\iota,1$) be the concatenation of $\iota$ with $0$ (resp. $1$), 
and we choose closed intervals $I_{\iota,0}, I_{\iota,1}$ that are disjoint from those previously chosen so that:

\begin{itemize}
\item $I_{\iota,0}$ (resp. $I_{\iota,1}$) is on the left (resp. right) of $I_{\iota}$, 
\item there is no interval of type $I_{\kappa}$, with $\kappa \in \{0,1\}^k$ and $k \leq n$, 
between $I_{\iota}$ and $I_{\iota,0}$, and the same holds true for $I_{\iota}$ and $I_{\iota,1}$. 
\end{itemize}

We  denote $I_{\iota} = [a_{\iota},b_{\iota}]$ and we let $m_{\iota}$ be the midpoint of $I_{\iota}$. 
We suppose that the length of each $I_{\iota}$ is at least one third of the length of the connected component of \, 
$[0,1] \setminus \bigcup_{k \leq n} \bigcup_{\kappa \in \{0,1\}^k} I_{\kappa}$ \, in which it is contained.  If 
we let $I_{\iota}^{\mathrm{o}}$ be the interior of $I_{\iota}$, then 
$$ K := [0,1] \setminus \bigcup_{n \in \mathbb{N}} \bigcup_{\iota \in \{0,1\}^n} I_{\iota}^{\mathrm{o}}$$ 
is a Cantor set. 
%\textcolor{red}{Il ne faut pas prendre les intérieurs des $I$?}

\vspace{0.35cm}

\noindent\underbar{Step 2:} We let $u$ be the devil staircase function associated to the Cantor set $K$ previously constructed.
This function is constant (say, equal to $u_{\iota}$) on each interval $I_{\iota}$. It is recursively defined by:

\begin{itemize}

\item the restriction of $u$ to $I$ is identically equal to $1/2$, 

\item the restriction of $u$ to $I_0$ (resp. $I_1$) is identically equal to $1/4$ \, (resp. $3/4$),

\item more generally, for all $\iota \in \{0,1\}^n$, the restriction of $u$ to $I_{\iota,0}$ (resp. 
$I_{\iota,1}$) is identically equal to $u_{\iota} - 1 / 2^{n+2}$ \, (resp. $u_{\iota} + 1/2^{n+2}$).
\end{itemize}

The function $u$ has bounded variation, since it is monotone (non decreasing) and bounded.\\ 

Now, for every integer $n \geq 1$, we let $\phi_n: [0,1] \to [0,1]$ be a homeomorphism such that:

\begin{itemize} 

\item if $\iota \in \{0,1\}^n$, then  $\phi_n (b_{\iota,0}) = a_{\iota}$, 

\item $\phi_n$ fixes $m_{\iota}$ for all $\iota \in \{ 0, 1 \}^{n} \cup \{ 0, 1 \}^{n+1}$.

\end{itemize}
\begin{figure}[h!]
%\label{f:f12}
\centering
\includegraphics[width=15cm]{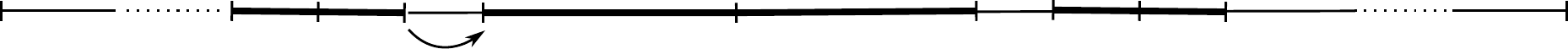}
\put(-429,17){$\scriptstyle{0}$}
\put(-2,17){$\scriptstyle{1}$}
\put(-370,17){$\scriptstyle{a_{\iota,0}}$}
\put(-365,37){$\overbrace{\hspace{1.75cm}}$}
\put(-364,-14){$\scriptstyle{\phi_n=\id}$}
\put(-347,17){$\scriptstyle{m_{\iota,0}}$}
\put(-347,50){$\scriptstyle{I_{\iota,0}}$}
\put(-363,0){$\underbrace{\hspace{0.8cm}}$}
\put(-322,17){$\scriptstyle{b_{\iota,0}}$}
\put(-311,-10){$\scriptstyle{\phi_n}$}
\put(-300,17){$\scriptstyle{a_{\iota}}$}
\put(-295,37){$\overbrace{\hspace{4.7cm}}$}
\put(-229,17){$\scriptstyle{m_{\iota}}$}
\put(-227,0){$\underbrace{\hspace{4.7cm}}$}
\put(-229,50){$\scriptstyle{I_{\iota}}$}
\put(-163,17){$\scriptstyle{b_{\iota}}$}
\put(-170,-14){$\scriptstyle{\phi_n=\id}$}
\put(-144,17){$\scriptstyle{a_{\iota,1}}$}
\put(-140,37){$\overbrace{\hspace{1.70cm}}$}
\put(-124,17){$\scriptstyle{m_{\iota,1}}$}
\put(-124,50){$\scriptstyle{I_{\iota,1}}$}
\put(-100,17){$\scriptstyle{b_{\iota,1}}$}
\caption{Action of $\phi_n$ near $I_\iota$, with $\iota\in\{0,1\}^n$.}
\end{figure}
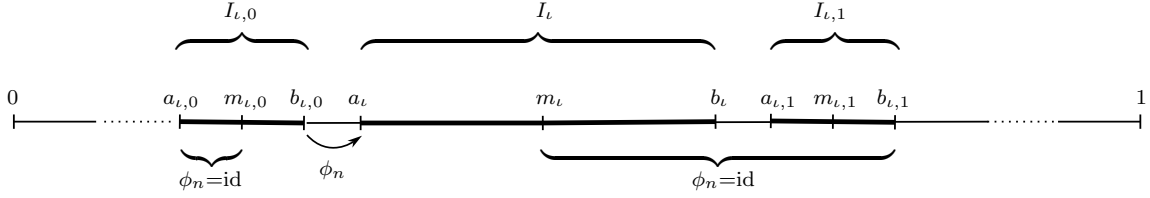

%$$\underbrace{\hspace{3cm}}$$

We claim that these conditions yield (\ref{e:but}). Indeed, for all $n \geq 1$,
\begin{eqnarray*}
\mathrm{var} (u \circ \phi_n - u) 
&\geq& \sum_{\iota \in \{0,1\}^n} \big| (u \circ \phi_n - u) (b_{\iota,0}) - (u \circ \phi_n - u) (m_{\iota,0}) \big| \\
&=&  \sum_{\iota \in \{0,1\}^n} \big| (u(a_{\iota})  - u (b_{\iota,0})) - (u (m_{\iota,0}) - u (m_{\iota,0}) ) \big| \\
&=&  \sum_{\iota \in \{0,1\}^n} \big| u(a_{\iota})  - u (b_{\iota,0}) \big| 
\,\, = \sum_{\iota \in \{0,1\}^n} \frac{1}{2^{n+2}} 
\,\, = \,\, \frac{2^n}{2^{n+2}} 
\,\, = \,\, \frac{1}{4}.
\end{eqnarray*}

%%%%%%%%%%%%%%%%%%%%%%%%%%%
%\section{Quelques conditions pour la diff\'erentiabilit\'e}
 
\vspace{0.35cm}

\noindent\underbar{Step 3:} We now deal with the differentiability of the sequence of maps. To achieve this, we explicitly define $\phi_n$ by letting:
% Remarquons que ces conditions excluent le Cantor triadique standard, mais ne sont pas contradictoires: on donne un exemple d'un tel ensemble dans l'exemple 3. 
 %Cependant, elle permettent de construire par le lemme 2 des diff\'eomorphismes $\phi_n$ de classe $C^1$. Pour cela, il suffit de poser :
 \begin{itemize}
 
 \item for $x \in [m_{\iota,0},m_{\iota}]$, with $\iota \in \{0,1\}^n$,
 \begin{equation}\label{e:interp}
 \phi_n (x) := x + \frac{a_{\iota} - b_{\iota,0}}{(b_{\iota,0}-m_{\iota,0})^2 (m_{\iota} - b_{\iota,0})^2} (x - m_{\iota,0})^2 (x - m_{\iota})^2,
 \end{equation}
 
 \item $\phi_n (x) := x$ otherwise.
 
 \end{itemize}
 
%\noindent We claim that $\phi_n$ thus defined is a $C^1$ diffeomorphism. Indeed, this directly follows from the equalities 
Note that $\phi_n (b_{\iota,0}) = a_{\iota}$, as requested above. Moreover, the definition immediately yields 
$$D \phi_n (x) = 1 +  \frac{4 (a_{\iota} - b_{\iota,0})  (x - m_{\iota,0}) (x - m_{\iota}) (x - \frac{m_{\iota} + m_{\iota,0}}{2})} {(b_{\iota,0}-m_{\iota,0})^2 (m_{\iota} - b_{\iota,0})^2}$$ 
for $x \in [m_{\iota,0},m_{\iota}]$.  In particular, 
$D \phi_n (m_{\iota,0}) = D \phi_n (m_{\iota}) = 1$, which  shows that $\phi_n$ is of class $C^1$. The fact that it is a  
$C^1$ diffeomorphism at least for well-chosen parameters follows from Step 4 below, where it will be shown that 
$\| D \phi_n - 1\|_{\infty}$ can be taken as close to 0 as desired and, in particular, $D \phi_n$ can be taken to be 
everywhere positive.

\vspace{0.35cm}

\noindent\underbar{Step 4:} We next deal with the $C^1$ convergence of $\phi_n$ to the identity. To achieve this, 
we fix a sequence of positive numbers $\varepsilon_n < 1$ converging to $0$. We suppose that the intervals $I_{\iota}$ 
are chosen so that 
 \begin{equation}\label{e:para-c-1}
 \frac{|a_{\iota} - b_{\iota,0}|}{|a_{\iota} - m_{\iota,0}|} \leq \varepsilon_n \quad \mbox{ for all } \quad \iota \in \{0,1\}^n.
 \end{equation}
 and
 \begin{equation}\label{e:bound}
     \frac{1}{M}\leq\frac{m_\iota-a_\iota}{b_{\iota,0}-m_{\iota,0}}\leq M
 \end{equation}
for a certain constant $M \geq 2$ (independent of $n$). Note that these conditions are not satisfied by the standard ternary Cantor set. 
However, we will see in Step 5 that they do not contradict each other. 

The announced convergence is then a direct consequence of the next lemma.

\begin{lem}\label{l:family}
Given real numbers $a<c<d<b$ such that $\frac{d-c}{d-a}\leq \epsilon $ and $\frac{1}{M}\leq \frac{d-b}{a-c}\leq M$, 
where $\epsilon < 1$ is positive and $M\geq 2$, the map $\Phi$ defined by 
$$\Phi(x) := x+\frac{d-c}{(c-a)^2(b-c)^2}(x-a)^2 (x-b)^2$$
is a $C^1$ diffeomorphism of $[a,b]$ that sends $c$ to $d$ and satisfies   
$$||D\Phi-1||_{\infty} 
\leq \max \Bigl\{16M^2\frac{\epsilon}{1-\epsilon},8M\frac{\epsilon}{1-\epsilon}\Bigl(1+\frac{\epsilon}{1-\epsilon}\Bigr), \frac{16M^3\epsilon}{1-\epsilon}\Bigr\}.$$
%    and $D\phi(a)=D\phi(b)=1$.
\end{lem}    

    \begin{proof}
    We first note that \, 
    $d-c\leq\epsilon(d-a) = \epsilon(d-c)+\epsilon(c-a).$ \, 
    This gives \esp $(d-c)(1-\epsilon)\leq\epsilon(c-a)$, \esp
    hence
    $$\frac{d-c}{c-a}\leq\frac{\epsilon}{1-\epsilon}.$$
    Moreover, 
    $$b-c=\frac{b-d}{2}+\frac{b-d}{2}+d-c 
    \geq  \frac{c-a}{2M}+\frac{b-d}{2}+d-c \geq \frac{b-a}{2M}.$$
    These two inequalities will be used without reference several times below. 
    
We compute
    $$D\Phi(x)
    %=1+\frac{(d-c)}{(c-a)^2(b-c)^2}\Bigl[2(x-a)(b-x)^2-2(x-a)^2(b-x)\Bigr].$$
    =1+\frac{4(d-c)(x-\frac{a+b}{2})(x-a)(b-x)}{(c-a)^2(b-c)^2},$$
    and we denote $E(x) := D\Phi(x)-1$. To estimate the norm of $E$, we distinguish three cases. 
    \begin{itemize}
   
   \item if $x\in[a,c]$, then 
    $$|E(x)|=4\cdot\frac{d-c}{c-a}\cdot\frac{x-a}{c-a}\cdot\frac{b-x}{b-c}\cdot\frac{|x-\frac{a+b}{2}|}{b-c} 
    \leq \frac{d-c}{c-a}\cdot 1\cdot\frac{b-x}{b-a}\cdot\frac{|x-\frac{a+b}{2}|}{b-a}16M^2.$$
    Since the last two terms are $\leq 1$, we obtain in this case 
    $$|E(x)|\leq 16M^2\frac{\epsilon}{1-\epsilon}.$$

    \item If $x\in [c,d]$, then 
    $$|E(x)|=4\cdot\frac{d-c}{c-a}\cdot\frac{x-a}{c-a}\cdot\frac{b-x}{b-c}\cdot\frac{|x-\frac{a+b}{2}|}{b-c}.$$
    %The first terms is  $\frac{\epsilon}{1-\epsilon}$
    The last two terms are bounded from above by 1 and $2M$, respectively. This gives
    $$|E(x)|\leq 8M\frac{\epsilon}{1-\epsilon}\cdot\frac{x-a}{c-a} 
    \leq 8M\frac{\epsilon}{1-\epsilon}\cdot\frac{d-a}{c-a}=8M\frac{\epsilon}{1-\epsilon}\cdot\Bigl(1+\frac{d-c}{c-a}\Bigr) 
    \leq 8M\frac{\epsilon}{1-\epsilon}\Bigl(1+\frac{\epsilon}{1-\epsilon}\Bigr).$$
    
    \item Finally, if $x\in[d,b]$, then 
    $$|E(x)|=4\cdot\frac{d-c}{c-a}\cdot\frac{b-x}{c-a}\cdot\frac{x-a}{b-c}\cdot\frac{|x-\frac{a+b}{2}|}{b-c} 
    \leq 4\cdot\frac{\epsilon}{1-\epsilon}\cdot\frac{b-d}{c-a}\cdot\frac{x-a}{b-a}\cdot\frac{|x-\frac{a+b}{2}|}{b-a}4M^2.$$
    Since the last two terms are $\leq 1$, this gives
    $$|E(x)|\leq 4\cdot\frac{\epsilon}{1-\epsilon}\cdot M\cdot 4M^2=\frac{16M^3\epsilon}{1-\epsilon}.$$
    \end{itemize}
    This closes the proof.
    \end{proof}
 
\begin{rem} The hypothesis of the preceding lemma are somewhat necessary to obtain maps that are close to the identity 
in $C^1$ topology. This follows from he following general fact (compare \cite[Lemma 2.7]{Na08}): If $\phi: [a,b] \to [a,b]$ is a 
$C^1$ diffeomorphism then, for all $c \in (a,b)$, 
$$\frac{|\phi(c)-c|}{b-a} \leq \| D \phi - 1\|_{\infty}.$$
Indeed, choosing $c' \in (a,b)$ such that $D\phi (c') = | \phi (c) - \phi (a) | / |c - a|$ and using that $\phi (a) = a$, 
we obtain 
$$ \frac{| \phi (c) - c|}{ b - a } 
< \frac{| \phi (c) - c|}{ c - a }
= \left| \frac{\phi (c) - \phi (a)}{c - a} - 1 \right| 
= | D \phi (c') - 1| \leq \| D \phi - 1 \|_{\infty}.$$
\end{rem}

\vspace{0.35cm}

\noindent\underbar{Step 5:} Finally, to show that the conditions (\ref{e:para-c-1}) and (\ref{e:bound}) are compatible, 
we build our Cantor set recursively as follows. Suppose that, for a given $n \geq 1$, the intervals $I_{\iota}$ have been 
constructed for all $\iota \in \{0,1\}^k$ and all $k \leq n$. Fix $\iota \in \{ 0 , 1 \}^n$, and denote $L_{\iota}$ (resp. $R_{\iota}$) 
the connected component to the left (resp. right) of $I_{\iota}$ of the set 
$$[0,1] \setminus \Big( I \,\,\,\, \cup \bigcup_{\iota' \in \bigcup_{k < n} \{0,1\}^k} I_{\iota'} \Big).$$
%$[0,1]$ minus the union of all these intervals and the initial one $I$. 
Next we define:

\begin{itemize} 

\item $I_{\iota,1}$ as being the central third subinterval of $R_{\iota}$, and

\item $I_{\iota,0}$ as being the interval whose left endpoint $a_{\iota,0}$ coincides with the midpoint of $L_{\iota}$ and 
for which (\ref{e:para-c-1}) is an equality. 

\end{itemize}  
 
One easily convinces that, with this tricky choice, (\ref{e:bound}) is also satisfied for an appropriate constant $M$. 
This concludes the construction of our counterexample to Lemma \ref{l:astuce} for a bounded variation, non 
absolutely continuous function $u$. 

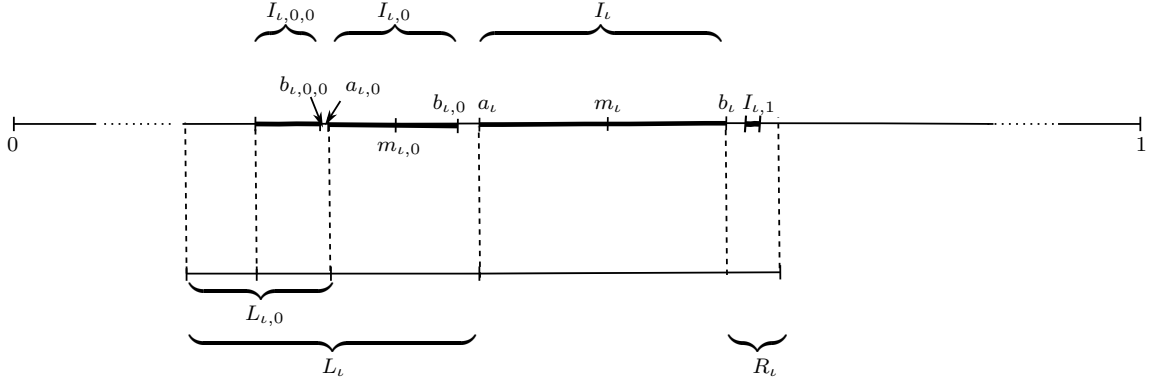
\begin{figure}[h!]
%\label{f:f12}
\centering
\includegraphics[width=15cm]{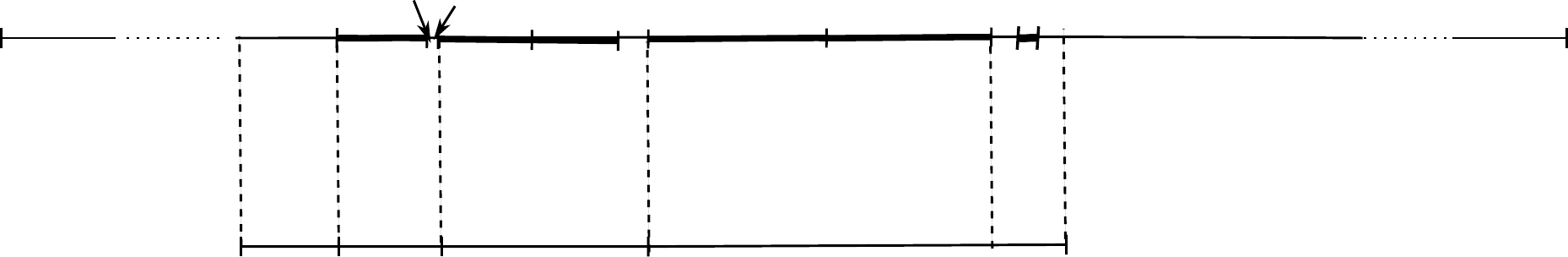}
\put(-429,49){$\scriptstyle{0}$}
\put(-2,49){$\scriptstyle{1}$}
%\put(-370,17){$\scriptstyle{a_{\iota,0}}$}
\put(-335,90){$\overbrace{\hspace{0.8cm}}$}
%\put(-364,-14){$\scriptstyle{\phi_n=\id}$}
\put(-289,49){$\scriptstyle{m_{\iota,0}}$}
\put(-289,100){$\scriptstyle{I_{\iota,0}}$}
\put(-360,0){$\underbrace{\hspace{1.9cm}}$}
\put(-360,-20){$\underbrace{\hspace{3.8cm}}$}
\put(-310,-35){$\scriptstyle{L_{\iota}}$}
\put(-156,-20){$\underbrace{\hspace{0.2cm}}$}
\put(-147,-35){$\scriptstyle{R_{\iota}}$}
\put(-339,-15){$\scriptstyle{L_{\iota,0}}$}
\put(-326,72){$\scriptstyle{b_{\iota,0,0}}$}
\put(-331,100){$\scriptstyle{I_{\iota,0,0}}$}
%\put(-311,-10){$\scriptstyle{\phi_n}$}
\put(-301,72){$\scriptstyle{a_{\iota,0}}$}
\put(-268,65){$\scriptstyle{b_{\iota,0}}$}
\put(-251,65){$\scriptstyle{a_{\iota}}$}
\put(-207,65){$\scriptstyle{m_{\iota}}$}
\put(-250,90){$\overbrace{\hspace{3.2cm}}$}
%\put(-229,17){$\scriptstyle{m_{\iota}}$}
\put(-207,100){$\scriptstyle{I_{\iota}}$}
%\put(-227,0){$\underbrace{\hspace{4.7cm}}$}
%\put(-229,50){$\scriptstyle{I_{\iota}}$}
\put(-160,65){$\scriptstyle{b_{\iota}}$}
\put(-151,65){$\scriptstyle{I_{\iota,1}}$}
%\put(-170,-14){$\scriptstyle{\phi_n=\id}$}
%\put(-144,17){$\scriptstyle{a_{\iota,1}}$}
\put(-305,90){$\overbrace{\hspace{1.6cm}}$}
%\put(-124,17){$\scriptstyle{m_{\iota,1}}$}
%\put(-124,100){$\scriptstyle{I_{\iota,1}}$}
%\put(-100,17){$\scriptstyle{b_{\iota,1}}$}
\caption{Proportions of the intervals $I_\iota$.}
\end{figure}

\bigskip

 To close this appendix, let us show why this is closely related to our claim concerning the topology of $\Diff_+^{1+bv} ([0,1])$. Adding a constant to $u$, we 
 may assume that the integral of $e^u$ equals 1. Since this is a positive and continuous function, it coincides with the derivative of a $C^1$ diffeomorphism 
 $f$ of $[0,1]$, which is actually of class $C^{1+bv}$ since $u$ has bounded variation. Now, for the sequence of $C^1$ diffeomorphisms $\phi_n$ constructed 
 above, we have
 \begin{eqnarray*}
 d_{1+bv} (f \phi_n, f) 
&=& \| f \phi_n - f \|_{\infty} + \var ( \log D (f\phi_n) -  \log D f ) \\
& =& \| f \phi_n - f \|_{\infty} + \var ( (\log D f) \circ \phi_n + \log D \phi_n - \log D f ).
 \end{eqnarray*}
 Assume for a while that $\phi_n$ converges to the identity not only in $C^{1}$ topology but also in the $d_{1+bv}$ metric. This means that 
 $\var (\log D \phi_n)$ goes to 0 as $n$ goes to infinity. Since $\| f \phi_n - f\|_{\infty}$ also goes to 0, the last  sum above asymptotically behaves as 
 $$\var ( (\log Df) \circ \phi_n - \log Df) = \var (u \circ \phi_n - u),$$
 which does not converge to 0. As a consequence, $f \phi_n$ does not converge to $f$ in the $d_{1+bv}$ metric, which shows that left multiplication 
 is not continuous in $\Diff_+^{1+bv} ([0,1])$.
 
 To conclude the discussion, let us explain how to ensure the convergence  $\var (\log D \phi_n) \to 0$. This is achieved by slightly 
 improving the previous construction.  Namely, an easy computation shows that the map $\Phi$ from Lemma \ref{l:family} satisfies 
 $$\mathrm{var} (D\Phi; [a,b]) \leq \frac{16 (d-c)(b-a)^3}{(c-a)^2(b-c)^2}.$$
As a consequence, using (\ref{e:para-c-1}) and  (\ref{e:bound}), this yields
$$\mathrm{var} (D\phi_n; [m_{\iota,0},m_{\iota}]) \leq M' \varepsilon_n$$ 
for a certain universal constant $M'$ and all $\iota \in \{0,1\}^n$, with $n\geq 1$. 
In particular, for $\varepsilon_n := 1 / 3^n$, 
$$\var (D \phi_n) = \sum_{\iota \in \{0,1\}^n} \mathrm{var} (D\phi_n; [m_{\iota,0},m_{\iota}]) \leq M'  \left( \frac{2}{3} \right)^n,$$  
which easily implies that $\var (\log D \phi_n) \to 0$. 

\medskip 

\begin{rem} One can even modify the preceding construction of the maps $\phi_n$ so that they become $C^{\infty}$ 
diffeomorphisms converging to the identity in $C^{\infty}$ topology though they still satisfy (\ref{e:but}).
\end{rem}

\medskip

 \begin{rem} 
 In \cite{Co20} it is proved that, for {\em every}
  $f \in \Diff^{1+bv}_+([0,1]) \setminus \Diff^{1+ac}_+([0,1])$, there exists a sequence of $C^{1+bv}$ diffeomorphisms 
 $\phi_n$ converging to the identity in the $d_{1+bv}$ metric and such that  $f \phi_n$ does not converge to $f$ in the $d_{1+bv}$ metric.  
 \end{rem}
 
 \vspace{0.5cm}
 
 \noindent{\bf Acknowledgments.}  Both authors were funded by the ANR project GROMEOV and Fondecyt grant 1220032. 
 H.~Eynard-Bontemps was also supported by Grenoble INP-UGA through the IRGA project ADMIN. During the elaboration 
 of this work, she was detached at CNRS and was thus partially funded by the IRL Centro de Modelamiento Matemático (CMM), 
 FB210005, BASAL funds for centers of excellence from ANID-Chile.  At the final stage, A. Navas was also supported by the 
 MathsAmsud Project 210020 Dynamical Group Theory. The authors have also benefitted from the welcoming environment 
 of the Math. and Computer Science Department of USACH and the Math. Department of the University of Chile in Santiago, 
 as well as the Institut Fourier in Grenoble, the Mittag Leffler Institut in Stockholm and the ICMAT Institute in Madrid. 
 
 %%%%%%%%%%%%%%%%%%%%%%%%%%%%%%%%%%%%%%%%%%%%%%%%%%%%%%%%%%%%%%%%%%%%%
\bigskip

\begin{footnotesize}

\vspace{0.42cm}

\noindent {\bf H\'el\`ene Eynard-Bontemps} \hfill{\bf Andr\'es Navas}

\noindent Institut Fourier \hfill{Dpto de Matem\'atica y C.C.}

\noindent  Universit\'e Grenoble Alpes \hfill{ Universidad de Santiago de Chile}

\noindent 100 rue des Math\'ematiques \hfill{Alameda Bernardo O'Higgins 3363}

\noindent 38610 Gi\`eres, France \hfill{Estaci\'on Central, Santiago, Chile} 

\noindent helene.eynard-bontemps@univ-grenoble-alpes.fr \hfill{andres.navas@usach.cl}

\end{footnotesize}
\bigskip

\end{document}